\documentclass{amsart}
\usepackage{amsfonts}
\usepackage[active]{srcltx}
\usepackage{amsmath}
\usepackage{amssymb}
\usepackage{amsthm}
\usepackage{mathtools}
\usepackage{shuffle}
\usepackage{bm} %for bold math symbol

\usepackage{enumitem}
\usepackage{wasysym}
\usepackage{mathscinet}
\usepackage{verbatim}
\usepackage{graphicx}
\usepackage[
bookmarks=true,         %generates bookmarks for all entries in the "Table of Contents"
bookmarksnumbered=true, %includes chapter-/header-/section-/subsection-/... -
colorlinks=true, pdfstartview=FitV, linkcolor=blue, citecolor=blue,
urlcolor=blue]{hyperref}
\usepackage{microtype}
\def\replacecolorred{{}}

\theoremstyle{plain}
\newtheorem{theorem}{Theorem}[section]
\newtheorem{prop}[theorem]{Proposition}

\newtheorem{corollary}[theorem]{Corollary}

\newtheorem{lemma}[theorem]{Lemma}
\newtheorem{problem}[theorem]{Problem}
\newtheorem{proposition}[theorem]{Proposition}

\theoremstyle{remark}

\theoremstyle{definition}
\newtheorem{remark}[theorem]{Remark}
\newtheorem{definition}[theorem]{Definition}
\newtheorem{example}[theorem]{Example}

\def\RR{\mathbb{R}}
\def\EE{\mathbb{E}}
\def\NN{\mathbb{N}}

\def\cF{{\mathcal F}}

\def\cH{{\mathcal H}}

\def\be{{\beta}}
\def\de{{\delta}}
\def\la{{\lambda}}
\def\si{{\sigma}}

\def\cP{{\mathcal  P}}
\def\De{{\Delta}}

\def\Om{{\Omega}}
\def\om{{\omega}}
\def\al{{\alpha}}

\def\be{{\beta}}

\def\ga{{\gamma}}
\def\de{{\delta}}
\def\De{{\Delta}}

\def\si{{\sigma}}

\def\la{{\lambda}}

\def\vare{{\varepsilon}}
\def \eref#1{\hbox{(\ref{#1})}}
\def\th{{\theta}}
\def\th{{\theta}}
\def\si{{\sigma}}
\def\al{{\alpha}}

\renewcommand{\theequation}{\arabic{section}.\arabic{equation}}

\newcommand{\Rd}{{\RR^d}}
\newcommand{\rE}{\mathcal E}
\newcommand{\KK}{\mathcal K}

\newcommand{\HH}{\mathcal H}

\newcommand{\DD}{\mathbb D}
\newcommand{\JJ}{\mathcal J}

\renewcommand{\SS}{\mathcal S}
\def\LL{\mathcal L}
\newcommand{\DDD}{\DD^{1,2}(\HH)}

\newcommand{\ep}{\ensuremath{\varepsilon}}
\newcommand{\ah}{{\alpha_H}}

\renewcommand{\d}{d}
\newcommand{\dcir}{ \!\mathrm{d}^\circ\!}
\newcommand{\dsym}{ \!\mathrm{d}^{\mathrm{sym}}}
\newcommand{\loc}{\mathrm{loc}}
\newcommand{\dom}{\mathrm{Dom\,}}

\newcommand{\sprod}[2]{\left\langle #1,#2 \right\rangle}

\newcommand{\yint}[1]{\int_{\RR^d}\! #1\,\mathrm{d}y}

\newcommand{\sd}{{}^\circ\!\d}
\newcommand{\test}{C_c^\infty(\RR^{d+1})}
\def\dt{d t}
\def\dx{d x}
\def\dy{d y}

\def\dr{d r}
\def\ds{d s}
\newcommand{\Div}{\mathrm{Div\,}}

\renewcommand{\theequation}{\arabic{section}.\arabic{equation}}
\let\Section=\section
\def\section{\setcounter{equation}{0}\Section}

\title{Nonlinear Young integrals and  differential systems in H\"older media}
\date{April 2014}
\author[Y. Hu]{Yaozhong Hu}
\thanks{Y. Hu is partially supported by a grant from the Simons Foundation \#209206
and by a General Research Fund of University of Kansas.}
\address{Department of Mathematics \\
The University of Kansas \\
Lawrence, Kansas, 66045}
\email{yhu@ku.edu, khoale@ku.edu}
\author[K. L\^e]{Khoa N. L\^e}
\begin{document}
\begin{abstract} For   H\"older continuous  random field  $W(t,x)$ and stochastic process
$\varphi_t$, we define nonlinear integral $\int_a^b W(dt, \varphi_t)$ in
various senses, including pathwise and It\^o-Skorohod.  We study
their properties and relations. The stochastic flow in a time dependent
rough  vector field  associated
 with
$\dot \varphi_t=(\partial _tW)(t, \varphi_t)$ is also studied and its
applications to the   transport equation $\partial _t u(t,x)-\partial _t
W(t,x)\nabla u(t,x)=0$ in rough media is   given. The Feynman-Kac solution to the
stochastic partial differential equation with random coefficients
$\partial _t u(t,x)+Lu(t,x) +u(t,x) \partial_t W(t,x)=0$ are given,  where $L$
is a second order elliptic differential operator with random
coefficients (dependent on $W$).  To establish such formula  the main
difficulty is the exponential integrability of some  nonlinear
integrals, which is proved to be true  under some mild conditions on
the covariance of $W$ and on the coefficients of $L$.
Along the way, we also obtain an upper bound for increments of stochastic processes on  multidimensional rectangles by majorizing measures.
\end{abstract}
\maketitle
\tableofcontents
\section{Introduction}
Feynman integral  is an important tool in quantum physics. The
Feynman-Kac formula is a  variant of Feynman integral and plays very
important role in the study of (parabolic) partial differential
equations (see \cite{freidlin} and \cite{simon}).  Recently, there have been several
successes in extending the Feynman-Kac formula to the following
stochastic partial differential equations with noisy (random)
potentials on $[0, T]\times \RR^d$  (see e.g.  \cite{hu-hu-nu-ti},
\cite{hu-lu-nualart12}, and \cite{hunualartsong}): $
\partial _t u(t,x)=\frac12 \Delta u(t,x)+u(t,x)\partial_tW(t,x)
$,  \ where $\Delta$ is the
Laplacian  with respect to spatial variable and $\{\partial_t W(t,x)\,, 0\le
t\le T\,, x\in \RR^d\}$ is a Gaussian noise {\replacecolorred  (the derivatives in the
sense of Schwartz distribution of a Gaussian field)}. As indicated in
the aforementioned papers, there are three tasks  to accomplish for
establishing the Feynman-Kac formula. The first one is to give a
meaning to the nonlinear stochastic integral $\int_a^b  W(\ds,
x+B_s)$ for a $d$-dimensional Brownian motion (whose generator is
$\frac12 \Delta$), independent of $W$. The second one is to
establish the exponential integrability of $\int_a^b W(\ds, x+B_s)$
and hence   the Feynman-Kac expression (which we may
call the Feynman-Kac solution)   has a rigorous meaning. The final task is
to show that  the Feynman-Kac expression is indeed a solution to   the equation
in   certain sense. It
should be emphasized that the independence between $B$ and $W$ plays
crucial role in  previous studies.

In many applications,  one needs to study more general  stochastic
partial differential equations. For example,  in   modeling of   the
pressure in an oil reservoir in the Norwegian sea
with a log normal stochastic
permeability one was led  to study the stochastic partial differential
equation on some bounded domain in $\RR^d$ of the form ${\rm div}
(k(x)\nabla u(x))=f(x)$, where the permeability $k(x)$ is the (Wick)
exponential of white noise,  ${\rm div}$ is the divergence operator,
and $\nabla $ is the gradient operator, see
  \cite{holdenhu} and  in particular the references therein.
Recently, there have been a great amount of research on {\it
uncertainty quantification}.  Among the huge literature on this topic
let us just mention the  books \cite{grigoriu}, \cite{xin}, and  the
references therein. Many different types of stochastic partial
differential equations with random coefficients have been studied.

This motivates us to study the Feynman-Kac formula for general
stochastic partial differential equations with random coefficients,
namely,
\begin{equation}
\partial _t u(t,x)+Lu(t,x) +u(t,x)\partial_tW(t,x)=0\,,\label{e.generalspde}
\end{equation}
where
\[
Lu(t,x)=\frac12 \sum_{i,j=1}^d a_{ij}(t,x, W) \partial _{x_ix_j}^2
u(t,x) +\sum_{i =1}^d b_{i }(t,x, W) \partial _{x_i } u(t,x)
\]
and for notational simplicity and up to a time change we assume that
the terminal condition $u(T, x)=u_T(x)$ is given. {\replacecolorred  The product
$u(t,x)\partial_tW(t,x)$ in \eref{e.generalspde} is the ordinary product}.   If
$\si(t,x)=(\si_{ij}(t,x,W))_{1\le i,j\le d}$ satisfies
$a=\si \si^T$ and if $X_t^{r,x}$ is the solution of the following
stochastic differential equation
\begin{equation}
\d X_t^{r,x}=\si(t,X_t^{r,x}, W)) \delta B_t+b(t,X_t^{r,x}, W) \dt\,, \quad r\le
t\le T\,, \quad X_r^{r,x}=x\,, \label{e.X-trx}
\end{equation}
then $u(r,x)=\EE ^B \left\{ u_T(X_T^{r,x})\exp\left[\int_r^T W(\ds,
X_s^{r,x})\right]\right\}$  should be  the Feynman-Kac solution to
\eref{e.generalspde} with $u(T, x)=u_T(x)$.   As indicated above,
there are three tasks to complete to justify the above claim. The
first task   to give a meaning to the nonlinear stochastic integral
$\int_r^T W(\ds, X_s^{r,x})$ is much more challenging than what has
been accomplished before (see for instance \cite{hu-hu-nu-ti},
\cite{hu-lu-nualart12}, and  \cite{hunualartsong}).  Although  the major
focus of the work
\cite{hu-lu-nualart12} is to give a meaning to the nonlinear integral
$\int_r^T W(\ds, X_s^{r,x})$.  However, in that paper
$X_s^{r,x}=B_s$ is a Brownian  motion  {\it independent of $W$} and
then we can consider $X_s^{r,x}$ as ``deterministic".
In our current situation
  since $X_s^{r,x}$ and $W$ are correlated,
  the nonlinear integral is a true stochastic one. In addition,   the noise $W$ may enter to
$X_s^{r,x}$ in an  anticipative way. Thus, the general stochastic
calculus for semimartingales cannot be applied in a straightforward
way due to the lack of adaptedness.

If $W(t,x)$ is only continuous
in $t$ (without any H\"older continuity in $t$) but  has  certain
differentiability on $x$,  then we can use semimartingale structure
of $X_t^{r,x}$ plus some   new techniques developed in Section
\ref{sec.feykac} to define $\int_r^T W(\ds, X_s^{r,x})$ and study
the corresponding Feynman-Kac solution to \eref{e.generalspde}. This
result   extends  the work of \cite{hu-lu-nualart12} in two aspects.
One is that the Laplacian
 is replaced by general second order elliptic operator with general  and in particular  random
 coefficients. The other one  is that in \cite{hu-lu-nualart12},   the
 Hurst parameter $H$ in time is assumed to be greater than $1/4$, while the
  result of this paper is applicable to fractional Brownian field
  whose  Hurst parameter $H$ in time can be any number between $0$ and $1$.

When $W(t,x)  $ has certain (H\"older)
regularity in time variable,  it is natural to see whether  one can
 reduce its   regularity in spatial variable
$x$ to define  $\int_r^T W(\ds, X_s^{r,x})$. Having in mind the
recent development on rough path analysis and encouraged by the
previous success in the case when $X_s^{r,x}$ is the Brownian
motion (\cite{hu-hu-nu-ti},
\cite{hu-lu-nualart12}, and  \cite{hunualartsong}),  we dedicate ourselves to
a systematic study of the   nonlinear integral $\int_a^b W(\ds,
\varphi_s)$, where $W(s,x)$ is a H\"older continuous function on $s$
and $x$ and $\varphi_s$ is also a H\"older continuous function.
% {\replacecolorred  We present mainly three additional   approaches to study the  above nonlinear
% integral,  besides  the semimartingale approach mentioned
% previously. Two of them are pathwise ones: by smooth approximation
% of $W$ and by fractional integration by parts.}
Some elementary
properties of the integral are obtained as well. These results are presented in
 Section \ref{sec.pathint}.  Let us emphasize  that this nonlinear
integral $\int_a^b W(\ds, \varphi_s)$ is defined in a purely
deterministic way.  In fact, it is an extension of integration of Young type
(\cite{young}).

{\replacecolorred  For Gaussian noise a very important (linear) stochastic integral
is the It\^o (or It\^o-Skorohod) integral.  It is also called divergence
 integral.  In  probability theory, this integral is a central concept in stochastic analysis.
For our  stochastic partial differential equation  \eref{e.generalspde}
it  is needed  if the product $u(t,x)\partial_tW(t,x)$ there is Wick product.
   We shall introduce   the nonlinear   It\^o-Skorohod  integral $\int_a^b W(\d
s, \varphi_s)$ ($\varphi_s$ depends on $W$)  by using  Malliavin
calculus.  This is done  in    Appendix \ref{sec.prelim}.
 The relation of
this integral with other types of integrals
is  also discussed in  this section.}
Naturally,  readers may ask the question to study the It\^o-Skorohod type stochastic
differential equation $\partial _t u(t,x)+Lu(t,x)
+u(t,x)\diamond\partial_t W(t,x)=0$, where $u(t,x)\diamond W(t,x)$
denotes the Wick product between $u(t,x) $ and $\partial_t W(t,x)$.
However,   this seems to be very complex since $L$ depends on $W$ in
a sophisticated way and {\replacecolorred  will  not be considered in this work.}

When  $W(t,x)$ is a semimartingale in $t$ for any
fixed $x$ and is smooth in $x$ for any fixed $t$, there has been many studies on
stochastic flows which contributes significantly to the study of
stochastic partial differential equations (see \cite{kunita} and the
references therein).  {\replacecolorred  The important tool there is the nonlinear
stochastic integral (with respect to semimartingale) and the corresponding flow.
After defining the nonlinear Young integral and}
motivated by this aspect, we   study the
pathwise flow associated with time dependent   rough vector field
$W(t,x)$. That is, we study the differential equation $
\varphi_t=x+\int_0^t W(\ds, \varphi_s)$ under joint H\"older continuity assumptions of
$W(t,x)$.  We shall
study the flow and other properties of the solution $\varphi_t$.
This is presented in Section \ref{sec.diffeqn}. The applications to the
transport equation in rough media of the form $\partial_t u
(t,x)-\partial_t W(t,x)\nabla u(t,x) =0$ are also investigated in
Subsection \ref{subsec.transport}.

After completion of the {\replacecolorred  first}  task of defining the nonlinear integral
another major difficulty {\replacecolorred  (the above mentioned second task)}
to overcome in the construction of the
Feynman-Kac solution is the exponential integrability of $\int_r^T
W(\ds, X_s^{r,x})$.  In the previous work of \cite{hu-hu-nu-ti},
\cite{hu-lu-nualart12}, and \cite{hunualartsong}, this is achieved
by showing $\EE \left[u^2(r, x)\right]$ is finite.  If we continue
to follow the idea in aforementioned papers, then we are led   to show
\[
 \EE ^{B, \tilde B} \EE^ W\left\{ u_T(X_T^{r,x}) u_T(\tilde
X_T^{r,x})\exp\left[\int_r^T W(\ds, X_s^{r,x})+\int_r^T W(\ds, \tilde
X_s^{r,x})\right]\right\}\,,
\]
is finite,  where $\tilde X_t^{r,x}$ is the solution to the equation
\eref{e.X-trx} with a Brownian motion $\tilde B$, independent of $B$
and $W$.   It seems to us that in our situation, due to the
dependence of $X_{t}^{ r,x}$ on $W$,  it is   hard to show the
above quantity is finite. To get around this difficulty, our strategy is then to show that
$u(r,x)=\EE ^B \left\{ u_T(X_T^{r,x})\exp\left[\int_r^T W(\ds,
X_s^{r,x})\right]\right\}$ is finite for every fixed path of $W$,
assuming some mild pathwise conditions on $W$  {\replacecolorred  (see for instance
  \eqref{e.w-fk}).
    The third (and the last)  task to  show that the Feynman-Kac solution is indeed
a solution to \eref{e.generalspde}  is relatively easier  and will be completed by using
approximation technique.
All these will be done Section \ref{sec.feykac}.
}

{\replacecolorred  Intentionally,  the paper is divided into two parts. The first three  chapters
can be read without knowledge of probability theory.  A single (rough) sample $W(t,x)$
satisfying some joint H\"older continuity and growth conditions is considered. For instance, the
  (stochastic) partial
differential equation \eref{e.generalspde}, the nonlinear Young integral (Definition \ref{sub:definition}),
   and the transport equation
\eref{eq:transport} are considered for every fixed sample path $W(t,x)$.   Since
$W(t,x)$  is fixed, we also drop the dependence of $a_{ij}(t,x)$ and $b_{i}(t,x)$ on
$W$ throughout the paper. So,  the integrals and  equations are defined and studied for a
(fixed) rough function.  The stochastic partial differential
equation considered  in Section \ref{sec.feykac} is for a single
rough sample path.  But Brownian motion is used to represent
the solution.
}

% achieve this in two steps. First,  we assume  some  mild  pathwise
% conditions on $W$ to obtain pathwise finiteness of $u(t,x)$.
%Then
%this pathwise condition on $W$ is shown to hold true under some mild
%conditions on   the covariance of $W$. This is accomplished   in
%Section \ref{sec.cov-path}  by using
%concentration inequalities and by appealing to the Dudley's entropy integral bound
%(\cite{dudley}).

{\replacecolorred  As a probabilist,} one may ask whether a stochastic process satisfies
the joint H\"older continuity conditions together with the growth conditions   assumed throughout the paper. For instance, condition \eref{e.w-fk} in Section \ref{sec.feykac}   requires the paths of $W$ to satisfy
\begin{equation}\label{eqn.intro.hder}
    |W(s,x)-W(s,y)-W(t,x)+W(t,y)|\le C (1+|x|^\beta+|y|^\beta) |t-s|^{\tau}|x-y|^{\lambda}
\end{equation}
for all $s,t\in[0,T]$ and $x,y\in\RR^d$. We give a partial answer for this problem in Section \ref{sec.cov-path}, where an extra assumption $|x-y|\le \delta$ for a fixed constant $\delta$ is imposed. Pathwise boundedness and  pathwise  regularity (H\"older continuity) have been extensively studied  in the literature (see Section \ref{sec.cov-path} for more detailed discussions.)  However,  estimates similar to \eqref{eqn.intro.hder} has not been studied thoroughly. Comparing with the existing literature (e.g. \cite{marcus-rosen}, \cite{talagrand-book}), where estimates for increments over one parameter interval  are obtained, the left side of \eqref{eqn.intro.hder} is an increment over   two parameter rectangle. Difficulties arise because the increments behave differently when the number of parameters get large. For instance, the corresponding entropic volumetric to the left side of \eqref{eqn.intro.hder}, $d((s,x),(t,y))=(\EE|W(s,x)-W(s,y)-W(t,x)+W(t,y)|^2)^{1/2}$, does not satisfy the triangular inequality. Therefore, classical estimates (such as those appear in \cite{talagrand-book}) are no longer applicable, new tools are needed to prove \eqref{eqn.intro.hder}.
If in \eqref{eqn.intro.hder}, $x,y$ are restricted in a compact set, a similar problem has been considered by the authors
by extending the Garsia-Rodemich-Rumsey inequality (\cite{hule2012}). Nevertheless, the exact growth rate when $x,y$ get large is not discussed in that paper.
Motivated by this requirement,  we extend and sharpen our previous work in \cite{hule2012}
so that it is applicable to our current situation.
 Since in many applications, $W$ will be a
Gaussian noise, we focus   on the case $W$ satisfies normal
concentration inequalities to obtain the desirable pathwise property
from the covariance structure of the process.
As is well-known it is usually hard to obtain   properties for each  sample path  in the theory of stochastic processes.
We hope this work will shed some light along this direction.

%We also present  some preliminary material on Malliavin calculus and
%fractional calculus we frequently used in Section \ref{sec.prelim}.

\textbf{Notations:}
We collect here some notations that we will use throughout the entire paper. $A\lesssim B$ means there is a constant $C$ such $A\le CB$. We represent a vector $x$ in $\RR^d$ as  a matrix of dimension $d\times 1$, $A^T$ represents the transpose of a matrix $A$. Sometimes we write $x_\bullet$  for column vector $x^T$ and   $x^\bullet$ for the row vector $x$. We use the Einstein convention on summation over repeated  indices. For instance, $b_ic_i$ abbreviates for $\sum_{i=1}^db_ic_i$

\setcounter{equation}{0}
\section{Nonlinear Young integral}\label{sec.pathint}
\global\long\def\ctau{C^{\tau}\left([t_{0}-T,t_{0}+T]\right)}

% In this section we shall use fractional calculus to define the pathwise nonlinear stochastic integral $\int_{a}^{b}W(\dt,\varphi_{t})$. This method, to a certain extent, has broader applications since it only relies on regularity of the sample paths of $W$ and $\varphi$. More precisely, it is applicable to stochastic processes with H\"older continuous sample paths.

% Another advantage of this approach is that in the theory of stochastic processes it is usually difficult to obtain almost sure  type of results. If the  sample paths of the process is H\"older continuous, then one can apply this approach to each sample path and almost surely results are then automatic. {\replacecolorred Presumably, this section is a continuation of the previous section. } Since the method described below does not make use of the covariance structure of $W$, and hence is rather different from what is considered in Section \ref{sec.sto.int}, we devote one entire section for it.

Let $W$ and $\varphi$ be $\RR^d$-valued functions defined on $\RR\times\RR^d$ and $\RR^d$ respectively. We define in the current section the nonlinear Young integration $\int W(ds,\varphi_s)$.

% {\replacecolorred Although the current research is carried independently, let us mention that the concept of nonlinear Young integration also appears in \cite{catellier2012averaging,chouk2013nonlinear,chouk2014nonlinear}, see also \cite{GubinelliTindel}. }
We make the following assumption on the regularity of $W$
\begin{enumerate}[label=\bm{$(W)$}]
% \item\label{cond.w1} $W: [a, b]\times \RR^d\rightarrow \RR^d$  is H\"older
% continuous with respect to $t$ and $x$:
% \begin{align}
%  \left|W(s,x)-W(t,x) \right|
% &\le  L (1+|x|  )^\be |t-s|^{\tau}  \label{eq:cond.w1-t}\\
% \left|W(t,y)-W(t,x) \right|
% &\le  L (1+|x| +|y|)^\be |x-y|^{\la}  \label{eq:cond.w1-x}\,.
% \end{align}
\item\label{cond.w} There are constants $\tau\,, \la\in (0, 1]$, $\beta\ge0$ such that
for all $a<b$, the seminorm
\begin{equation}\label{eq:cond.w2}
\begin{split}
  &\|W\|_{\be, \tau, \la; a, b }\\
  :&=\sup_{\substack{a\le s<t\le b\\ x,y\in \RR^d;x\neq y}}\frac{
\left|W(s,x)-W(t,x)-W(s,y) + W(t,y)\right|}{
 (1+|x|+|y| )^\be  |t-s|^{\tau}|x-y|^{\lambda}}\\
 &\quad+ \sup_{\substack{a\le s<t\le b\\ x\in \RR^d}}\frac{
\left|W(s,x)-W(t,x)\right|}{
 (1+|x| )^{\be+\lambda}  |t-s|^{\tau}}+\sup_{\substack{a\le t\le b\\ x,y\in \RR^d;x\neq y}}\frac{
\left|W(t,y) - W(t,x)\right|}{
 (1+|x|+|y| )^\be  |x-y|^{\lambda}} \,,
\end{split}
\end{equation}
is finite.
\end{enumerate}
About the function $\varphi$, we assume
\begin{enumerate}[label=$\bm{(\phi)}$]
  \item\label{cond.phi} $\varphi$ is locally H\"older continuous of order $\ga\in (0, 1]$. That is the seminorm
  \[
  \varphi_{\ga; a, b}=\sup_{a\le s<t\le b}\frac{|\varphi(t)-\varphi(s)|}{|t-s|^\ga}\,,
  \]
  is finite for every $a<b$.
\end{enumerate}
Throughout the current section, we assume that $\tau+\lambda \gamma>1$. Among three terms appearing in \eqref{eq:cond.w2}, we will pay special attention to the first term. Thus, we denote
\begin{equation*}
  [W]_{\beta,\tau,\lambda;a,b}=\sup_{\substack{a\le s<t\le b\\ x,y\in \RR^d;x\neq y}}\frac{
\left|W(s,x)-W(t,x)-W(s,y) + W(t,y)\right|}{
 (1+|x|+|y| )^\be  |t-s|^{\tau}|x-y|^{\lambda}}\,.
\end{equation*}
 When $\be=0$, then we denote $\|W\|_{ \tau, \la; a,b}:= \|W\|_{0,  \tau, \la; a,b}$.
If $a, b$ are clear in the context, we frequently omit the dependence on $a, b$.
For instance,   $\|W\|_{\be,  \tau, \la}$ is an abbreviation for $\|W\|_{\be,  \tau, \la; a,b}$, $\|\varphi\|_{\gamma}$ is an abbreviation for $\|\varphi\|_{\gamma;a,b}$  and so on.
We shall assume that $a$ and $b$ are finite.  It is easy
to see that for any $c\in [a, b]$
\[
\sup_{a\le t\le b}|\varphi(t)|=
\sup_{a\le t\le b}| \varphi(c)+\varphi(t)-\varphi(c)|
\le |\varphi(c)|+\| \varphi\|_{\ga} |b-a|^\ga<\infty\,.
\]
Thus assumption \ref{cond.phi} also implies that
\[
\|\varphi\|_{\infty; a, b}
:=\sup_{a\le t\le b}|\varphi(t)|<\infty\,.
\]
For the results presented in this section, the condition \ref{cond.w} can be relaxed to
\begin{enumerate}[label=\bm{$(W')$}]
\item \label{cond.wprime} There are constants $\tau\,, \la\in (0, 1]$, such that
for all $a<b$ and compact set $K$ in $\RR^d$, the seminorm
\begin{equation*}
\begin{split}
  &\sup_{\substack{a\le s<t\le b\\ x,y\in K;x\neq y}}\frac{
\left|W(s,x)-W(t,x)-W(s,y) + W(t,y)\right|}{
   |t-s|^{\tau}|x-y|^{\lambda}}\\
 &\quad+ \sup_{\substack{a\le s<t\le b\\ x\in K}}\frac{
\left|W(s,x)-W(t,x)\right|}{
  |t-s|^{\tau}}+\sup_{\substack{a\le t\le b\\ x,y\in K;x\neq y}}\frac{
\left|W(t,y) - W(t,x)\right|}{
  |x-y|^{\lambda}} \,,
\end{split}
\end{equation*}
is finite.
\end{enumerate}
However, the polynomial growth rate is needed in the following sections to solve
 differential equations.

For later purpose, we denote $C^{(\tau,\lambda)}_{\beta}(\RR\times\RR^d)$ (respectively $C^{(\tau,\lambda)}_\loc(\RR\times\RR^d)$) the collection of all functions $W$ satisfying condition \ref{cond.w} (respectively \ref{cond.wprime}). $\kappa$ denotes a universal generic constant depending only on $\la, \tau, \al$
and independent of $W$, $\varphi$ and $a, b$. The value
of $\kappa$ may vary from one  occurrence to another.
\subsection{Definition} % (fold)
\label{sub:definition}
We define the nonlinear integral $\int W(ds,\varphi_s)$ as follows.
\begin{definition}
	Let $a,b$ be two fixed real numbers, $a<b$. Let $\pi=\{a=t_0<t_1<\cdots<t_m=b\}$ be a partition of $[a,b]$ with mesh size $|\pi|=\max_{0\le i\le m-1}|t_{i+1}-t_i|$. The Riemann sum corresponding to $\pi$ is
	\begin{equation}\label{eq.JW}
	 	J_{\pi}=\sum_{i=1}^{m-1}W(t_{i+1},\varphi_i)-W(t_i,\varphi_i)\,.
	\end{equation}
	If the sequence of Riemann sums $J_\pi$'s is convergent when $|\pi|$ shrinks to 0, we denote the limit as the nonlinear integral $\int_a^b W(ds,\varphi_s)$.
\end{definition}
We observe that in the particular case when $W(t,x)=g(t)x$  for some functions $g:\RR\to\RR$, the nonlinear integral $\int_a^b W(ds,\varphi_s)$ defined above, if exists, coincides with the Riemann-Stieltjes integral $\int_a^b \varphi_s dg(s)$. It is well known that if $\varphi$ and $g$ are H\"older continuous with exponents $\alpha,\beta$ respectively and $\alpha+\beta>1$, then the Riemann-Stieltjes integral $\int_a^b \varphi_s dg(s)$ exists and is called Young integral (\cite{young}).

More generally, for each partition $\pi$ of an interval $[a,b]$, one can consider the (abstract) Riemann sum
\begin{equation}\label{eq.jmu}
	J_\pi(\mu) =\sum_{i=1}^{m-1}\mu(t_i,t_{i+1})
\end{equation}
where $\mu$ is a function defined on $[a,b]^2$ with values in a Banach space. A sufficient condition for convergence of the limit $\lim_{|\pi|\downarrow0}J_\pi(\mu)$ is obtained by Gubinelli in \cite{Gubinelli04} via the so-called sewing map. This point of view has important contributions to  Lyons' theory of rough paths (\cite{lyons94,Lyons98}).  Since we will apply Gubinelli's sewing lemma, we restate the result as follows.

% To see the connection between the integral $\int_a^bW(\ds,\varphi_s)$ defined in \eqref{eqn.def.fintw} (through fractional calculus)  and  the   Riemann-Stieltjes integral, we follow the approach of D. Feyel and A. de La Pradelle in \cite{FeyelPradelle06}.{\replacecolorred should cite to \cite{Gubinelli04} } We first recall two important lemmas in that paper.
\begin{lemma}[Sewing lemma]\label{sew.lem}
  Let $\mu$ be a continuous function on $[a,b]^2$ with values in a Banach space $(B,\|\cdot\|)$  and
  let $\ep>0$. Suppose that $\mu$ satisfies
  \[
 \|\mu(s,t)-\mu(s,c)-\mu(c,t)\|\le K|t-s|^{1+\ep}\quad \forall \ a\le s\le c\le t\le b \,.
  \]
  Then there exists a function $\JJ\mu(t)$ unique up to an additive constant such that
  \begin{equation}\label{est.sew}
  \|\JJ\mu(t)-\JJ\mu(s)-\mu(s,t)\|\le K (1-2^{-\ep})^{-1} |t-s|^{1+\ep}\quad \forall \  a\le s\le t\le b\,.
  \end{equation}
  In addition, when $|\pi|$ shrinks to 0, the Riemann sums \eqref{eq.jmu} converge to $\JJ \mu(b)-\JJ \mu(a)$.
\end{lemma}
In what follows, we adopt the notation $\JJ_a^b \mu=\JJ \mu(b)-\JJ \mu(a)$. The map $\mu\mapsto \JJ \mu$ is called the sewing map. The setting of Lemma \ref{sew.lem} is adopted from \cite{FeyelPradelle06}. In several occasions, one needs to prove a relation between two or more integrals. The following result provides a simple method for this problem.
\begin{lemma}\label{lemma.arbi}
    Suppose $\mu_1$ and $\mu_2$ are two functions as in Lemma \ref{sew.lem}. In addition, assume that
    \begin{equation*}
        |\mu_1(s,t)-\mu_2(s,t)|\le C |t-s|^{1+\varepsilon'}\quad\forall a\le s\le t\le b
    \end{equation*}
    for some positive constant $\ep'$. Then $\JJ \mu_1$ and $\JJ \mu_2$ are different by an absolute constant. That is $\JJ_s^t \mu_1=\JJ_s^t \mu_2$ for all $s,t$.
\end{lemma}
\begin{proof}
    From Lemma \ref{sew.lem}, $\JJ(\mu_1- \mu_2)=\JJ \mu_1-\JJ \mu_2$ and
    \begin{align*}
        |\JJ_s^t(\mu_1- \mu_2)|&\lesssim|\mu_1(s,t)- \mu_2(s,t)|+|t-s|^{1+\ep}
        \\&\lesssim |t-s|^{1+\ep'}+|t-s|^{1+\ep}
    \end{align*}
    for all $s,t$. This implies $\JJ_s^t(\mu_1- \mu_2)=0$ for all $s,t$.
\end{proof}

Returning to our main objective of the current section, we consider
 \[
 \mu(s,t)=W(t,\varphi_s)-W(s,\varphi_s).
 \]
Then the condition in  Lemma \ref{sew.lem}  is guaranteed by  \ref{cond.w},  and \ref{cond.phi}. Indeed, for every $s<c<t$,
\begin{align*}
  |\mu(s,t)&-\mu(s,c)-\mu(c,t)|\\
  &=|W(t, \varphi_s)-W(c, \varphi_s)-W(t, \varphi_c)+W(c, \varphi_c)|\\
  &\le [W]_{\beta,\tau,\lambda} (1+\|\varphi\|_\infty^\beta)(t-s)^\tau|\varphi_s- \varphi_c|^\lambda\\
  &\le [W]_{\beta,\tau,\lambda}(1+\|\varphi\|_\infty^\beta)\|\varphi\|^{\lambda}_{\gamma} (t-s)^{\tau+\lambda \gamma}\,.
\end{align*}
Hence, by combining the sewing lemma and the previous estimate, we obtain
\begin{proposition}\label{prop.riemann.w} Assuming the conditions \ref{cond.w}, \ref{cond.phi}  with  $\tau+\lambda \gamma>1$, the sequence of Riemann sums \eqref{eq.JW} is convergent when $|\pi|$ goes to $0$. In other words, the nonlinear integral $\int_a^bW(ds,\varphi_s)$ is well-defined.

In addition, the following estimate holds
\begin{multline}
\label{est.W.c}
\left| \int_s^{t}W(dr,\varphi_{r})-W(t,\varphi_c)+W(s,\varphi_c) \right|\\
\le\kappa  \|W\|_{\tau, \la\,;  a, b}
(1+\|\varphi\|_{\infty }^\be) \|\varphi\|_{\ga\,;a, b}^\la (t-s)^{\tau+\la \ga}
\end{multline}
for all $a\le s \le c \le t\le b$.
\end{proposition}

 \begin{remark}{\replacecolorred  After the completion of this work,  we are brought to the attention of the work \cite{catellier2012averaging}  (and  also \cite{chouk2013nonlinear,chouk2014nonlinear,GubinelliTindel}),
where   a similar nonlinear Young integral is studied.    The objective of that paper is to define the averaging of the form $\int_0^t f(X_u) du$ for some process $X_u$ and for
some irregular function $f$.
The sewing lemma that
we follow is from  \cite{FeyelPradelle06} , which  is  after the work of  \cite{Gubinelli04}.}
\end{remark}

\begin{remark} (i) In the particular case when $W(t,x)=g(t)x$, Proposition \ref{prop.riemann.w} reduces to the existence of the Young integral $\int \varphi_s dg(s)$. Hence, from now on we refer   the integral $\int W(ds,\varphi_s)$ as nonlinear Young integral.

	(ii) In Proposition \ref{prop.riemann.w}, we can also consider the Riemann sums with right-end points \[J_\pi^+=\sum_{i=0}^{m-1}[W(t_{i+1},\varphi_{t_{i+1}} )-W(t_i,\varphi_{t_{i+1}})]\,.\] Then the corresponding limit exists and equals to $\int_a^b W(ds,\varphi_s)$. This is a straightforward consequence of Lemma \ref{lemma.arbi}.
\end{remark}
It is evident that
\[
\int_s^t W(d  r, \varphi_r)=\int_s^c W(d  r, \varphi_r) +\int_c^t W(d
r, \varphi_r)\qquad \forall \
s<c<t\,.
\]
This together with \eref{est.W.c} imply easily the following.
\begin{proposition}\label{prop.frac} Assume that   \ref{cond.w} and \ref{cond.phi}
hold with $\lambda\gamma+\tau>1$.  As a function of $t$,  the indefinite
integral $\left\{ \int_{a}^t W(\d s,\varphi_{s})\,, \ a\le t\le b\right\}$ is H\"older continuous
of exponent $\tau$.
\end{proposition}
Fractional calculus is very useful in the study of (linear) Young integral.
It leads to some detailed properties of
the integral and solution of a differential equation (see \cite{hunualart},
\cite{hunualart09},    and the references therein).
It is interesting  to extend this approach to nonlinear Young integral. In fact, the authors obtain in \cite{hule2015}  the following presentation for the nonlinear Young integral by
using fractional calculus. Since this method is not pursued in the current paper, we refer the readers to \cite{hule2015} for further details.

\begin{theorem}\label{fractional}
	Assume the conditions  \ref{cond.w} and \ref{cond.phi} are satisfied. In addition, we suppose that $\la \ga+\tau>1$. Let $\alpha\in(1- \tau,\lambda \tau)$. Then the following identity holds
	\begin{equation}\label{eq:def.w}
		\begin{split}
			&\int_{a}^{b}W(\dt,\varphi_{t})
			% &=(-1)^{\alpha}\int_{a}^{b}D_{a+}^{\alpha,t'}D_{b-}^{1-\alpha,t}W_{b-}(t,\varphi_{t'})|_{t'=t}\dt
			\\& = -\frac{1}{\Gamma(\alpha)\Gamma(1-\alpha)}\left\{ \int_{a}^{b}\frac{W_{b-}(t,\varphi_{t})}{(b-t)^{1-\alpha}(t-a)^{\alpha}}\dt\right.\\
			 &\left.+\alpha\int_{a}^{b}\int_{a}^{t}\frac{W_{b-}(t,\varphi_{t})-W_{b-}(t,\varphi_{r})}{(b-t)^{1-\alpha}(t-r)^{\alpha+1}}\dr\dt\right. \\
   			&+ (1-\alpha)\int_{a}^{b}\int_{t}^{b}\frac{W(t,\varphi_{t})-W(s,\varphi_{t})}
 			{(s-t)^{2-\alpha}(t-a)^{\alpha}}\ds\dt \\
   			& + \left.\alpha(1-\alpha)\int_{a}^{b}\int_{a}^{t}\int_{t}^{b}
 			\frac{W(t,\varphi_{t})-W(s,\varphi_{t})-W(t,\varphi_{r})+W(s,\varphi_{r})}{(s-t)^{2-\alpha}
 			(t-r)^{\alpha+1}}\ds\dr\dt\right\},
		\end{split}
	\end{equation}
	where $W_{b-}\left(t,x\right)=W\left(t,x\right)-W\left(b,x\right)$.
\end{theorem}
% subsection definition (end)
\subsection{Mapping properties} % (fold)
\label{sub:mapping_properties}
Let $\mu$ be a function as in Lemma \ref{sew.lem}. Let us define the quality
\begin{equation*}
  [\mu]_{1+\varepsilon;I}=\sup_{s,c,t\in I:s<c<t}\frac{|\mu(s,t)-\mu(s,c)-\mu(c,t)|}{|t-s|^{1+\varepsilon}}\,.
\end{equation*}
In several occasions, given two functions $\mu_1$ and $\mu_2$ such that $[\mu_1]_{1+\varepsilon} $ and $[\mu_2]_{1+\varepsilon}$ are finite, one would like to compare the  integrals $\JJ \mu_1$ and $\JJ \mu_2$.
The following result answers this question.
\begin{lemma}\label{lem.mu12}
  Let $\mu_1$ and $\mu_2$ be two continuous functions on $[a,b]^2$ such that $[\mu_1]_{\alpha} $ and $[\mu_2]_ \alpha $ are  finite for some $\alpha>1$. Then for every $s,t\in[a,b]$
  \begin{equation*}
    |\JJ_s^t \mu_1-\JJ_s^t \mu_2|\le|\mu_1(s,t)-\mu_2(s,t)|+(1-2^{1- \alpha})^{-1} [\mu_1- \mu_2]_{\alpha;[s,t]}|t-s|^\alpha
  \end{equation*}
\end{lemma}
\begin{proof} The proof is rather trivial thanks to the linearity nature of Lemma \ref{sew.lem}.
  Put $\mu=\mu_1- \mu_2$. Notice that $[\mu]_ \alpha\le [\mu_1]_ \alpha+[\mu_2]_ \alpha<\infty$.
  Thus we can apply Lemma \ref{sew.lem} to $\mu$. The claim follows after observing that $\JJ \mu=\JJ \mu_1-\JJ \mu_2$.
\end{proof}

As an application, we study the dependence of the nonlinear Young integration $\int W(\ds,\varphi_s)$ with respect to the medium $W$ and the integrand $\varphi$.
\begin{proposition}\label{prop.int.w12}
  Let $W_1$ and $W_2$ be real valued functions on $\RR\times\RR^d$ satisfying the condition \ref{cond.w}.  Let $\varphi$ be a function in $C^\gamma(\RR;\RR^d)$ and let   $\tau+\lambda \gamma>1$. Then
  \begin{multline*}
    |\int_a^b W_1(\ds,\varphi_s)-\int_a^b W_2(\ds,\varphi_s)|\le |W_1(b,\varphi_a)- W_1(a,\varphi_a)-W_2(b,\varphi_a)+W_2(a,\varphi_a)|\\
    +c(\|\varphi\|_\infty) [W_1-W_2]_{\beta,\tau,\lambda}\|\varphi\|_{\gamma}|b-a|^{\tau+\lambda \gamma}
  \end{multline*}
\end{proposition}
\begin{proof}
  Let $a<c<b$. Put
  \begin{align*}
    & \mu_1(a,b)=W_1(b,\varphi_a)-W_1(a,\varphi_a)\,,\\
    & \mu_2(a,b)=W_2(b,\varphi_a)-W_2(a,\varphi_a)\,,\\
    & \mu=\mu_1- \mu_2\,.
  \end{align*}
  The argument before Proposition \ref{prop.riemann.w} shows that
  \begin{align*}
    & [\mu]_{\tau+\lambda \gamma}\le[W_1-W_2]_{\beta,\tau,\lambda}(1+\|\varphi\|_{\infty}^\beta)\|\varphi\|_{\gamma}\,.
  \end{align*}
  The proposition follows from Lemma \ref{lem.mu12}.
\end{proof}
\begin{proposition}\label{prop.int.phi12}
  Let $W$ be a function on $\RR\times\RR^d$ satisfying the condition \ref{cond.w}. Let $\varphi^1$ and $\varphi^2$ be two functions in $C^\gamma(\RR;\RR^d)$ and let  $\tau+\lambda \gamma>1$. Let $\theta\in(0,1)$ such that $\tau+\theta \lambda \gamma>1$. Then for any $u<v$
  \begin{multline*}
    |\int_u^v W(\ds,\varphi_s^1)-\int_u^v W(\ds,\varphi_s^2)|\\
    \le C_1[W]_{\beta,\tau,\lambda}\|\varphi^1- \varphi^2\|_{\infty}^\lambda|v-u|^\tau \\
    +C_2[W]_{\beta,\tau,\lambda}\|\varphi^1- \varphi^2\|_{\infty}^{\lambda(1-\theta)} |v-u|^{\tau+\theta\lambda \gamma}\,,
  \end{multline*}
  where $C_1=1+\|\varphi^1\|_{\infty}^\beta+\|\varphi^2\|_{\infty}^\beta$ and $C_2=2^{1-\theta} C_1(\|\varphi^1\|_{\gamma}^\lambda+\|\varphi^1\|_{\gamma}^\lambda)^\theta$.
  % \begin{equation*}
  %   C_1=1+\|\varphi^1\|_{\infty}^\beta+\|\varphi^2\|_{\infty}^\beta\,,\quad
  %   C_2=2^{1-\theta} C_1(\|\varphi^1\|_{\gamma}^\lambda+\|\varphi^1\|_{\gamma}^\lambda)^\theta\,.
  % \end{equation*}
\end{proposition}
\begin{proof}
  We put $\mu_1(a,b)=W(b,\varphi_a^1)-W(a,\varphi_a^1)$, $\mu_2(a,b)=W(b,\varphi_a^2)-W(a,\varphi_a^2)$
  % \begin{align*}
  %   &\mu_1(a,b)=W(b,\varphi_a^1)-W(a,\varphi_a^1)\\
  %   &\mu_2(a,b)=W(b,\varphi_a^2)-W(a,\varphi_a^2)
  % \end{align*}
  and $\mu=\mu_1- \mu_2$. Applying Lemma \ref{lem.mu12}, we obtain, for any $\theta\in(0,1)$ such that $\tau+\theta \lambda \gamma>1$
  \begin{multline*}
    |\int_u^v W(\ds,\varphi_s^1)-\int_u^v W(\ds,\varphi_s^2)|\\
    \le |W(v,\varphi_u^1)- W(u,\varphi_u^1)-W(v,\varphi_u^2)+W(u,\varphi_u^2)|\\
    +[\mu]_{\tau+\theta\lambda \gamma} |v-u|^{\tau+\theta\lambda \gamma}\,.
  \end{multline*}
  Notice that
  \begin{equation*}
    |W(v,\varphi_u^1)- W(u,\varphi_u^1)-W(v,\varphi_u^2)+W(u,\varphi_u^2)|
    \le C_1[W]_{\beta,\tau,\lambda}|u-v|^\tau \|\varphi^1- \varphi^2\|_{\infty}^\lambda\,.
  \end{equation*}
  It remains to estimate $[\mu]_{\tau+\theta\lambda \gamma}$. It is obvious that for $i=1,2$
  \begin{equation*}
    [\mu_i]_{\tau+\lambda \gamma}\le[W]_{\beta,\tau,\lambda}(1+\|\varphi^i\|_{\infty}^\beta)\|\varphi^i\|_{\gamma}^\lambda\le C_1 [W]_{\beta,\tau,\lambda}\|\varphi^i\|_{\gamma}^\lambda
  \end{equation*}
  and hence
  \begin{align*}
    [\mu]_{\tau+\lambda \gamma}\le[\mu_1]_{\tau+\lambda \gamma}+[\mu_2]_{\tau+\lambda \gamma}
    % &\le[W]_{\beta,\tau,\lambda}\sum_{i=1}^2(1+\|\varphi^i\|_{\infty}^\beta)\|\varphi^i\|_{\gamma}^\lambda\\
    \le C_1[W]_{\beta,\tau,\lambda}\sum_{i=1}^2\|\varphi^i\|_{\gamma}^\lambda\,.
  \end{align*}
  On the other hand
  \begin{align*}
    |\mu(a,b)&-\mu(a,c)-\mu(c,b)|\\
    &\le|W(b,\varphi^1_{a})-W(b,\varphi^2_{a})-W(c,\varphi^1_{a})+W(c,\varphi^2_{a})| \\
    &\qquad\quad +|W(b,\varphi^1_{c})-W(b,\varphi^2_{c})-W(c,\varphi^1_{c})+W(c,\varphi^2_{c})| \\
    & \le 2C_1[W]_{\beta,\tau,\lambda} |b-c|^{\tau}\|\varphi^1-\varphi^2\|_{\infty}^{\lambda}\,.
  \end{align*}
  Combining the two bounds for $\mu$ we get for any $\theta\in(0,1)$ such that $\tau+\theta \lambda \gamma>1$,
  \begin{equation*}
    [\mu]_{\tau+\theta \lambda \gamma}\le C_2 [W]_{\beta,\tau,\lambda}\|\varphi^1- \varphi^2\|_\infty^{\lambda(1-\theta)}\,.
  \end{equation*}
  This completes the proof.
\end{proof}
\begin{corollary}\label{cor.compact} Let $I$ be a nonempty closed, bounded and connected interval. Let $t_0$ be in $I$.
  Assuming condition \ref{cond.w} with $\tau+\lambda \gamma>1$. Then the map
  \begin{align*}
     M&:C^{\gamma}(I) \to C^{\tau}(I)\\
     &Mx(t)=\int_{t_0}^t W(\ds,x_s)
  \end{align*}
  is continuous and compact.
\end{corollary}
\begin{proof}
  Continuity follows immediately from Proposition \ref{prop.int.phi12}.
  For compactness, suppose
  $B$ is  a bounded subset of $C^{\gamma}(I)$. The estimate in Proposition \ref{prop.int.phi12} implies that ${\{Mx\}}_{x\in B}$ is bounded in $C^{\tau}(I)$.
  By the Arzel\`a-Ascoli theorem, the set ${\{Mx\}}_{x\in B}$ is relatively
  compact in $C^{\tau'}(I)$ for every $\tau'<\tau$.
  We show that $\{{Mx\}}_{x\in B}$ is indeed relatively compact in
  $C^{\tau}(I)$. More precisely, suppose $\left\{ Mx^{n}\right\} $
  is a convergent sequence in $M(B)$ in the norm of $C^{\tau'}(I)$,
  by taking further subsequence, we can assume that the sequence $\left\{ x^{n}\right\} $
  converges to $x$ in $C^{\gamma'}(I)$, for some $\gamma'<\gamma$ (this is possible
  since $B$ is bounded). It is sufficient to show that $Mx^{n}$ converges
  to $Mx$ in $C^{\tau}(I)$.  To prove this, we choose $\theta\in(0,1)$ and $\gamma'<\gamma$ such that $\tau+\theta \lambda \gamma'>1$, and then we apply Proposition \ref{prop.int.phi12} to obtain
    \begin{equation*}
  \|Mx-Mx^n\|_{\tau}
  \le c\|W\|_{\beta,\tau,\lambda} (\|x-x^n\|_{\infty}^{\lambda}+\|x-x^n\|_{\infty}^{\lambda(1-\theta)}) \,.
  \end{equation*}
  The constant $c$ depends only on $\|x\|_\infty,\|x\|_{\gamma'}$ and $\|x^n\|_\infty,\|x^n\|_{\gamma'}$ which is uniformly bounded with respect to $n$. This shows $Mx^{n}$ converges to $Mx$ in $C^{\tau}(I)$ and completes the proof.
\end{proof}

\section{Dif{}ferential equations}\label{sec.diffeqn}
Let   $W:\RR\times \RR^d\rightarrow \RR^d$
satisfy the condition
\ref{cond.w} stated at  the beginning of Section \ref{sec.pathint}  with $\tau(1+\lambda)>1$.
In this section we  consider the following  differential equation
\begin{equation}
\varphi_{t}=\varphi_{t_{0}}+\int_{t_{0}}^{t}W(\ds,\varphi_{s})\,. \label{eq:ode}
\end{equation}
We are concerned with the existence,  uniqueness, boundedness and
the flow property  of the solution. We shall also study the dependence
of the solution on the initial conditions.  {Some related results on this direction are also obtained independently by Catellier and Gubinelli \cite{catellier2012averaging}. } Applications of the results obtained
are represented in Subsections \ref{subsec.reg.flow} and \ref{subsec.transport} where we consider a transport equation of the type
\begin{align*}
    u(dt,x)=\nabla u(t,x)W(dt,x)\,.
\end{align*}
Literature on transport equations is vast and  mostly focuses
 on irregularity of the spatial variables of the vector field (see for instance \cite{diperna-lions} for Sobolev vector fields, \cite{ambrosio} for BV vector fields and \cite{bahouri} for Besov vector fields). In the case $W$ being a semi-martingale, the above equation is treated in \cite{kunita}. It appears to be new in the context of nonlinear Young integration.

\subsection{Existence and uniqueness}\label{subsec.2.1}
\begin{theorem}[Existence]
\label{thm:2e}Suppose that $W$ satisfies the assumption
\ref{cond.w}   with   $\tau(1+\lambda)>1$ and $\beta+\lambda\le1$.
Then the equation
(\ref{eq:ode}) has a solution in the space of H\"older continuous
functions $C^{\tau}\left([t_{0}-T,t_{0}+T]\right)$ for any $T>0$.
Moreover, if $\varphi$ is a solution in $C^{\tau}\left([t_{0}-T,t_{0}+T]\right)$,
then
\begin{equation}
\sup_{t\in[t_{0}-T,t_{0}+T]}|\varphi_{t}|+\sup_{t_0-T\le s<t\le t_0+T  }\frac{|\varphi_{t}-\varphi_s|}{|t-s|^{\tau}}
\le C
 _{\tau, \la,T} e^{  \kappa_{\tau, \la,T}  \|W\|_{\tau, \la} ^{\frac{1-\tau+\tau\la}{\tau \la}  } }
 \left(1\vee|\varphi_{t_{0}}|  \right) \,,   \label{eq:est.supphi}
\end{equation}
%and
%\begin{equation}
%\sup_{t_0-T\le s<t\le t_0+T  }\frac{|\varphi_{t}-\varphi_s|}{|t-s|^{\tau}}
%\le C
% _{\tau, \la} e^{  \kappa_{\tau, \la, \si}  \|W\|_{\tau, \la} ^{\frac{1-\tau+\tau\la}{\tau \la}  } }
% \left(|x_{t_{0}}| +1  \right)  \,, \label{e.4.3}
%\end{equation}
where the constant $k_{\tau, \la,T} $ and $C_{\tau, \la,T}$ depend   only on $\lambda$, $\tau$ and $T$.
\end{theorem}
\begin{proof}
Fix $T>0$, we denote $\|W\|= \|W\|_{\beta,\tau,\lambda;[t_0-T,t_0+T]}$. We define a  mapping $M$ acting  on
$C^{\tau}([t_{0}-T,t_{0}+T])$ as  follows
\[
Mx=x_{0}+\int_{t_{0}}^{\cdot}W(\ds,x_{s})\,, \quad \forall x \in  C^{\tau}([t_{0}-T,t_{0}+T]) \,.
\]
% From Theorem \ref{fractional} with $\gamma=\tau$, it follows that
% $\int_{t_{0}}^{\cdot}W(\ds,x_{s})$ is well-defined and is H\"older continuous of exponent $\tau$.
%   Thus $M$ maps $C^{\tau}([t_{0}-T,t_{0}+T])$
% into itself.
We shall verify that $M$ satisfies the hypothesis of Leray-Schauder theorem (see \cite[Theorem 11.3]{Gilbarg-Trudinger}).

\textbf{Step 1.} $M$ is well-defined, continuous and compact. This  immediately
 follows from    Corollary \ref{cor.compact}.

\textbf{Step 2.} Now we explain that the
set $\left\{ {x\in C^{\tau}([t_{0}-T,t_{0}+T]):x=\sigma Mx,0\le\sigma\le1}\right\} $
is bounded.   Let $x$ satisfy $x=\sigma Mx$
for some $\sigma\in[0,1]$. By definition of $M$, we see
$x=\sigma Mx$ can be written as
\[
x_{b}-x_{a} =\sigma \int_{a}^{b}W(\ds,x_{s})\,.
\]
From \eqref{est.W.c}, it follows that for any $a,b\in[t_{0}-T,t_{0}+T]$,
we have
\begin{align*}
|x_{b}-x_{a}|
&=\sigma\left|\int_{a}^{b}W(\ds,x_{s})\right|\\
&\le  \si    \|W\|(1+\|x\|_{\infty;a,b} ^\beta) \|x\|_{\infty;a,b}  ^{\lambda}(b-a)^\tau \\
&\quad+\sigma\kappa \|W\|(1+\|x\|_{\infty;a,b}^\beta)   \|x\|_{\tau; a,b}^{\lambda}  |b-a|^{\tau+\la\tau}\,.
\end{align*}
Since $\sigma\le1$, this yields
\begin{equation*}
  \|x\|_{\tau;a,b}\le\|W\|(1+\|x\|_{\infty;a,b} ^\beta) \|x\|_{\infty;a,b}  ^{\lambda}+\kappa \|W\|(1+\|x\|_{\infty;a,b}^\beta)   \|x\|_{\tau; a,b}^{\lambda}  |b-a|^{\la\tau}\,,
\end{equation*}
for every $a,b$ in $[t_0,t_0+T]$ with $a<b$. We emphasize that the constant $\kappa$ appears in the previous inequality is independent of  $\sigma$.
An application of Young inequality gives
\begin{equation*}
  \|x\|_{\infty;a,b} ^\beta\|x\|_{\tau;a,b}^{\lambda}\le\|x\|_{\infty;a,b} ^{\beta+\lambda}+\|x\|_{\tau;a,b}^{\beta+\lambda}\,.
\end{equation*}
Thus
\begin{align*}
  \|x\|_{\tau;a,b}&\le\|W\|(\|x\|_{\infty;a,b}^{\lambda}+\|x\|_{\infty;a,b} ^{\beta+\lambda}) +\kappa \|W\|\|x\|_{\infty;a,b}^{\beta+\lambda}|b-a|^{\la\tau}\\
  &\quad +\kappa\|W\|  ( \|x\|_{\tau; a,b}^{\lambda}+ \|x\|_{\tau; a,b}^{\beta+\lambda})|b-a|^{\lambda \tau}   \,.
\end{align*}
Applying the inequality $z^\theta\le 1\vee z$ ($\theta\in[0,1]$ and $z\ge0$),   we  obtain
\begin{align*}
    \|x\|_{\tau;a,b}&\le\|W\|(2+\kappa|b-a|^{\la\tau} )(1\vee\|x\|_{\infty;a,b})
   +\kappa\|W\|  ( 1\vee\|x\|_{\tau; a,b})|b-a|^{\lambda \tau}   \,.
\end{align*}
We further use
\begin{equation*}
  \|x\|_{\infty;a,b}\le|x_a|+\|x\|_{\tau;a,b}|b-a|^{\tau}
\end{equation*}
to obtain
\begin{equation}\label{est.xtau}
    \|x\|_{\tau;a,b}\le A\|W\|(1\vee|x_a|)+A\|W\|  ( 1\vee\|x\|_{\tau; a,b})|b-a|^{\lambda \tau}   \,,
\end{equation}
where $A$ is a constant depending only on $\tau,\lambda$ and $T$.
Let $\Delta$ be a positive number such that
\begin{equation}
A\|W\| \Delta^{\tau\lambda}=\frac{1}{2}.\label{eq:delta-1}
\end{equation}
% We can also assume that $\Delta\le1$ (by choosing $\kappa$ sufficiently large).
If $|b-a|\le\Delta$, then from \eqref{est.xtau}
\begin{equation}\label{est.xinf1}
    \|x\|_{\tau;a,b}\le2A\|W\|(1\vee|x_a|)\,.
\end{equation}
Hence, we obtain
\begin{equation}
(1\vee\|x\|_{\infty,a,b})\le( 2A\|W\| \Delta^{\tau}+1)(1\vee|x_{a}|) \,.\label{est.xinf2}
\end{equation}
Divide the interval $[t_{0},t_{0}+T]$ into $n=[T/\Delta]+1$ subintervals
of length less or equal than $\Delta$. Applying  the inequality (\ref{est.xinf2})
on the intervals $[t_{0},t_{0}+\Delta]$, $[t_{0}+\Delta]$,..., $[t_{0}+(n-1)\Delta,t_{0}+n\Delta\wedge T]$,
recursively, we obtain
\begin{equation}
(1\vee\|x\|_{\infty,t_{0},t_{0}+T})\le (2A\|W\|\Delta^\tau+1)^n(1\vee|x_a|)\,.
\label{est.supx1}
\end{equation}
We can also assume that $\De\le T$. Thus $n\le 2T/\De$. We use the bound $2A\|W\|\Delta^\tau+1\le \exp(2A\|W\|\Delta^\tau)$.  Then  \eref{est.supx1} yields
\begin{align*}
    (1\vee\|x\|_{\infty,t_{0},t_{0}+T})\le \exp(2A\|W\|\Delta^\tau\frac{2T}{\Delta})(1\vee|x_{t_0}|)\,.
\end{align*}
Using \eref{eq:delta-1}, namely,
\[
\De= (2A\|W\|) ^{-\frac1{\tau \la}}\,,
\]
we have
\begin{equation*}
(1\vee\|x\|_{\infty;t_{0},t_{0}+T})
 \le e^{    T(2A\|W\|) ^{\frac{1-\tau+\tau\la}{\tau \la}  } }
 \left(1\vee|x_{t_{0}}|  \right) \,,
\end{equation*}
where $C  _{\tau, \la  }$ and $ \kappa_{\tau, \la } $ are uniformly bounded
in $\si\in [0, 1]$.
The argument goes similarly on the other interval $[t_{0}-T,t_{0}]$.
Thus
\begin{equation}
(1\vee\|x\|_{\infty;t_{0}-T,t_{0}+T})
 \le e^{  T (2A\|W\|) ^{\frac{1-\tau+\tau\la}{\tau \la}  } }
 \left(1\vee|x_{t_{0}}| \right) \,.\label{e.4.16}
\end{equation}
Together with the estimate (\ref{est.xinf1}), this inequality \eref{e.4.16}  implies that the
set
\[
\left\{ {x\in C^{\tau}([t_{0}-T,t_{0}+T]):x=\sigma Lx,0\le\sigma\le1}\right\}
\]
 is bounded in $C^{\tau}([t_{0}-T,t_{0}+T])$.

\textbf{Step 3.} Applying  Leray-Schauder theorem, we see that the equation (\ref{eq:ode})
has a solution $\left\{\varphi_t\,, t\in [t_{0}-T,t_{0}+T]\right\}$
in $C^{\tau}([t_{0}-T,t_{0}+T])$ for every $T$. The
estimate (\ref{eq:est.supphi}) comes from (\ref{e.4.16})  together with
 \eref{est.xinf1}.
\end{proof}

Next, we study some stability result.  In particular, we want to know
how the solution depends on the initial condition $x_{t_0}$.

\begin{theorem}
\label{thm:unique} Let  the condition \ref{cond.w} be
 satisfied with $\tau+\tau \lambda>1$. In addition, we assume that   $W(t,x)$ is differentiable with respect to $x$
for every $t$ and  the spatial gradient matrix  of $W$ is denoted by  $\nabla W(t, x)=
\left(\frac{\partial W_i(t, x)}{\partial x_j}\right)_{1\le i,j\le d}$.
Suppose
\begin{eqnarray*}
&& \|\nabla W\|_{\tau,\la; [t_0-T,t_0+T]\times K}
 :=\sup_{\substack{
 t_0-T\le s<t\le t_0+T \\x  \in K} }
\frac{ |\nabla W(t,x)-\nabla W(s,x)|}{ |t-s|^{\tau}}
\\
&&\quad +\sup_{\substack{
t_0-T\le s<t\le t_0+T\\ x, y\in K\,, x\not=y}}
\frac{|\nabla W(t,x)-\nabla W(s,x)-\nabla W(t,y)+\nabla W(s,y)|}{ |t-s|^{\tau}|x-y|^{\lambda}}
\label{eq:cond.w1}
\end{eqnarray*}
is finite for all compact set $K$ in $\RR^d$. Let $x_{t}$ and $y_{t}$
be two solutions in $C^{\tau}([t_{0}-T,t_{0}+T])$ to the integral
equation (\ref{eq:ode}) with initial conditions $x_{0}$ and $y_{0}$
respectively. Then the following estimate holds
\begin{equation}
\sup_{t\in[t_{0}-T,t_{0}+T]}|x_{t}-y_{t}|\le
%\|z\|_{\infty,t_0,t_0+T} \le
2^{\kappa  TA^{\frac{1}{\tau}}} |x_{0}-y_{0}|,\label{eq:estxy}
\end{equation}
where $A$ is a constant depending on $\nabla W$, $x,y$ and $T$ (precise formula is given in \eqref{cons.A} below).
\end{theorem}
\begin{proof} We put $R=\max\{\|x\|_{\infty;[t_0-T,t_0+T]},\|y\|_{\infty;[t_0-T,t_0+T]}\}$, $K=\{x\in\RR^d:|x|\le R\}$ and $\|\nabla W\|  =\|\nabla W\|_{\tau, \lambda;  [t_0-T, t_0+T]\times K} $. We also denote $z_{t}=x_{t}-y_{t}$,
$\rho_{\tau}=(\|x\|_{\tau}+\|y\|_{\tau})^{\lambda}$
and $\eta_{t}=\eta x_{t}+(1-\eta)y_{t}$
for each $\eta\in(0,1)$. For every $s,t$ and $x$, we use the notation $W([s,t],x)=W(t,x)-W(s,x)$.

We shall obtain estimate for $z$ in $C([t_0-T,t_0+T])$. Fix $a<b$ in $[t_0-T,t_0+T]$.  We then write
\begin{equation*}
  z_b-z_a=\int_a^bW(\ds,x_s)-\int_a^bW(\ds,y_s)=\JJ_a^b \mu\,,
\end{equation*}
where $\mu$ is the function
\begin{equation*}
	\mu(s,t)= W([s,t],x_s)-W([s,t],y_s)=\int_0^1 \nabla W([s,t],\eta_s)z_s d \eta \,.
\end{equation*}
For every $s\le c\le t$ in $[a,b]$, we can write
\begin{multline*}
	\mu(s,t)-\mu(s,c)-\mu(c,t)
	\\= \int_0^1 \left(\left[\nabla W([c,t],\eta_s)-\nabla W([c,t],\eta_c) \right]z_s+\nabla W([c,t],\eta_c)(z_t-z_c)\right) d \eta\,.
\end{multline*}
We note that $|\eta_{t}-\eta_{s}|^{\lambda}=  |\eta(x_{t}-x_{s})+(1-\eta)(y_{t}-y_{s})|^{\lambda} \le \rho_{\tau}|u-v|^{\tau\lambda}$. It follows that
\begin{align*}
	[\mu]_{\tau(1+\lambda);[a,b]}\le \|\nabla W\|(\rho_ \tau\|z\|_{\infty;a,b}+|b-a|^{\tau(1- \lambda)}\|z\|_{\tau;a,b})\,.
\end{align*}
On the other hand, it is obvious
 that $|\mu(a,b)|\le\|\nabla W\||b-a|^\tau\|z\|_{\infty;a,b}$. Hence, the estimate \eqref{est.sew} implies
\begin{equation*}
	|z_b-z_a|\le  \|\nabla W\||b-a|^\tau\|z\|_{\infty;a,b}+ \kappa\|\nabla W\||b-a|^{\tau+\lambda \tau} (\rho_ \tau\|z\|_{\infty;a,b}+|b-a|^{\tau(1- \lambda)}\|z\|_{\tau;a,b})\,.
\end{equation*}
In other words,
\begin{equation*}
  \|z\|_{\tau; a,b}
  \le A\|z\|_{\infty;a,b}+A\|z\|_{\tau;a,b}(b-a)^{ \tau}\,,
\end{equation*}
where
\begin{equation}\label{cons.A}
  A=\kappa\|\nabla W\|[1+\rho_{\tau}T^{\lambda \tau}] \,.
\end{equation}
Therefore, using the bound $\|z\|_{\infty;a,b}\le|z_a|+\|z\|_{\tau;a,b}$ one gets
\begin{equation}\label{main-est}
  \|z\|_{\tau; a,b}
  \le A|z_a|+2A\|z\|_{\tau;a,b}(b-a)^{ \tau}\,.
\end{equation}
Now we shall use the above inequality to show our theorem.
Choose $a,b$ such that
\[
|b-a|\le\Delta=\left(\frac{1}{4A}\right)^{\frac{1}{\tau}}\,.
\]
Then inequality \eref{main-est} implies $\|z\|_{\tau,a,b}\le 2A|z_{a}|$
for all $a<b$.  By the definition of the H\"older norm, we see that  if
$|b-a|\le\Delta$,  then
\begin{align*}
\|z\|_{\infty,a,b} &\le  |z_{a}|+\|z\|_{\tau,a,b} (b-a)^\tau \\
 &\le  |z_{a}|+2A |z_{a}|\Delta^{\tau} \\
 &\le  2|z_{a}|.
\end{align*}
Divide the interval $[t_0,t_0+T]$ into $n=[T/\Delta]$+1 subintervals of
length less or equal than $\Delta$. Applying the previous inequality
on the intervals $[t_0,t_0+\Delta],[t_0+\Delta,t_0+2\Delta],\dots,[t_0+(n-1)\Delta,t_0+n\Delta\wedge T],$
recursively, we obtain
\[
\|z\|_{\infty,t_0,t_0+T}\le 2^{n}|z_{t_0}|\,.
\]
We can assume $\De\le T$. Thus
\[
n=[T/\Delta]+1\le \frac{2T}{\De}=2T \left(4A \right)^{\frac{1}{\tau}} \,.
\]
This implies
\[
\|z\|_{\infty,t_0,t_0+T} \le 2^{2^{1+2/\tau} TA^{\frac{1}{\tau}}}|z_{t_0}|\,.
\]
which yields the bound \eqref{eq:estxy} on the interval $[t_0,t_0+T]$.
Estimates on $[t_0-T,t_0]$ are analogous.
\end{proof}
An immediate consequence  of the theorem  is the following uniqueness result.
\begin{corollary}  Under the hypothesis of Theorem \ref{thm:unique} the equation
\eref{eq:ode} has a unique solution.
\end{corollary}
% \begin{remark}
%     In Theorem \ref{thm:unique}, it is possible to assume $\|\nabla W\|_{\tau_1,\lambda_1}$ is finite with constants $\tau_1,\lambda_1$ different from $\tau,\lambda$ respectively. One, however, has to impose additional conditions on $\tau_1,\lambda_1$ and $\tau,\lambda$ so that the integration by parts formula \eqref{eq:def.w} is applicable. We have chosen $\tau_1=\tau$ and $\lambda_1=\lambda$ for simplicity and transparency.
% \end{remark}
\subsection{Compositions}

Given a function $G:\RR^2\to\RR^d$, we may define the Riemann-Stieltjes integral $\int_a^b G(\ds,s)$ as the limit of Riemann sums
\begin{equation*}
    \sum_{i}G(t_i,t_{i-1})-G(t_{i-1},t_{i-1})\,.
\end{equation*}
The sewing lemma (Lemma \ref{sew.lem}) gives a sufficient condition so that the aforementioned limit exists, namely $G$ satisfies
\begin{equation*}
    |G(s,s)-G(s,t)-G(t,s)+G(t,t)|\lesssim |t-s|^{1+\ep}
\end{equation*}
for some $\ep>0$. In such case, Lemma \ref{lemma.arbi} also allows one to choose Riemann sums with right-end points. In other words, the Riemann sums with right-end points
\begin{equation*}
    \sum_{i}G(t_i,t_{i})-G(t_{i-1},t_{i})
\end{equation*}
also converges to the Riemann-Stieltjes integral $\int_a^b G(ds,s)$. In what follows, all integrals are understood as Riemann-Stieltjes integration, except for a few occasions, which we will indicate. The following result can be regarded as It\^o formula or chain rule for compositions of functions in the context of nonlinear Young integration.
\begin{theorem}
Let   $F $ be a function in $C^{(\tau_F,\lambda_F)}_\loc(\RR\times\RR^d)$ (i.e. $F$ satisfies the condition \ref{cond.wprime} with $\tau_F$ and $\la_F$), $g$ and $x$ be H\"older continuous functions with exponents $\tau_g$ and $\tau$ respectively. We suppose that $\tau_F+\lambda_F \tau>1$ and $\tau_g+\tau_F>1$. The following integration by parts formula holds
\begin{equation}\label{eqn.intbypart}
    \int_0^T g(t)d F(t,x_t)=\int_0^T g(t)F(dt,x_t)+\int_0^T g(t)F(t,dx_t)\,.
\end{equation}

In particular, suppose that $F$ belongs to $C^{\tau_F}_\loc(\RR;C^{1+\lambda_F}_\loc(\RR^d))$, $x$ is of the form $x_t=\int_a^t W(\ds, \phi_s)$, where
$W$ satisfy the condition \ref{cond.wprime} with $\tau$ and $\la$, $\phi$ satisfy \ref{cond.phi} with $\ga$, $\tau+\lambda \gamma>1$ and $\tau \lambda_F+\tau>1$.
Then \eqref{eqn.intbypart} becomes
\begin{equation}
\int_0^T g(t) dF(t, x_{t}) =\int_0^T g(t) F(\dt,  x_{t})
+\int_0^T g(t) (\nabla F)(t,x_{t})W(\dt,\phi_{t})\,. \label{mainito}
\end{equation}
An important consequence of \eqref{mainito} is when $g$ is  a constant function
\begin{equation}
F(b,x_{b})-F(a,x_{a})=\int_{a}^{b} F(\dt,  x_{t})
+\int_{a}^{b}(\nabla F)(t,x_{t})W(\dt,\phi_{t})\,. \label{ito}
\end{equation}
\end{theorem}
\begin{proof}
    We choose a compact set $K$ such that $K$ contains $\{x_t,0\le t\le T\}$ and denote $\|F\|=\|F\|_{\tau_F,\lambda_F;[0,T]\times K}$. We put
    \begin{align*}
        \mu(a,b)&=g(b)F(b,x_b)-g(b)F(a,x_b)\,,
        \\\nu(a,b)&=g(a)F(a,x_b)-g(a)F(a,x_a)\,,
        \\\vartheta(a,b)&=g(a)F(b,x_b)-g(a)F(a,x_a) \,.
    \end{align*}
    For every $a<c<b$, we have
    \begin{align*}
        &|\mu(a,b)-\mu(a,c)-\mu(c,a)|
        \\&=|-g(b)F(a,x_b)-g(c)F(c,x_c)+g(c)F(a,x_c)+g(b)F(c,x_b)|
        \\&\le|g(c)||F(a,x_b)-F(c,x_c)+F(a,x_c)+F(c,x_b)|
        \\&\quad+|g(c)-g(b)||F(c,x_c)-F(a,x_c)|
        \\&\le\|g\|_\infty\|F\|\|x\|_{\tau}^{\lambda_F} |b-a|^{\tau_F+\lambda_F \tau}+\|g\|_{\tau_g}\|F\||b-a|^{\tau_g+\tau_F}\,,
    \end{align*}
    and
    \begin{align*}
        &|\nu(a,b)-\nu(a,c)-\nu(c,a)|
        \\&=|g(a)F(a,x_b)-g(a)F(a,x_c)-g(c)F(c,x_b)+g(c)F(c,x_c)|
        \\&\le|g(c)||F(a,x_b)-F(a,x_c)-F(c,x_b)+F(c,x_c)|
        \\&\quad+|g(a)-g(c)||F(a,x_b)-F(a,x_c)|
        \\&\lesssim\|g\|_{\infty}\|F\||b-a|^{\tau_F+\lambda_F \tau}+\|g\|_{\tau_g}\|F\|\|x\|_{\tau}^{\lambda_F}|b-a|^{\tau_g+\lambda_F \tau}\,.
    \end{align*}
    Hence, from Lemmas \ref{sew.lem} and \ref{lemma.arbi}, $\JJ_0^T\mu=\int_0^T g(t)F(dt,x_t)$ and $\JJ_0^T \nu =g(t)F(t,dx_t)$. On the other hand,
    \begin{align*}
        &|\vartheta(a,b)-\mu(a,b)-\nu(a,b)|
        \\&=|[g(a)-g(b)][F(b,x_b)-F(a,x_b)]|
        \le\|g\|_{\tau_g}\|F\||b-a|^{\tau_g+\tau_F}\,.
    \end{align*}
    This together with Lemma \ref{lemma.arbi} implies \eqref{eqn.intbypart}.

    To prove \eqref{mainito}, it suffices to show
    \begin{equation}\label{eqn.FdxW}
        \int_0^T g(t)F(t,dx_t)=\int_0^Tg(t)(\nabla F)(t,x_t)W(dt,\phi_t)\,.
    \end{equation}
    We put
    \begin{equation*}
        \tilde \nu(a,b)=g(a)\nabla F(a,x_a)[W(b,\phi_a)-W(a,\phi_a)]\,.
    \end{equation*}
    Then we write
    \begin{align*}
        \nu(a,b)&=g(a)\int_0^1\nabla F(a,\eta x_a+(1- \eta)x_b)d \eta(x_a-x_b)
        \\&=g(a)\int_0^1\nabla F(a,\eta x_a+(1- \eta)x_b)d \eta\int_a^b W(ds,\phi_s)\,.
    \end{align*}
    Using the estimate \eqref{est.W.c}, we obtain
    \begin{align*}
        &|\nu(a,b)-\tilde \nu(a,b)|
        \\&\le| g(a)\int_0^1[\nabla F(a,\eta x_a+(1- \eta)x_b)-\nabla F(a,x_a)]d \eta\int_a^b W(ds,\phi_s)|
        \\&\quad+|g(a)\nabla F(a,x_a)[\int_a^bW(ds,\phi_s)-W(b,\phi_b)+W(a,\phi_a)]|
        \\&\lesssim |b-a|^{\lambda_F \tau+\tau}+|b-a|^{\tau+\lambda \gamma}\,.
    \end{align*}
    Identity \eqref{eqn.FdxW} follows from Lemma \ref{lemma.arbi} and the previous estimate.
\end{proof}

\subsection{Regularity of flow}\label{subsec.reg.flow}
In the rest of the current section, we assume the hypothesis of Theorem \ref{thm:unique}. This assumption guarantees that $\varphi(t,x)$, the solution to
\begin{equation*}
\varphi(t,x)=x+\int_0^tW(\ds,\varphi(s,x))
\end{equation*}
is unique. Moreover, by the result in Subsection \ref{subsec.2.1}, for fixed $t$, $\varphi(t,\cdot)$ is an automorphism on $\RR^d$, its inverse is $\varphi(t,\cdot)^{-1}=\varphi({-t},\cdot)$. Hence, the family $\{\varphi(t,\cdot):t\in\RR\}$ forms a flow of homeomorphism, i.e. it satisfies the following properties:
\begin{itemize}
  \item  $\varphi({t+s},\cdot)=\varphi(t,\varphi(s,\cdot))$ holds for all $s,t$,
  \item $\varphi(0,\cdot)$ is the identity map,
  \item the map $\varphi(t,\cdot):\RR^d\to\RR^d$ is a homeomorphism for all $t$.
\end{itemize}
Moreover, one can show that $\varphi(t,\cdot)$ is indeed a diffeomorphism.
\begin{theorem} \label{t.phi-flow}
Assume the hypothesis of Theorem \ref{thm:unique}. For any $t$ in $\RR$, the map $\varphi(t,\cdot)$ is a diffeomorphism. The following conclusions hold
\begin{enumerate}[label=(\roman*)]
  \item\label{flow.grad}   The gradient of $\varphi_t$ at $x$, denoted by $\nabla \varphi(t,x)=\{\partial_j \varphi^i(t,x) \}_{i,j}$ satisfies the equation
  \begin{equation}\label{eqn.flow.grad}
    \partial_i \varphi^\bullet(t,x)=\delta_{\bullet i} +\int_0^t \partial_k W^\bullet(\ds,\varphi(s,x))\partial_k \varphi^i(s,x)
 \end{equation} where $\delta_{ij}$ is the Kronecker symbol. Equation \eqref{eqn.flow.grad} can be written in short
 \begin{equation*}
    \nabla \varphi(t,x)=I_d+\int_0^t \nabla W(ds,\varphi(s,x))\nabla \varphi(s,x)\,.
 \end{equation*}
 \item \label{flow.inverse} For every $t$ and $x$, the matrix $\nabla \varphi(t,x)$ is invertible and its inverse $M(t,x)=[\nabla \varphi(t,x)]^{-1} $ satisfies the equation
 \begin{equation}\label{eqn.flow.inverse}
    M(t,x)^{j\bullet}=\delta_{j\bullet}-\int_0^tM(s,x)^{jk}\partial_\bullet W^k(ds,\varphi(s,x))
 \end{equation}
 or in short
 \begin{equation*}
    M(t,x)=I_d-\int_0^t M(s,x)\nabla W(ds,\varphi(s,x)\,.
 \end{equation*}
 \item\label{flow.jointholder} $\varphi$ is jointly H\"older continuous of order $(\tau,1)$. That is
 \begin{equation}\label{phi.jointholder}
    |\varphi(s,x)-\varphi(s,y)-\varphi(t,x)+\varphi(t,y)|\lesssim|t-s|^\tau|x-y|
 \end{equation}
\item\label{flow.det}   Let $J(t,x)$ denote the determinant of $\nabla \varphi(t,x)$. Then $J$ satisfies the following scalar linear equation \begin{equation}\label{eqn.det}
    J(t,x)=1+\int_0^t J(s,x)\Div(W(\ds,\varphi(s,x)))\,.
  \end{equation}
\item\label{flow.lagrange}  The flow $\varphi(t,x)$ is a Lagrangian flow,
  namely there exists a constant $L$ such that \begin{equation}\label{eqn.Lflow}
    \LL^d(\varphi(t,\cdot)^{-1}(A))\le L\LL^d (A) \quad \mbox{for every Borel set } A\subseteq \RR^d
  \end{equation} where $\LL^d$ is the Lebesgue measure on $\RR^d$.
\end{enumerate}
\end{theorem}
\begin{proof}
     Let $e$ be a unit vector in $\RR^d$. For each $h$ in $\RR$, we denote $$\eta^h_t=\frac1h(\varphi(t,x+he)-\varphi(t,x)).$$  To prove \ref{flow.grad}, it is sufficient to show that for every sequence $h_n$ converging  to 0, there is a subsequence $h_{n_k}$ such that $\eta^{h_{n_k}}$ converges to the solution of the following equation
     \begin{equation}\label{eqn.eta}
       \eta_t=e+\int_0^t \nabla W(\ds,\varphi(s,x))\eta_s.
     \end{equation}
     We remark that the equation \eqref{eqn.eta} is linear and the existence and uniqueness of solution in $C^\tau(\RR)$ follows from our method discussed in Subsection \ref{subsec.2.1}. From Theorem \ref{thm:unique} we see that \[\|\eta^h\|_{\tau;K}\le \kappa_K\] uniformly in $h$ for every compact interval $K$ in $\RR$. Hence, by  the
      Arzel\`a-Ascoli theorem, there is a subsequence, still denoted by $h_n$ such that $\eta^{h_n}$ converges to $\eta$ in $C^{\tau'}(K)$ for any arbitrary $\tau'<\tau$. On the other hand, we notice that $\eta^h$ satisfies
     \begin{equation}\label{eqb.etah}
       \eta^h_t=e+\int_0^1d\tau\int_0^t \nabla W(\ds,\tau \varphi(s,x+he)-(1-\tau)\varphi(s,x))\eta^h_s.
     \end{equation}
     Passing through the limit $h_n\to0$, we see that $\eta$ satisfies the equation \eqref{eqn.eta} and then \ref{flow.grad} follows. Assertion \ref{flow.jointholder} is a consequence of the estimate \eqref{eq:estxy} in Theorem \ref{thm:unique}. In fact,
     \begin{align*}
        |\varphi(s,x)-\varphi(s,y)-\varphi(t,x)+\varphi(t,y)|&\le\|\varphi(\cdot,x)-\varphi(\cdot,y)\|_{\tau;[s,t]}|t-s|^\tau
        \\&\lesssim|t-s|^\tau|x-y|\,.
     \end{align*}
     Assertion \ref{flow.det}  follows from the It\^o formula \eqref{ito}
     applied to   $J(t,x)=\det(\nabla \varphi(t,x))$ and the Jacobi's formula \[\d\det(M)=\det(M) tr(M^{-1}dM).\]
     To prove \ref{flow.lagrange}, we notice that the equation \eqref{eqn.det} can be solved explicitly thanks to \eqref{ito}
     \begin{equation}
       J(t,x)=\exp \int_0^t\Div (W(dt,\varphi(t,x))).
     \end{equation}  Therefore, from \eqref{est.W.c}, we obtain
     \begin{equation*}
       |J(t,x)^{-1}|\le e^{\kappa |t|^\tau}.
      \end{equation*} Together with the area formula
      \begin{equation*}
        \LL^d(\varphi(t,\cdot)^{-1}(A))=\int_{\varphi(-t,A)}dx=\int_{A}|\det (\nabla \varphi)(-t,x)|dx
      \end{equation*}
       this estimate implies \eqref{eqn.Lflow}.
\end{proof}

\subsection{Transport differential equation}\label{subsec.transport}

As an application of the above It\^o formula \eref{mainito} and flow property (Theorem \ref{t.phi-flow}),
we study the following
  transport differential equation in  {\it H\"older media}.  Specifically,
  let $W:\RR_+\times \RR^d\rightarrow \RR^d$ satisfy the conditions in Theorem \ref{thm:unique}. Consider the following first order
  partial differential equations (transport equation in H\"older media $W$)
\begin{equation}
\frac{\partial }{\partial t} u (t, x)+\left(\frac{\partial }{\partial t}W(t,x)\right) \cdot \nabla u(t, x) =0.\label{eq:transport}
\end{equation}
Here $\nabla$ is the gradient operator  (with respect to
spatial variables). Since $W$ is only H\"older continuous in time, the equation \eqref{eq:transport} is only formal. We can however define solutions in integral
 form. More precisely, a continuous function $u:\mathbb{R}_+\times\mathbb{R}^{d}\to\mathbb{R}$ is called a solution
to   (\ref{eq:transport}) with the initial condition $u(0, x)=h(x)$
 if it is differentiable with respect to $x\in \RR^d$ and the following equation holds.
\begin{equation}\label{eqn.transport}
u(t, x)=h(x)-\int_0^t   \nabla u(s, x)W(\ds,x) \quad \forall \ t\ge 0\,, \ x\in \RR^d\,.
\end{equation}
\begin{theorem} \label{thm.transport.existence}
    Assuming $W$ satisfies the conditions in Theorem \ref{thm:unique}. Let $h$ be a function in $C^{1+\lambda_0}_\loc(\RR^d)$ where $\lambda_0$ satisfies    $(1+\lambda_0)\tau>1$. Let $\varphi(t,x)$ be the unique solution to
\begin{equation*}
\varphi(t, x)=x+\int_0^t W(\ds, \varphi(s, x))\,, \ \forall t\ge 0 \,.
\end{equation*}
Let $\psi(t,x)$ be the inverse of $\varphi$ as a function $x\in \RR^d$ to $ \RR^d$.
Namely, $\varphi(t, \psi(t, x))=x$ for all $t\ge 0\,, \ x\in \RR^d$.
Then the function $u$ defined by
\begin{equation*}
u(t,x)=h(\psi(t, x))
\end{equation*}
is a solution to the above transport equation.
\end{theorem}
% \begin{proof}{
%   \replacecolorred wrong proof
%   }
%   The condition on $h$ entails that $u$ belongs to the space $C^{\tau,\lambda_0}_\loc$, we can therefore apply It\^o formula \eqref{ito} for the function $u(t,\varphi(t,x))$ to obtain
%   \begin{align*}
%       u(t,\varphi(t,x))-u(0,\varphi(0,x))=\int_0^t u(ds,\varphi(s,x))+\int_0^t\nabla u(s,\varphi(s,x))W(ds,\varphi(s,x))\,,
%   \end{align*}
%   for every $t\ge0$ and $x$ in $\RR^d$.
%   We notice that $u(t,\varphi(x))=h(x)$, the left hand side on the previous identity vanishes. On the other hand, replacing $x$ by $\psi(t,x)$, using the relation $\varphi(s,\psi(s,x))=x$, yields
%   \begin{align*}
%       0=\int_0^tu(ds,x)+\int_0^t\nabla u(s,x)W(ds,x)
%   \end{align*}
%   which shows the function $u(t,x)=h(\psi(t,x))$ satisfies \eqref{eqn.transport}.
% \end{proof}
\begin{proof} From Theorem \ref{t.phi-flow} such $\psi(t,x)$
exists and both $\varphi(t,x)$ and $\psi(t,x)$ are differentiable with respect to $x$.
Differentiate $\varphi(t, \psi(t,x))=x$ with respect to $x$ and    we see that
\[
(\nabla \varphi)(t, \psi(t,x))\nabla \psi(t, x)=I\,,
\]
or
\[
(\nabla \psi(t, x))^{-1}=(\nabla \varphi)(t, \psi(t,x))\,.
\]
Let $\rho(r)=\varphi(r, \psi(r,x))$, $0\le r< \infty$. Thanks to Theorem \ref{t.phi-flow}\ref{flow.jointholder}, It\^o formula
\eref{mainito}  is applicable. More precisely, for any $C^\tau$-function $g(r)$, we have
\begin{equation*}
\int_0^t g(r) d\rho(r)= \int_0^t g(r) \varphi(\dr, \psi(r, x))+\int_0^t g(r) (\nabla \varphi)(r, \psi(r, x)) \psi(\dr, x)\,.
\end{equation*}
Since $\rho(r)=x$, we have $d\rho(r)=0$.  Thus
\begin{equation}
\int_0^t g(r) (\nabla \varphi)(r, \psi(r, x)) \psi(\dr, x)= -\int_0^t g(r) \varphi(\dr, \psi(r, x))\,.\label{touse}
\end{equation}
Now the It\^o formula \eref{ito}  applied to $h(\psi(t,x))$ yields
\begin{align*}
u(t,x)&= h(\psi(t,x))
=h(x)+\int_0^t (\nabla h)(\psi(r,x))\psi(\dr, x)\\
&=h(x)+\int_0^t  \nabla\left[  h (\psi(r,x))
\right] \left(\nabla \psi(r, x)\right)^{-1} \psi(\dr, x)\\
&=h(x)+\int_0^t  \nabla u(r,x)   \left(\nabla \psi(r, x)\right)^{-1} \psi(\dr, x)\\
&=h(x)+\int_0^t  \nabla u(r,x)   \left(\nabla \varphi\right) (r, \psi(r,x)) \psi(\dr, x)\,.
\end{align*}
Using the  equation \eref{touse} for $g(r)=\nabla u(r,x) $,  we have
\begin{align*}
u(t,x)
&=h(x)-\int_0^t  \nabla u(r,x) \varphi(\dr, \psi(r, x))\\
&=h(x)-\int_0^t  \nabla u(r,x) W(\dr, \varphi(r, \psi(r, x)))\\
&=h(x)-\int_0^t  \nabla u(r,x) W(\dr, x)     \,.
\end{align*}
This completes the proof of the theorem.
\end{proof}
We also have the following uniqueness result.
\begin{theorem}Assuming $W$ satisfies the conditions in Theorem \ref{thm:unique}.
    Let $\lambda_0$ be in $(0,1]$ such that $(\lambda_0+1)\tau>1$.  Equation \eqref{eqn.transport} has unique solution in the class $C^{(\tau,\lambda_0)}_\loc(\RR\times\RR^d)$. More precisely, suppose $u$ belongs to $C^{(\tau,\lambda_0)}_\loc(\RR\times\RR^d)$ and satisfies \eqref{eqn.transport}, then $u$ is uniquely defined by the relation $u(t,x)=h(\psi(t,x))$, where $\varphi$ and $\psi$ are the functions defined in Theorem \ref{thm.transport.existence}.
\end{theorem}
\begin{proof}
    Let $u$ be a solution to \eqref{eqn.transport}.
    Applying It\^o formula \eqref{ito} for the function $u(t,\varphi(t,x))$ we have
    \begin{align*}
        u(t,\varphi(t,x))-h(x)=\int_0^t u(ds,\varphi(s,x))+\int_0^t\nabla u(s,\varphi(s,x))W(ds,\varphi(s,x))\,.
    \end{align*}
    It suffices to show the right hand side vanishes. In other words the following relation between the two nonlinear Young integrals holds
    \begin{equation}\label{eqn.pr.uniq}
        \int_0^t u(ds,\varphi(s,x))=-\int_0^t\nabla u(s,\varphi(s,x))W(ds,\varphi(s,x))\,.
    \end{equation}
    For clarity, we will omit $x$ in the notations. We put
    \begin{align*}
        \mu_1(a,b)&=u(b,\varphi_a)-u(a,\varphi_a)\,,
        \\ \mu_2(a,b)&=\nabla u(a,\varphi_a)[W(b,\varphi_a)-W(a,\varphi_a)]\,.
    \end{align*}
    Since $u$ satisfies the equation \eqref{eqn.transport}, we can write
    \begin{align*}
        \mu_1(a,b)=-\int_a^b \nabla u(s,\varphi_a)W(ds,\varphi_a)\,.
    \end{align*}
    Thus
    \begin{multline*}
        \mu_1(a,b)+\mu_2(a,b)
        \\=-\int_a^b \nabla u(s,\varphi_a)W(ds,\varphi_a)+\nabla u(a,\varphi_a)[W(b,\varphi_a)-W(a,\varphi_a)]\,.
    \end{multline*}
    The estimate \eqref{est.sew} (or \eqref{est.W.c}) implies
    \begin{equation*}
        |\mu_1(a,b)+\mu_2(a,b)|\lesssim |b-a|^{2 \tau}\,.
    \end{equation*}
    Since $2 \tau>1$, Lemma \ref{lemma.arbi} yields $\JJ_0^t \mu_1=-\JJ_0^t \mu_2$. This completes the proof after observing that the aforementioned identity is exactly the same as \eqref{eqn.pr.uniq}.
\end{proof}
\begin{remark}
    In the context of ordinary differential equation of the type
    \[\frac{dX}{dt}(t,x)=b(t,X(t,x))\,, \]
    with non-regular vector field $b$, existence and uniqueness and stability of regular Lagrangian flows were proved by R.J. DiPerna and P.-L. Lions (\cite{diperna-lions}) for Sobolev vector fields with bounded divergence. This result has been extended by L. Ambrosio (\cite{ambrosio}) to BV coefficients with bounded divergence.  In \cite{crippa}, it is shown that under slightly relaxed assumptions many of the ODE results of DiPerna-Lions theory can be recovered, from a priory estimates, similar to \eqref{eqn.Lflow}. The current paper proposes another extension of this theory, where the vector field is distribution (rough) in time (derivative of a H\"older continuous function) and smooth in space. It is also
interesting to extend the results presented here for vector fields which are rougher in time
(see e.g. \cite{hunualart09} for the linear case) or which are both rough in time and   in space.
\end{remark}

\setcounter{equation}{0}
\section{Feynman-Kac formula - A pathwise approach}\label{sec.feykac}
\def\Cgr{C^{0,1+\alpha}_\beta([0,T]\times\RR^d)}
In this section we shall study the stochastic parabolic equation
  with H\"older continuous noise in a H\"older random media (see
equation \eref{eqn.Ldw} below).  A feature of this problem is that
for the noise we
 don't assume any H\"older continuity in time
 variable.  To make up for lack of regularity in time, we assume some regularity on spatial variables. In this case, the method presented in this section works for each sample path of the noise.

{\replacecolorred  Throughout the current section, $T$ is a fixed positive time.} To describe  the noise,   we introduce the following space.
% $C^{0,1+\alpha}_\beta([0,T]\times\RR^d) $.
Let $\beta$ be a fixed non-negative number.
We say that  $f$ is in $\Cgr$ if it belongs to $C([0,T],C^{1+\alpha}_\loc(\RR^d))$ and
satisfies the following   condition
\begin{equation}\label{growth.daf}
    [\nabla f]_{\beta,\alpha} :=\sup_{\substack {t\in[0,T];\\ x,y\in\RR^d;x\neq y}}\frac{|\nabla
    f(t,x)-\nabla f(t,y) |} {|x-y|^\alpha (1+|x|^{\beta}
    +|y|^{\beta}) }<\infty\,.
\end{equation}
We notice that the condition \eqref{growth.daf} implies the growth   conditions on $\nabla f$ and $f$. More precisely, one has
\begin{equation}\label{growth.df}
    [\nabla f]_{\alpha+\beta,\infty}:= \sup_{t\in[0,T],x\in\RR^d}\frac{|\nabla f(t,x)|}{1+|x|^{\alpha+\beta}} <\infty\,,
\end{equation}
and
\begin{equation}\label{growth.f}
 [f]_{\alpha+\beta+1,\infty} := \sup_{t\in[0,T],x\in\RR^d}\frac{|f(t,x)|}{1+|x|^{\alpha+\beta+1}} <\infty   \,.
\end{equation}

% \begin{equation}\label{growth.w}
%  \kappa_0(f):= \sup_{t\in[0,T],x\in\RR^d}\frac{|f(t,x)|}{1+|x|^{\beta_0}} <\infty   \,,
% \end{equation}
% \begin{equation}\label{growth.dw}
% \kappa_1(f):=  \sup_{t\in[0,T],x\in\RR^d}\frac{|\nabla
% f(t,x)|}{1+|x|^{\beta_1}} <\infty  \,,
% \end{equation}
% and
% \begin{equation}\label{growth.daw}
% \kappa_2(f) := \sup_{t\in[0,T];x,y\in\RR^d}\frac{|\nabla
% f(t,x)-\nabla f(t,y) |} {|x-y|^\alpha (1+|x|^{\beta_2}
% +|y|^{\beta_2}) }<\infty\,.
% \end{equation}
It is easy to see that $\|f\|_{C^{0,1+\alpha}_\beta}:=[f]_{\alpha+\beta+1,\infty}+[\nabla f]_{\alpha+\beta,\infty}+[\nabla f]_{\beta,\alpha}
$ forms a norm on $\Cgr$. In the rest of this section, we denote
\begin{equation*}
    C^{0,1+\alpha^-}_{\beta}=\bigcap_{0<\alpha'<\alpha}C^{0,1+\alpha'}_{\beta}([0,T]\times\RR^d)\,.
\end{equation*}

Similar to the classical H\"older spaces, the space of smooth functions  is not dense  in $\Cgr$. However, we can still approximate a function in $\Cgr$ by smooth functions with a little trade off in spatial regularity. More precisely, let $\eta$ be function in $C^\infty_c(\RR^{d+1})$ supported in $(-1,1)^{d+1}$ and $\iint \eta(t,x)  dtdx=1$. For $\epsilon>0$, we put $\eta_{\epsilon}(t,x)=\epsilon^{-d-1}\eta(\epsilon^{-1}(t,x))$. Let $f$ be in $\Cgr$, we define $f_{\epsilon}(t,x) =(f * \eta_{\epsilon})(t,x)$.
% \begin{equation*}
%     f_{\epsilon}(t,x) =(f * \eta_{\epsilon})(t,x)\,.
% \end{equation*}
It is clear that $f_{\epsilon}$ belongs to $C_c^\infty(\RR^{d+1})$. In addition, we have the following result.
\begin{lemma}\label{lem.fepf}
    % $f_{\epsilon}$ converges to $f$ pointwise and
    For every $\alpha'<\alpha$, $[\nabla f_{\epsilon}-\nabla f ]_{\beta,\infty}$ and $[\nabla f_{\epsilon}-\nabla f ]_{\beta,\alpha'}$ converge to 0 as $\epsilon$ goes to 0.
    % \begin{align*}
    %     [\nabla f_{\epsilon}-\nabla f ]_{\beta,\infty}\to0\,,\\
    %     [\nabla f_{\epsilon}-\nabla f ]_{\beta,\alpha'}\to0\,,
    % \end{align*}
    % as $\epsilon$ goes to 0.
\end{lemma}
\begin{proof}
	We have
	\begin{align*}
	    |\nabla f_{\epsilon}(t,x)-\nabla f(t,x)|&\le \iint |\nabla f(t,z)-\nabla f(t,x)|\eta_{\epsilon}(t,x-z)dtdz\\
	    &\le [\nabla f]_{\beta,\alpha}\iint |x-z|^\alpha(1+|x|^\beta+|z|^\beta) \eta_{\epsilon}(t,x-z)dtdz\\
	    &\lesssim [\nabla f]_{\beta,\alpha}\epsilon^\alpha(1+|x|^\beta)\,,
	\end{align*}
	which implied $[\nabla f_{\epsilon}-\nabla f]_{\beta,\infty}\to0$. This also implies
	\begin{equation*}
	    |\nabla f_{\epsilon}(t,x)-\nabla f_{\epsilon}(t,y) -\nabla f(t,x)+\nabla f(t,y)|
	    % \\\le |\nabla f_{\epsilon}(t,x)-\nabla f(t,x)|+|\nabla f(t,y)-\nabla f_{\epsilon}(t,y)|
	    \lesssim [\nabla f]_{\beta,\alpha}\epsilon^\alpha(1+|x|^\beta+|y|^\beta)\,.
	\end{equation*}
	On the other hand
	\begin{align*}
	    |\nabla f_{\epsilon}(t,x)-\nabla f_{\epsilon} (t,y)|&\le \iint |\nabla f(t,x-z)-\nabla f(t,y-z)|\eta_{\epsilon}(t,z)dtdz\\
	    &\le [\nabla f]_{\beta,\alpha}|x-y|^\alpha\iint (1+|x-z|^\beta+|y-z|^\beta) \eta_{\epsilon}(t,z)dtdz\\
	    &\lesssim [\nabla f]_{\beta,\alpha}|x-y|^\alpha(1+|x|^\beta+|y|^\beta)\,,
	\end{align*}
	thus
	\begin{equation*}
	    |\nabla f_{\epsilon}(t,x)-\nabla f_{\epsilon}(t,y) -\nabla f(t,x)+\nabla f(t,y)|\lesssim [\nabla f]_{\beta,\alpha}|x-y|^\alpha(1+|x|^\beta+|y|^\beta)\,.
	\end{equation*}
	Interpolating these two bounds, we get
	\begin{multline*}
	    |\nabla f_{\epsilon}(t,x)-\nabla f_{\epsilon}(t,y) -\nabla f(t,x)+\nabla f(t,y)|
	    \\\lesssim [\nabla f]_{\beta,\alpha}\epsilon^{\alpha- \alpha'}|x-y|^{\alpha'} (1+|x|^\beta+|y|^\beta)
	\end{multline*}
	for every $\alpha'<\alpha$. This implies $[\nabla f_{\epsilon}-\nabla f]_{\beta,\alpha'}\to0$.
\end{proof}

% It is
% clear that every function in $\Cgr$ can be approximated  in the norm
% $\|\cdot\|_{C^{0,1+\alpha}_\beta}$  by smooth functions
%  with compact support (namely, the elements in $C_0^\infty([0,T]\times\RR^d)$).

In Section \ref{sec.cov-path} we shall give conditions on the covariance of a
 Gaussian field $W(t,x)$ such that it is in $\Cgr$.

 %$C([0,T],C^{1+\alpha}(\RR^d))$ for some $\alpha\in(0,1)$.
%Since $W$ represents  a stochastic random field, naturally $W$ and its spatial
%derivative $\nabla W$ can have polynomial growth. To be more precise,  Throughout this section,
%we assume that $W$ belongs to $\Cgr$.

Assume that $W$ belongs to the space $\Cgr$, throughout this section, we denote $W_{n}=W*\eta_{1/n}$. We consider the
following parabolic equation with multiplicative noise:
\begin{equation}\label{eqn.Ldw}
  \partial_t u+Lu +u \partial_t W=0\,,
  \qquad{} u(T,x)=u_T(x)\,,
  \end{equation}
where the terminal function $u_T$ is assumed to be measurable with
polynomial growth and
 $L$ is a second order differential operator of the form
\begin{equation}
  L=\frac12\sum_{i,j=1}^da^{ij}(t,x)\partial_{x_i}\partial_{x_j}+\sum_{i=1}^d b^i(t,x)\partial_{x_i}\,.
  \label{eqn:L}
\end{equation}
 Here the novelty   is that we allow the coefficients $a^{ij}(t,x)=a^{ij}(t,x, W)
$ and $b^i(t,x)=b^i(t,x, W)$  depend  on $W$.  Since we are going to
solve the equation and   to establish a Feynman-Kac type formula
pointwise for   $W$,  we omit the explicit dependence of $a^{ij}$ and
$b^i$ on $W$. Notice that with a time reversal $t\rightarrow T-t$,
we can solve the stochastic parabolic equation with initial
condition:
\begin{equation*}
  \partial_t u=Lu -u \partial_t W \,,
  \qquad{} u(0,x)=u_0(x)\,.
  \end{equation*}

The stochastic differential equations with random coefficients have
been studied in a large amount of  papers. For example,  it has been
used in the modeling of the pressure in an oil reservoir with a log
normal random permeability in
  \cite{holdenhu}  (see in particular the references therein).
Recently, there have been great amount of research work on
uncertainty quantization from the numerical computation
community.   Many different types of stochastic partial
differential  equations with random coefficients  have been studied.
%{\replacecolorred (rephrase "also proposed to
%study"!) }.
Let us only mention the books \cite{grigoriu}, \cite{xin},
and the references therein.     {\replacecolorred  Since the classical Feynman-Kac formula  has already experienced many applications including the so-called
Monte-Carlos
particle approximation (see \cite{delmoralbook, delmoralbook2})},
we expect that the Feynman-Kac formula we obtained
will be a significant addition to this literature
 in particular
in the use of Monte-Carlo method  for the computations.

 We assume the following conditions on the operator $L$ appearing in
 the equation  \eref{eqn.Ldw}.
\begin{enumerate}[label=\textbf{(L\arabic*)}]
  \item\label{cond.L.elliptic}   $L$ is uniformly  elliptic, that is there exist positive numbers $\lambda$ and $\Lambda$
  %(independent of $W$)
  such that
  \begin{equation*}
    \lambda |\xi|^2\le \sum_{i,j=1}^da^{ij}(t,x)\xi^i \xi^j\le \Lambda |\xi|^2\,,  \quad \forall \ \xi
        \in \RR^d\,.
  \end{equation*}
  % \item\label{cond.l2} The coefficients $a,b$ are measurable functions of $(t,x)\in\RR^{1+d}$ and for some constant $K\in[0,\infty)$, $|b|\le K$ where $b=(b_1,\dots,b_d)$.
  \item\label{cond.L.areg}   For every $t$, the coefficients $a(t,\cdot)$ belong to $C^{2+\alpha}_b(\RR^d) $ with bounded derivatives uniformly in $t$.  That is
  \begin{equation*}
    \sup_{t}\|a(t,\cdot)\|_{C^{2+\alpha}_b(\RR^d)}\le \Lambda\,.
  \end{equation*}
  \item\label{cond.Lbreg} $b$ is Lipschitz continuous and has linear growth, that is,  there exists a positive constant $\kappa(b)$ such that
  \begin{eqnarray*}
 &&   \sup_{t}|b^i(t,x)|\le \kappa(b)(1+|x|)\,,\quad \forall \xi\in\RR^d\,,\\
 &&   \sup_{t} |b^i(t, y)-b^i(t, x)|  \le   \kappa(b) |y-x |\,,
  \quad \forall \ \ x, \ y\in \RR^d\,.
  \end{eqnarray*}

  % \item\label{cond.l3} There exists an increasing function $\omega(\epsilon)$, $\epsilon\ge0$ such that $\omega(\epsilon)\downarrow 0$ as $\epsilon\downarrow0$ and for all $t\in\RR$, $x,y\in\RR^d$ and $i,j=1,\dots,d$ we have
  % \begin{equation*}
  %   |a^{ij}(t,x)-a^{i,j}(t,y)|\le \omega(|x-y|)
  % \end{equation*}
\end{enumerate}

Under our conditions on $W$, it turns out that we can define the Feynman-Kac solution to equation
\eqref{eqn.Ldw}, namely,
\[u(r,x)=\EE^B \left[u_T(X_T^{r,x}) \exp\left\{\int_r^TW(ds,X_s^{r,x})\right\} \right]
\,,
\]
where $\{X_s^{r,x},s\ge r\}$ is the diffusion process generated by $L$ starting from $x$ at time $r$. More precisely, for every $r\le t\le T$ and $x\in\RR^d$, let $X_t^{r,x}$ be  the
diffusion process given by the stochastic differential equation
\begin{equation}
  \d X_t^{i,r,x}=\sigma^{ij}(t,X_t^{r,x})\de B_t^j+b^i(t,X_t^{r,x})\dt\,,\quad
  X_r^{r,x}=x\,,
  \label{eqn.diffusion}
\end{equation}
where $\sigma$ is the square root matrix of $a$, namely,
$a^{ij}=\sum_{k=1}^d \sigma^{ik}\si^{jk}$ and $\de B_t$ denotes
the It\^o differential. We will occasionally omit the index ${r,x}$ and write $X_s$ for $X_s^{r,x}$. Under conditions \ref{cond.L.elliptic}-\ref{cond.Lbreg}, it is well-known that the diffusion process $X_t^{r,x}$ exists and has finite moments of all orders.

Equation \eqref{eqn.Ldw} with $W$ replaced by $W_n$ is classic and one can obtain a smooth solution $u_n$ (see for instance \cite{krylov-priola10} where a more
general situation is studied). The main result of the current section is to show that $u_n$ converges to the Feynman-Kac solution $u$ defined above. There are three main tasks to be accomplished:
\begin{enumerate}[label=(\roman*)]
    \item One needs to define the nonlinear integration $\int W(ds,X_s)$. Since here $W$ is only continuous in time, this integration is different from the Young integration considered in Section \ref{sec.pathint}.
    \item One needs to show exponential integrability of $\int W(ds,X_s)$. In particular, the function $u$ defined by Feynman-Kac formula is well-defined.
    \item One needs to show that the exponential functional of this integration is stable under approximations by smooth functions.
\end{enumerate}
The outline of this section is as follows.  In subsection \ref{subsec.intwx}, we
define the nonlinear stochastic integration $\int W(\ds,
X_s)$ and show that it has finite moment of all orders. Exponential integrability is obtained if $W$ has strictly sub-quadratic growth, namely,
if    $\al$ and $\beta$ in
\eref{growth.daf}-\eref{growth.f} satisfy  $\beta+\alpha<1$.  In subsection \ref{fey.fey}, we show that the Feynman-Kac solution is indeed a solution in certain sense.
When $W$ has more regularity in time such as in the case of
Brownian sheets or fractional Brownian sheets, one can use this regularity to
reduce the regularity requirement in space.
This  case is considered in subsection \ref{sec.feykacii} when $W$ satisfies the conditions in Section \ref{sec.pathint}. {\replacecolorred  Along the way, we will make use of some fundamental estimates for exponential moment of various norms of the diffusion $X$ on finite intervals. These estimates are stated and proved in Appendix \ref{app.est.diff}.}

%  In remaining part of this section, we
% will show that Equation \eqref{eqn.Ldw} has a unique solution in the
% sense of Definition \ref{def.uew}. We shall pursue  the following
% idea. First,  we   approximate $W$ by the sequence of smooth functions
% $W_n$ and we solve  the equation \eqref{eqn.Ldw}  with $W$ replaced
% by $W_n$  to obtain a smooth solution $u_n$. Then  it is shown that
% the sequence $u_n$ converges to a limit $u$ which satisfies the
% relation \eqref{eqn.uewx}. In subsection \ref{subsec.intwx}, we
% define a non-linear stochastic integration of the type $\int W(\ds,
% X_s)$, where $X_s$ is a diffusion process. This type of integration
% appears in Feynman-Kac formula associated to the equation
% \eqref{eqn.Ldw} and is slightly different than that we studied in
% Section \ref{sec.pathint}. The use of It\^o formula in this subsection is inspired
% from the work  \cite{flandoli10}. In that work, an It\^o-Tanaka trick is applied to obtain some estimates to the commutator related to DiPerna-Lions' theory (\cite{diperna-lions}). We will see that a similar trick is applicable in our situation.  In subsection \ref{fey.fey}, we
% give the  Feynman-Kac formula for strong solution of
% \eqref{eqn.Ldw}.

In what follows, $\EE$ denotes the expectation with respect to a Brownian motion $B$, $\|\cdot\|_p$ denotes the $L^p$ norm corresponding to $\EE$.
\subsection{Nonlinear Stochastic integral}\label{subsec.intwx}

%Let $(B_t,t\ge0)$ be a Brownian motion in $\RR^d$ with respect to a filtration $\{\cF_t,t\ge0\}$. Fix $r\le T$, let $(X^{r,x}_t)_{r\le t\le T }$ be the Markov process such that
%\begin{equation}\label{eqn.Xx}
%  \d X^{r,x}_t=\sigma(t,X_t^{r,x})\d B_t+ b(t,X^{r,x}_t)\dt   \,,\quad{} X_r^{r,x}=x\,.
%\end{equation}
% For each $t\in[0,T]$, $\theta^t$ denotes the time change operator $\theta^t W(s,x)=W(t-s,x)$.
Let $X_t^{r,x}$ satisfy \eref{eqn.diffusion} and let $W$ be in $C_{\beta}^{0,
1+\al}([0, T]\times \RR^d)$.   We shall define  a new  nonlinear
integration $\int_r^T W(\ds,X_s^{r,x})$.  If  $W$ is differentiable
in time, the natural definition for this type of integration is
$\int_r^T
\partial_tW(s,X_s^{r,x} )\ds$.  If $W$ satisfies \ref{cond.w}
% \eref{eq:cond.w1-t}-\eref{eq:cond.w2},
then we can define it as in Section \ref{sec.pathint}. However, in this section, H\"older continuity of $W$ on $t$ is not required.   On the other
hand,  we shall use the crucial fact that $\left\{X_t^{r,x}\,, t\ge
r\right\} $ is a semimartingale. We first give the following
definition.
\begin{definition}\label{def.intw}
  Let $W_n$ be a sequence of smooth functions with compact support  converging to $W$ in $\Cgr$. We define
  \begin{equation}\label{id.defintW}
    \int_r^T  W(\ds,X_s^{r,x})=\lim_{n}\int_r^T \partial_sW_n(s,X_s^{r,x})\ds
  \end{equation}
  if the above limit exists in probability.
\end{definition}

Of course, at the first glance, there is no reason for the  limit in
\eqref{id.defintW} to converge. We will show, however, that the
above definition is well-defined, thanks to smoothing effect of the
diffusion process $X_s^{r,x}$.
Our first task is to find an appropriate representation for the
integration $\int_r^T \partial_tW_n(s,X_s^{r,x})ds$. To accomplish
this, we consider the partial differential equation
\begin{equation*}
  (\partial_t+L_0)v_n(r,x)=-\partial_tW_n(r,x)\,,\quad
  v(T,x)=-W_n(T,x)\,,
\end{equation*}
where we recall that $L$ is defined  by  \eref{eqn:L}
 and $$ L_0=L-b\nabla=\frac12\sum_{i,j=1}^da^{ij}(t,x)\partial_{x_i}\partial_{x_j}\,.$$
We could have chosen $L_0=L$ but the above choice of $L_0$ will allow  us to show exponential
integrability later.  Since $W_n$ is a smooth function, the
solution $v_n$ is a strong solution which is at least twice
differentiable in space and once differentiable in time. We then
apply It\^o formula to obtain
\begin{equation*}
  \begin{split}
    \d v_n(s,X_s^{r,x})&=(\partial_t+L)v_n(s,X_s^{r,x})\ds+\sigma^{ij}(s,X_s^{r,x})\partial_{x_i}v_n(s,X_s^{r,x})\de B_s^j\\
  &=-\partial_tW_n(s,X_s^{r,x})\ds-b(s,X_s^{r,x})\cdot\nabla v_n(s,X_s^{r,x})\ds\\
  &\quad+ \sigma^{ij}(s,X_s^{r,x})\partial_{x_i}v_n(s,X_s^{r,x})\de B_s^j\,.
  \end{split}
\end{equation*}
Thus, it follows that
\begin{equation}\label{itotrick}
  \begin{split}
    \int_r^T\partial_t&W_n(s,X_s^{r,x})\ds\\
    &=W_n(T,X_T^{r,x})+v_n(r,x)-\int_r^Tb(s,X_s^{r,x})\cdot\nabla v_n(s,X_s^{r,x})\ds\\
  &\quad+\int_r^T \sigma^{ij}(s,X_s^{r,x})\partial_{x_i}v_n(s,X_s^{r,x})\de
  B_s^j\,.
  \end{split}
\end{equation}
Notice that the time derivative in $W_n$ is transferred to the
spatial derivative in $v_n$. The next task is to show that $v_n$ and
its derivative $\nabla v_n$ converge. This is accomplished by some
estimates which are in the same spirit of the well-known Schauder
estimates for parabolic equations in H\"older spaces. More
precisely, we have
\begin{lemma}\label{lem.estv}
  Suppose that $W$ belongs to $C^2_\loc(\RR^{d+1})$ and satisfies
  \begin{equation*}
  	[W]_{\beta_1,\infty}:=\sup_{0\le t\le T} \sup_{x\in \RR^d}
  	\frac{|\nabla W(t,x)|}{1+|x|^{\beta_1}}<\infty
  \end{equation*}
  and
  \begin{equation*}
  	[W]_{\beta_2,\alpha} := \sup_{0\le t\le T} \sup_{x\not=y  }
  \frac{|\nabla W(t,x)-\nabla W(t,y)|}{|x-y|^{\al} (1+|x|^{\beta_2}
  +|y|^{\beta_2})}
   <\infty
  \end{equation*}
%  \begin{equation*}
%    [\nabla W]_{\beta_1,\infty}+[\nabla W]_{\beta_2,\alpha} <\infty\,,
%  \end{equation*}
  for some non-negative numbers $\beta_1,\beta_2$.
  Let $v$ be a strong solution with polynomial growth to the partial differential equation
  \begin{equation}
    (\partial_t+L_0)v=-\partial_t W\,,\quad v(T,x)=-W(T,x)\,.
    \label{e.v.definition}
  \end{equation}

  Let $t\mapsto\varphi_t$ be the diffusion process generated by $L_0$, that is
  \begin{equation}\label{diff.phi}
    \varphi_t^{r,x}=x+\int_r^t \sigma(s,\varphi_s^{r,x})\delta B_s\,,\quad t\ge r\,.
  \end{equation}
  Then $v$ is uniquely defined and verifies
  \begin{equation}\label{eqn.vw}
    (v+W)(r,x)=-\EE \int_r^T L_0W(s,\varphi_s^{r,x})\ds\,.
  \end{equation}
  In addition, the following estimates hold
  \begin{equation}\label{est.vw}
    \sup_{x\in\RR^d}\frac{|(v+W)(r,x)|}{1+|x|^{\beta_1}}\le c(\beta_1,\lambda,\Lambda)[(T-r)^{1/2}+(T-r)] [\nabla W]_{\beta_1,\infty} \,,
  \end{equation}
  \begin{equation}\label{est.dvw}
    \sup_{x\in\RR^d}\frac{|\nabla(v+W)(r,x)| }{1+|x|^{\beta_2}}
    \le c(\alpha,\beta_2,\lambda,\Lambda)[(T-r)^{\alpha/2}+(T-r)^{\alpha/2+1/2}] [\nabla W]_{\beta_2,\alpha} \,,
  \end{equation}
  and for every $\alpha'\in(0,\alpha)$,
  \begin{multline}\label{est.dvwa}
    \sup_{x\in\RR^d}\frac{|\nabla(v+W)(r,x)-\nabla(v+W)(r,y)| }{(1+|x|^{\beta_2}+|y|^{\beta_2} ) |x-y|^{\alpha'}}\\
    \le c(\alpha',\alpha,\beta_2,\lambda,\Lambda)[(T-r)^{(\alpha- \alpha')/2}+(T-r)^{(\alpha- \alpha')/2+1/2}] [\nabla W]_{\beta_2,\alpha} \,.
  \end{multline}
\end{lemma}
The proof of this result, even though lengthy, is straight forward and is provided in details in Appendix \ref{sec:schauder_estimates}.
\begin{proposition}\label{prop.v}
  Suppose that $W$ belongs $\Cgr$. Then there exists a $C^1$-generalized solution $v$
  to the parabolic partial differential equation
  \begin{equation}\label{eqn.v}
    (\partial_t+L_0)v=-\partial_tW\,,\quad v(T,x)=-W(T,x)\,,
  \end{equation}
  such that for every $0<\alpha'<\alpha$, the following estimates hold
  \begin{align}
    &[v+W]_{\alpha+\beta+1,\infty}\le c(\alpha,\beta,\lambda,\Lambda) [\nabla W]_{\alpha+\beta,\infty}\label{est.v}\,,\\
    &[\nabla(v+W)]_{\beta,\infty}\le c(\alpha,\beta\lambda,\Lambda)[\nabla W]_{\beta,\alpha}\label{est.dv} \,,\\
    &[\nabla(v+W)]_{\beta,\alpha'}\le c(\alpha,\alpha',\beta,\lambda,\Lambda)[\nabla W]_{\beta,\alpha}\label{est.dav}\,.
  \end{align}
  As a consequence, $v$ belongs to the space $C^{0,1+\alpha^-}_{\beta}([0,T]\times \RR^d)$.
  % \begin{equation}\label{est.v}
  %     [v+W]_{\alpha+\beta+1,\infty}\le c(\alpha,\beta,\lambda,\Lambda) [\nabla W]_{\alpha+\beta,\infty}
  % \end{equation}
  % \begin{equation}\label{est.dv}
  %     [\nabla(v+W)]_{\beta,\infty}\le c(\alpha,\beta\lambda,\Lambda)[\nabla W]_{\beta,\alpha} \,,
  % \end{equation}
  % \begin{equation}\label{est.dav}
  %     [\nabla(v+W)]_{\beta,\alpha'}\le c(\alpha,\alpha',\beta,\lambda,\Lambda)[\nabla W]_{\beta,\alpha}\,.
  % \end{equation}
\end{proposition}
\begin{proof}
 We recall that $\eta$ is the bump function defined at the beginning
 of this section and $W_{n}=W*\eta_{1/n}$. Lemma \ref{lem.fepf} yields $[W_n-W]_{\beta,\infty}$ and
 $[W_n-W]_{\beta,\alpha}$ converge to 0 as $n\to\infty$.
  Thanks to linearity of the equation \eref{eqn.v},  $v_n-v_m$ is a strong solution to
  \begin{equation*}
    (\partial_t+L_0)(v_n-v_m)=-\partial_t(W_n-W_m)\,,\quad (v_n-v_m)(T,x)=(W_n-W_m)(T,x)\,.
  \end{equation*}
  The results in Lemma \ref{lem.estv} (with $\beta_1=\beta_2=\beta$) imply
  \begin{equation*}
    [(v_n+W_n)-(v_m+W_m)]_{\beta,\infty} \lesssim [\nabla W_n-\nabla W_m]_{\beta,\infty} \,,
  \end{equation*}
  \begin{equation*}
    [\nabla(v_n+W_n)-\nabla(v_m+W_m)]_{\beta,\infty}
    \lesssim [\nabla W_n-\nabla W_m]_{\beta,\alpha} \,,
  \end{equation*}
  and for every $\alpha'\in(0,\alpha)$,
  \begin{equation*}
    [\nabla(v_n+W_n)-\nabla(v_m+W_m)]_{\beta,\alpha'}
    \lesssim[\nabla W_n-\nabla W_m]_{\beta,\alpha} \,.
  \end{equation*}
  As a consequence, $v_n$ is a Cauchy
   sequence in $C([0,T],C^{1}(K))$ for every compact set $K$ in $\RR^d$.  Thus $v_n$
   converges to $v$ in $C([0,T],C^{1}(K))$ for every compact set $K$. It is then straightforward
   to verify that $v$ is a weak solution to \eqref{eqn.v}. The estimates \eqref{est.v}, \eqref{est.dv} and \eqref{est.dav}  follow  from a limiting argument.
\end{proof}

\begin{theorem}\label{thm.wx}
  Suppose that $W$ belongs to $\Cgr$. Let $v$ be the $C^{0,1+\alpha'}_{\beta}$-generalized solution to \eqref{eqn.v} constructed
   in Proposition \ref{prop.v}. Then for every $t\in[r,T]$,  the integration $\int_r^tW(\ds,X_s^{r,x})$
     is well-defined (in the sense of Definition \ref{def.intw}).  Moreover, it
      has moment of all positive orders and satisfies
  \begin{multline}\label{rep.intw}
      \int_r^tW(\ds,X_s^{r,x})=v(r,x)-v(t,X_t^{r,x})\\-\int_r^tb(s,X_s^{r,x})\cdot\nabla v(s,X_s^{r,x})\ds
    +\int_r^t \sigma^{ij}(s,X_s^{r,x})\partial_{x_i}v(s,X_s^{r,x})\delta B_s^j\,.
  \end{multline}
  % In addition, let $\tilde{v}$ be the $C^1$-generalized solution to the parabolic equation \eqref{eqn.vtilde}. Then the following representation holds
  % \begin{equation}\label{rep.intw2}
  %   \int_r^tW(\ds,X_s^{r,x})=\tilde v(r,x)-\tilde v(t,X_t^{r,x})
  %   +\int_r^t \sigma^{ij}(s,X_s^{r,x})\partial_{x_i}\tilde v(s,X_s^{r,x})\d B_s^j\,.
  % \end{equation}
\end{theorem}
\begin{proof}
We consider $W_n=W*\eta_{1/n}$ as in the proof of the previous proposition.
  It follows from It\^o formula that (see \eqref{itotrick})
  \begin{equation*}
  \begin{split}
    \int_r^t\partial_t&W_n(s,X_s^{r,x})\ds\\
    &=v_n(r,x)-v_n(t,X_t^{r,x})-\int_r^tb(s,X_s^{r,x})\cdot\nabla v_n(s,X_s^{r,x})\ds\\
  &\quad+\int_r^t \sigma^{ij}(s,X_s^{r,x})\partial_{x_i}v_n(s,X_s^{r,x})\delta B_s^j\,.
  \end{split}
\end{equation*}
Lemma \ref{lem.estv} and Proposition \ref{prop.v} say that $v_n$
(and its derivatives) has polynomial growth and converges in
$C([0,T];C^{1+\alpha'}_\loc(\RR^d)) $ to $v$ for every $\alpha'<\alpha$. Hence, the right hand side of the above formula is
convergent in $L^p(\Omega)$ for every $p>1$. Passing through  the limit in
$n$ yields the equation  \eqref{rep.intw}.
% The proof of \eqref{rep.intw2} is similar.
\end{proof}
{\replacecolorred
\begin{remark} To define $ \int_r^tW(\ds,X_s^{r,x})$, usually one needs some regularity of $W$ on  the temporal  variable  $t$.   The equation
\eref{rep.intw}  states that the requirement of the regularity
  on   $t$  can be transformed to the one  on spatial
variable $x$ of another function $v$ (defined by \eref{e.v.definition}).
The use of $v$ appears in many situations. If $L_0$ is replaced by
$L$ in the definition of $v$ (e.g. equation
\eref{e.v.definition}) and the terminal condition
is replaced $v(0, x)=\de(x-y)$ for any fixed $y$,  then $v$ corresponds
to the transition density of the process $X_s$. This transition density
 is   a fundamental
concept in Markov processes and some other fields. It has also been used to simplify
the  proofs  of a number of  inequalities (see e.g. \cite{driverhu1996},
\cite{huchaos1997}).  The reason to use  $L_0$ instead of $L$ is
that  we don't need to assume condition on $b$ to define $v$ and  that
$\partial _iv$ will appear  in \eref{rep.intw}  even we use  $L$.
The removal of temporal regularity also appears in
other context. For example,  to study the equation $dX_t=b(X_t)+dB_t$, the transformation $Y_t=X_t-B_t$ will satisfy $\dot Y_t=b(Y_t+B_t)$. The map $(t,x)\mapsto \int_0^t b(x+B_s)ds$, averaging along the trajectories of a Brownian motion, then has better regularity than that of $b$. In the field of stochastic differential equations, this phenomena has been observed by A. M. Davie in \cite{davie} and is recently studied in more depth in \cite{catellier2012averaging}.
\end{remark}
}

As a direct consequence, we obtain
\begin{corollary}\label{cor.wnw}
  Let  $W$ be in $\Cgr$. Then for every $\alpha'<\alpha$, $p>2$ and $K$ compact subset of $\RR^d$,
  \begin{align*}
    &\|\int_r^TW(\ds,X_s^{r,x})-\int_r^TW(\ds,X_s^{r,y})-\int_r^TW_n(\ds,X_s^{r,x})+\int_r^TW_n(\ds,X_s^{r,y})\|_p
    \\&\le C(\alpha,\alpha',\beta,\lambda,\Lambda,K,T,p)([\nabla(W-W_n)]_{\beta,\infty}+[\nabla(W-W_n)]_{\beta,\alpha})|x-y|^{\alpha'}
  \end{align*}
   % $\int_r^TW_n(\ds,X_s^{r,x}) $
   % converges to $\int_r^TW(\ds,X_s^{r,x})$ in $L^p(\Omega)$ uniformly over all $r\in[0,T]$ and $x\in K$. In addition, the map $(r,x)\mapsto \int_r^t W(ds,X_s^{r,x})$ a.s. belongs to $C^{0,\alpha'}_\loc([0,T]\times \RR^d) $ for all $\alpha'<\alpha$.
  % A particular interesting choice of $W_n$ is the linear approximation
 % of $W$ along the nodes of a partition $\pi_n={r=s_0<s_1<\cdots<s_n=T}$. This gives
%   the convergence in $L^p$ of the Riemann sums
  % \begin{equation*}
  %   \sum_{k=1}^n W(s_{k},X_{s_k})
  % \end{equation*}
\end{corollary}
\begin{proof}
    Fix $\alpha'<\alpha$, $p>2$ and $K$ compact subset of $\RR^d$.
    We put $g(r,x)=\int_r^TW(ds,X_s^{r,x})$, $g_n(r,x)=\int_r^TW_n(ds,X_s^{r,x})$ and $h=v-v_n$.
    From \eqref{rep.intw},
    \begin{align*}
        \|g(r,x)-g(r,y)-g_n(r,x)+g_n(r,y) \|_p\le I_1+I_2+I_3+I_4\,,
    \end{align*}
    where
    \begin{align*}
        I_1&= |h(r,x)-h(r,y)|
        \\I_2&=\|h(T,X_T^{r,x})-h(T,X_T^{r,y})\|_p
        \\I_3&=\int_r^T\| (b\cdot\nabla h)(s,X_s^{r,x})-(b\cdot\nabla h)(s,X_s^{r,y})\|_pds
        \\I_4&=\| \int_r^T (\sigma \nabla h)(s,X_s^{r,x})-(\sigma \nabla h)(s,X_s^{r,y})\cdot \delta B_s\|_p\,.
    \end{align*}
    Proposition \ref{prop.v} implies
    \begin{align*}
        |\nabla h(z)|\lesssim ([\nabla(W-W_n)]_{\beta,\infty}+[\nabla(W-W_n)]_{\beta,\alpha})(1+|z|^\beta)\,,
    \end{align*}
    and
    \begin{align*}
        |\nabla h(x)-\nabla h(y)|\lesssim[\nabla (W-W_n)]_{\beta,\alpha}(1+|x|^{\beta'}+|y|^{\beta'})|x-y|^{\alpha'}
    \end{align*}
    where $\beta'=\beta+\alpha- \alpha'$.
    Therefore we can estimate
    \begin{align*}
        I_1=|\int_0^1 \nabla h(\tau x+(1- \tau)y)d \tau(x-y)| \lesssim \|W-W_n\| |x-y|\,,
    \end{align*}
    \begin{align*}
        I_2=\|\int_0^1 \nabla h(\tau X_T^{r,x}+(1- \tau)X_T^{r,y})d \tau(X_T^{r,x}-X_T^{r,y})\|_p
        \lesssim \|W-W_n\||x-y|\,,
    \end{align*}
    \begin{align*}
        I_3&\le \int_r^T \|[b(s,X_s^{r,x})-b(s,X_s^{r,y})]\nabla h(s,X_s^{r,x})\|_pds
        \\&\quad+\int_r^T\|b(s,X_s^{r,y})[\nabla h(s,X_s^{r,x})-\nabla h(s,X_s^{r,y})]\|_pds
        \\&\lesssim \|W-W_n\||x-y|^{\alpha'}\,,
    \end{align*}
    where we have used H\"older inequality. Similarly, we can estimate $I_4$ using Burkholder-Davis-Gundy inequality    to get $I_4\lesssim |x-y|^{\alpha'}$.
    % \begin{align*}
    %     I_4&\lesssim |x-y|^{\alpha'}\,.
    % \end{align*}
    From these bounds, the result follows.
\end{proof}
\begin{proposition}\label{prop.expw}
 Suppose $W$ belongs to $\Cgr$ with   $\alpha~+~\beta<~1$. Then
$\int_r^tW(\ds,X_s^{r,x})$ is exponentially integrable uniformly over compact sets. More precisely, for every $\gamma>0$, $K$ compact subset of $\RR^d$
  \begin{equation}
     \sup_{x\in K} \EE \exp \left\{\gamma\int_r^tW(\ds,X_s^{r,x})\right\} <\infty\,
  \end{equation}
  for all $\gamma>0$.
\end{proposition}
\begin{proof}
 From \eref{rep.intw}  it suffices to show for every $\gamma>0$,
  \begin{align}
    \label{exp.sb}& \sup_{x\in K} \EE \exp \left\{ \gamma\int_r^t \sigma^{ij}(s,X_s^{r,x})\partial_i v(s,X_s^{r,x}) \d B_s^j\right\}
    <\infty\,,\\
    \label{exp.bv}& \sup_{x\in K} \EE \exp \left\{ \gamma \int_r^t b(s,X_s^{r,x})\cdot\nabla v(s,X_s^{r,x})\ds\right\}<\infty\,,\\
    \label{exp.v}& \sup_{x\in K} \EE \exp\left\{\gamma|v(t,X_t^{r,x})|\right\} <\infty\,.
  \end{align}
 % Indeed, the Lemma follows by applying the representation \eqref{rep.intw}, H\"older inequality and these three inequalities.

  Let $0<\theta<2$.   We claim that
  \begin{equation}
    \sup_{x\in K}\EE \exp\left\{\gamma\int_r^t|X_s^{r,x}|^\theta\ds\right\} <\infty\,,\quad\forall \gamma>0\,.
  \end{equation}
  {\replacecolorred
  In fact, by Jensen inequality
  \begin{equation*}
  	\EE \exp\left\{\gamma\int_r^t|X_s^{r,x}|^\theta\ds\right\} \le  (T-r)^{-1}\int_r^T \EE e^{\gamma(T-r)|X_s^{r,x}|^{\theta}}ds\,.
  \end{equation*}
  The quality on the right hand side is finite thanks to \eqref{exp.supX}.}
 %  In fact, for every $p>1$, applying H\"older inequality and \eqref{est.xmoment} we see that
 %  \begin{align*}
 %    \sup_{x\in K}\EE \left(\int_r^t |X_s^{r,x}|^\theta\ds\right)^p&\le (t-r)^{p-1} \sup_{x\in K}\int_r^t \EE|X_s^{r,x}|^{p\theta }\ds\\
 %    &\lesssim p^{p \theta/2}A^{p \theta}
 %  \end{align*}
 %  where $A$ is some constant depending on $|K|$. It follows that for every $\gamma>0$
 %  \begin{align*}
 %    \sup_{x\in K}\EE \exp\left\{\gamma\int_r^t|X_s^{r,x}|^\theta\ds\right\}
 %    &=\sum_{p=0}^\infty \EE \frac1{p!} \left(\gamma\int_r^t |X_s^{r,x}|^\theta\ds\right)^p\\
 %    &\lesssim  \sum_{p=0}^\infty \frac{p^{p \theta/2}}{p!}(\gamma A^{\theta})^p \,.
 %  \end{align*}
 % The Stirling formula can be used to show that the above series is
 % convergent when $0<\th<2$.
  %\footnote{In the case $\theta=2$, this series is convergent only if $\gamma$ is sufficiently small}.

For any martingale $M_t$ with $\EE e^{ 2\langle M\rangle_t}<\infty$
we have
\begin{align*}
\EE e^{M_t}&=\EE e^{M_t-\langle M\rangle_t}e^{\langle
M\rangle_t}\\
&\le \left\{\EE e^{2M_t-2\langle M\rangle_t}\right\}^{1/2} \left\{
\EE e^{2\langle M\rangle _t}\right\}^{1/2}=\left\{ \EE e^{2\langle
M\rangle _t}\right\}^{1/2}\,.
\end{align*}
Thus we have
\begin{multline*}
      \EE \exp \left\{\gamma\int_r^t \sigma^{ij}(s,X_s^{r,x})\partial_i
      v(s,X_s^{r,x}) \delta B_s^j\right\}\\
       \le \left\{\EE  \exp \left[2\gamma^2
      \int_r^t(a^{ij}\partial_i v\partial_j v)(s,X_s^{r,x})\ds\right]\right\}^{1/2}\,.
  \end{multline*}
  Taking into account the growth property of $\nabla v$ (see \eqref{est.dv})  and $a$, we have
  \begin{equation*}
   \sup_{x\in K}\EE  \exp \left[2\gamma^2
      \int_r^t(a^{ij}\partial_i v\partial_j v)(s,X_s^{r,x})\ds \right] \lesssim \sup_{x\in K}\EE\exp \left[c\int_r^t|X_s^{r,x}|^{2(\alpha+\beta)}\ds\right] \,,
  \end{equation*}
  which together with the previous claim shows \eqref{exp.sb} since $2(\alpha+\beta)<2$.
  Similarly, since $b$ has linear growth
  \begin{equation*}
    \sup_{x\in K}\EE \exp\left[ \gamma \int_r^t b(s,X_s^{r,x})\cdot\nabla v(s,X_s^{r,x})\ds
    \right] \lesssim \sup_{x\in K}\EE \exp \left[c\int_r^t|X_s^{r,x}|^{1+\alpha+\beta}\right]
    \,,
  \end{equation*}
  which shows \eqref{exp.bv} since $1+\alpha+\beta<2$.

  % Let $t\mapsto x(t)$ be a solution to
  % \begin{equation*}
  %   x(t)=x+\int_r^t b(s,x(s))\ds\,.
  % \end{equation*}
  % We write
  % \begin{equation*}
  %   v(t,X_t^{r,x})-v(t,x(t))=\int_0^1 \nabla v(t,\theta X_t^{r,x}+(1- \theta)x(t))\cdot (X_t^{r,x}-x(t)) \d \theta \,.
  % \end{equation*}
  % Using the growth property of $\nabla v$, we have
  % \begin{align*}
  %   \EE \exp\gamma|v(t,X_t^{r,x})-v(t,x(t))|\le \EE \exp c(|X_t^{r,x}|^{1+\alpha+\beta} +|x(t)|^{1+\alpha+\beta})<\infty\,.
  % \end{align*}
  % Thus,
  % \begin{multline*}
  %    \EE \exp\gamma|v(t,X_t^{r,x})|\le \exp \gamma|v(t,x(t))|\EE \exp\gamma|v(t,X_t^{r,x})-v(t,x(t))|\\
  %     \le \exp \left\{c(|v(t,(x(t))|+|x(t)|^{1+\alpha+\beta})\right\}\EE \exp c|X_t^{r,x}|^{1+\alpha+\beta} \,,
  % \end{multline*}
  % which shows \eqref{exp.v}\,.
  Using the growth property of $v$, i.e. the estimate \eqref{est.v},
  \begin{equation*}
     \EE \exp
     \left[\gamma|v(t,X_t^{r,x})| \right] \lesssim \EE \exp
     \left[c|X_t^{r,x}|^{1+\alpha+\beta}\right]  \,,
  \end{equation*}
  which shows \eqref{exp.v}\,.
\end{proof}
\begin{lemma}\label{lem.expw}
Let $W$ be in $\Cgr$. Suppose $\alpha+ \beta<1$. For every $\gamma>0$ and $r\in[0,T]$, we put $u(r,x) =\EE\exp
\left[ \gamma\int_r^T W(ds,X_s^{r,x})\right]$ and $u_n(r,x) =\EE\exp
\left[\gamma\int_r^T W_n(ds,X_s^{r,x})\right]$. Then $u_n$ converges to $u$ in $C^{0,\alpha'}([0,T]\times K)$ for every $\alpha'<\alpha$
and $K$ compact in $\RR^d$.

 %  \begin{equation*}
 %    \lim_n  \EE  \exp\left\{ \gamma\int_r^T W_n(\ds,X_s^{r,x})
 %    \right\} =\EE \exp \left\{\gamma\int_r^T
 %    W(\ds,X_s^{r,x})\right\}
 %  \end{equation*}
 % uniformly  on any compact set. Moreover, there exist a subsequence $n_k$ such that the above convergence holds in $C^{\alpha'}(K)$ as $n_k\to\infty$, for every $\alpha'<\alpha$ and $K$ compact subset of $\RR^d$.
\end{lemma}
\begin{proof}
    For a smooth function $f$, using fundamental theorem of calculus, we obtain
    \begin{align*}
        f(x)-f(a)-f(y)+f(b)&=\int_0^1\int_0^1 f''(\xi)[\tau(x-y)+(1- \tau)(a-b)]d \eta d \tau(x-a)
        \\&\quad+\int_0^1f'(\theta)d \tau(x-a-y-b)\,,
    \end{align*}
    where
    \begin{align*}
        \xi&=\tau \eta x+(1- \tau)\eta a+\tau(1- \eta)y+(1-\tau)(1- \eta)b \,,
        \\ \theta&=\tau y+(1- \tau)b\,.
    \end{align*}
    Thus, for every $x,y$ in $K$, with $f(w)=\exp (\gamma w)$, we have
    \begin{align}
        &u(r,x)-u_n(r,x)-u(r,y)+u_n(r,y)
        \nonumber\\&=\gamma^2\EE\int_0^1\int_0^1 f(\xi)[\tau A(x,y)+(1- \tau)A_n(x,y)]d \eta d \tau B_n(x)
        \nonumber\\&\quad+\gamma\EE\int_0^1 f(\theta)d \tau C_n(x,y)\,,
        \label{id.u4point}
    \end{align}
    where
    \begin{align*}
        &A(x,y)=\int_r^T W(ds,X_s^{r,x})-\int_r^T W(ds,X_s^{r,y})\,,
        \\&A_n(x,y)=\int_r^T W_n(ds,X_s^{r,x})-\int_r^T W_n(ds,X_s^{r,y})\,,
        \\&B_n(x)=\int_r^T W(ds,X_s^{r,x})-\int_r^T W_n(ds,X_s^{r,x})\,,
        \\&C_n(x,y)=A(x,y)-A_n(x,y) \,.
    \end{align*}
    The random variables $\xi$ and $\eta$ are linear combinations of these terms. From Proposition \ref{prop.expw}, we know that moments of $f(\xi)$ and $f(\theta)$ are bounded uniformly in $x$ and $\tau,\eta$. On the other hand, from Corollary \ref{cor.wnw}, for every $\alpha'<\alpha$ and $p>2$
    \begin{align*}
        &\| A(x,y)\|_p\lesssim |x-y|^{ \alpha'}\,,
        \\ &\sup_n\| A_n(x,y)\|_p\lesssim |x-y|^{ \alpha'}\,,
        \\&\lim_{n\to0}\sup_{x\in K} \| B_n(x)\|=0\,,
    \end{align*}
    and
    \begin{equation*}
        \| C_n(x,y)\|_p\lesssim ([\nabla(W-W_n)]_{\beta,\infty}+[\nabla(W-W_n)]_{\beta,\alpha})|x-y|^{ \alpha'}\,.
    \end{equation*}
    From \eqref{id.u4point}, applying H\"older inequality and the above estimates for $A, B, C$ we obtain
    \begin{multline*}
        |u(r,x)-u_n(r,x)-u(r,y)+u_n(r,y)|
        \\\lesssim [\sup_{x\in K}\| B_n(x)\|_p+[\nabla(W-W_n)]_{\beta,\infty}+[\nabla(W-W_n)]_{\beta,\alpha}]|y-x|^{\alpha'}
    \end{multline*}
    for all $x,y$ in $K$ and $\alpha'<\alpha$. This completes the proof.
\end{proof}

\subsection{Feynman-Kac formula I}\label{fey.fey}
\def\d{{\de}}\def\ds{{ ds}}

If $W$ is a smooth function, then the classical Feynman-Kac formula
asserts that
\begin{equation}\label{id.FK}
  u(r,x)=\EE^B \left[u_T(X_T^{r,x})\exp\left(\int_r^T
  W(\ds,X_s^{r,x})\right) \right]
\end{equation}
is the unique strong solution to \eqref{eqn.Ldw}.  Indeed, suppose $W$ is smooth and
$u$ is a strong solution to \eqref{eqn.Ldw}.  Applying It\^o formula
to  the process
\begin{equation*}
  t\mapsto u(t,X_t^{r,x}) \exp\left\{\int_r^t\partial_tW(s,X_s^{r,x})\ds \right\}
\end{equation*}
we obtain
\begin{align*}
  \d  u(t,X_t^{r,x})& \exp\left\{\int_r^t\partial_tW(s,X_s^{r,x})\ds \right\}\\
  &=\exp\left\{\int_r^t\partial_tW(s,X_s^{r,x})\ds \right\}  (\partial_t+L+\partial_t W)u(t,X_t^{r,x}) \dt\\
  &\quad+\exp\left\{\int_r^t\partial_tW(s,X_s^{r,x})\ds \right\} \sigma^{ij}(t,X_t^{r,x})\partial_{x_i}u(t,X_t^{r,x})\d B_t^j
\end{align*}
Taking into account that $(\partial_t+L)u+\partial_tW u=0$ and
integrating over $[r,T]$, we have
\begin{multline*}
  u_T(X_T^{r,x})\exp\left\{\int_r^t\partial_tW(s,X_s^{r,x})\ds \right\}-u(r,x)\\
  =\int_r^T \exp\left\{\int_r^t\partial_tW(s,X_s^{r,x})\ds \right\}\sigma^{ij}
  (t,X_t^{r,x})\partial_{x_i}u(t,X_t^{r,x})\d B^j_t
\end{multline*}
Formula \eqref{id.FK} is deduced by taking expectation on both
sides.

\begin{theorem}\label{thm.unuctau}
  Assume $W$ belongs to $\Cgr$ with $\alpha+\beta<1$. Let $W_n=W*\eta_{1/n}$.
  Let $u_n$ be the solution to the parabolic equation
  \begin{equation*}
     \partial_tu_n+Lu_n+u_n\partial_tW_n=0\,,\quad u_n(T,x)=u_T(x)\,.
   \end{equation*}
  Let $u$ be the function defined in \eqref{id.FK}. Then $u_n$ converges to $u$ in $C^{0,\alpha'}([0,T]\times K)$ for every $\alpha'<\alpha$ and $K$ compact set in $\RR^d$. As a consequence, $u$ belongs to $C^{0,\alpha'}_\loc([0,T]\times \RR^d) $ for all $\alpha'<\alpha$.
\end{theorem}
\begin{proof}
  We notice that
  \begin{multline*}
   u_n(r,x)-u(r,x)\\
  =\EE\left\{
    u_T(X_T^{r,x})\left[\exp\left(\int_r^TW_n(\ds,X_s^{r,x})\right)
    -\exp\left(\int_r^TW(\ds,X_s^{r,x})\right)\right]\right\} \,.
  \end{multline*}
 This together with   Lemma \ref{lem.expw} yield  the theorem.
\end{proof}
% \begin{proposition}
%     Suppose the hypothesis of the previous theorem. Let $u$ be defined as in \eqref{id.FK}. Then $u$ belongs to $C^{0,\alpha'}_\loc([0,T]\times \RR^d) $ for all $\alpha'<\alpha$.
% \end{proposition}
% \begin{proof}
%     Fix $\alpha'<\alpha$, Corollary \ref{cor.wnw} assures that $(r,x)\mapsto\int_r^T W(ds,X_s^{r,x})$ belongs to $C^{0,\alpha'}_\loc([0,T]\times \RR^d)$. The result follows immediately.
% \end{proof}
We notice that if $f$ and $g$ are locally H\"older continuous functions on $\RR^d$ with exponents $\alpha$ and $\gamma$ respectively. Suppose that $f$ has compact support and $\alpha+\gamma>1$. Then we can define the Young integral
\begin{align*}
        \int_{\RR^d}f(x)g(d^jx)=\int_{\RR^d}f(x)g(x_1,\dots,x_{j-1},dx_j,x_{j+1},\dots,x_n)d\hat x_j
\end{align*}
 where $\hat x_j=(x_1,\dots,x_{j-1},x_{j+1},\dots,x_n)$.

We now show that if $W$ is sufficiently  regular in space, the Feynman-Kac solution $u$ in \eqref{id.FK} satisfies an equation derived from \eqref{eqn.Ldw} by a change of variable. To better explain our
procedure, let us first assume that $W$ is smooth in space and time
and $u_T$ is also smooth. In such case, the equation
\eqref{eqn.Ldw} has unique smooth
solution $u$  such that
\begin{equation*}
    \partial_t u(t,x)+Lu(t,x)+u\partial_tW(t,x)=0
\end{equation*}
for every $t\ge0$ and $x\in\RR^d$. We would like to obtain an equation of $u$
such that the time derivative of $W$ does not appear.
To this end, we notice that
\[\partial_t u+u\partial_t W=e^{-W}\partial_t(ue^{W})\,. \]
Hence, multiply the equation with $e^W$ and integrate in time, we obtain
% To this end,
% we put $h=e^{W}u$. Direct calculations show that $h$ satisfies
% \begin{equation*}
%   (\partial_t +e^{W}Le^{-W})h=0\,.
% \end{equation*}
% We rewrite this equation in the mild form
% \begin{equation*}
%   h_t=h_T+\int_t^T[e^{W_s}Le^{-W_s}]h_s\ds\,.
% \end{equation*}
% Putting everything back into $u$ and  using the relation $h=e^{W}u$
% we obtain
\begin{equation}\label{eqn.uew}
  u_t=e^{W_T-W_t}u_T+\int_t^T e^{W_s-W_t}Lu_s\ds\,.
\end{equation}
In contrast with \eqref{eqn.Ldw}, the equation \eqref{eqn.uew} does
not contain the time derivative of $W$. One can also interpret
\eqref{eqn.uew} in weak sense. More precisely, the following result holds.
\begin{theorem}\label{thm.feykac}
  Assume $W$ belongs to $\Cgr$ with $\alpha+\beta<1$.
  Let $u$ be the function defined in \eqref{id.FK}. Then there is a
  sequence of smooth functions  $W_n$  with compact supports   convergent to $W$ in $\Cgr$
  and a sequence of $u_n $ such that $u_n$ converges to $u$ uniformly over all compact
  sets.  Moreover, for every test function $\varphi\in C^\infty_c(\RR^d)$
 the sequence  $$\int_t^T \int_{\RR^d} \partial_i[e^{W_n(s,x)-W_n(t,x)}\varphi(x)]  a^{ij}(s,x) \partial_ju_n(s,x)\dx
    \ds$$ is convergent.
  %  If the limit is denoted by $\int_t^T \int_{\RR^d} \partial_i[e^{W(s,x)-W(t,x)}\phi(x)]
  %  a^{ij}(s,x) \partial_ju(s,x)\dx
  %  \ds  $,  then  the equation \eref{eqn.uewx} holds. Namely, $u$ is the solution
  % to \eref{eqn.Ldw} in the sense of Definition \ref{def.uew}.
    If $\alpha>1/2$, then we can identify the limit as
    \[\int_t^T\int_{\RR^d} \partial_i(e^{W(s,x)-W(t,x)}\varphi(x)) a^{ij}(s,x) u(s,d^jx)\ds\,.\]
     In such case, $u$ verifies the equation
    \begin{multline}\label{weak.eqn.u}
       \int_{\RR^d} u(t,x)\varphi(x)\dx =\int_{\RR^d} e^{W(T,x)-W(t,x)} u_T(x)\varphi(x)\dx\\
       +\frac{1}{2} \int_t^T\int_{\RR^d} \partial_i(e^{W(s,x)-W(t,x)}\varphi(x)) a^{ij}(s,x) u(s,d^jx)\ds\\
       -\int_t^T \int_{\RR^d} \partial_i\left(e^{W(s,x)-W(t,x)}\varphi(x)\left[
           b^i(s,x)-\frac12 \partial_j a^{ij}(s,x)\right]\right)  u(s,x)\dx\ds
       \,.
     \end{multline}
\end{theorem}
\begin{proof} We recall that $W_n=W*\eta_{1/n}$ defined at the beginning of this section. Let $u_n$ be the solution to the parabolic equation
  \begin{equation*}
     \partial_tu_n+Lu_n+u_n\partial_tW_n=0\,,\quad u_n(T,x)=-W_n(T,x)\,.
   \end{equation*}
 Then it is easily verified that
 \begin{multline*}
    \int_{\RR^d} u_n(t,x)\varphi(x)\dx =\int_{\RR^d} e^{W_n(T,x)-W_n(t,x)} u_n(T, x)\varphi(x)\dx\\
    +\frac12\int_t^T \int_{\RR^d} \partial_i[e^{W_n(s,x)-W_n(t,x)}\varphi(x)]  a^{ij}(s,x) \partial_ju_n(s,x)\dx \ds\\
    +\int_t^T \int_{\RR^d} e^{W_n(s,x)-W_n(t,x)}\varphi(x)\left[
    b^i(s,x)-\frac12 \partial_j a^{ij}(s,x)\right] \partial_i u_n(s,x)\dx\ds
    \,.
  \end{multline*}
 In other words,
   \begin{multline*}
  \frac12\int_t^T \int_{\RR^d} \partial_i[e^{W_n(s,x)-W_n(t,x)}\varphi(x)]  a^{ij}(s,x) \partial_ju_n(s,x)\dx \ds\\
   =  \int_{\RR^d} u_n(t,x)\varphi(x)\dx -\int_{\RR^d} e^{W_n(T,x)-W_n(t,x)} u_n(T, x)\varphi(x)\dx\\
         +\int_t^T \int_{\RR^d} \partial_i\left(e^{W_n(s,x)-W_n(t,x)}\varphi(x)\left[
    b^i(s,x)-\frac12 \partial_j a^{ij}(s,x)\right]\right)    u_n(s,x)\dx\ds
    \,.
  \end{multline*}
  Since $\varphi$ has compact support, it is clear that all  the terms on the right hand side
  are convergent.   This implies that $$ \int_t^T \int_{\RR^d}
    \partial_i[e^{W_n(s,x)-W_n(t,x)}\varphi(x)]  a^{ij}(s,x) \partial_ju_n(s,x)\dx \ds$$ is convergent.
  In case $\alpha>1/2$, by Theorem \ref{thm.unuctau}, this limit is convergent  in the context of Young integrations. Hence, taking the limit yields \eqref{weak.eqn.u}.
\end{proof}

\begin{remark}
(i) The use of It\^o formula in subsection \ref{subsec.intwx} is inspired
from the work  \cite{flandoli10}. In that work, an It\^o-Tanaka trick is applied to obtain some estimates to the commutator related to DiPerna-Lions' theory (\cite{diperna-lions}).

(ii) In the case $W$ belongs to $C^{0,2}_\loc([0,T]\times \RR^d)$, the It\^o-Tanaka formula \eqref{itotrick} is negligible.
    In fact, using integration by part, one has
    \begin{equation*}
        \int_r^T \partial_t W_n(s,X_s^{r,x})ds=W_n(T,X_T^{r,x})-W_n(r,x)-\int_r^T\nabla W_n(s,X_s^{r,x})dX_s^{r,x}
    \end{equation*}
    where the last integral is in Stratonovich sense. By passing through the limit $n\to\infty$, we obtain
    \begin{equation*}
        \int_r^T \partial_t W(s,X_s^{r,x})ds=W(T,X_T^{r,x})-W(r,x)-\int_r^T\nabla W(s,X_s^{r,x})dX_s^{r,x}\,.
    \end{equation*}
    Assuming $\nabla W$ has linear growth in the spatial variable and $\nabla^2 W$ is globally bounded, one can also show exponential integrability
    \begin{equation*}
        \EE^B \exp\left[\int_r^T \partial_t W(s,X_s^{r,x})ds\right] <\infty\,.
    \end{equation*}
    We consider $u$ as in \eqref{id.FK}. Using the approximation as in the proof of Theorem \ref{thm.feykac}, we can show that $u$ verifies
    \begin{align*}
        \int_{\RR^d}u(t,x)\varphi(x)dx&=\int_{\RR^d}e^{W(T,x)-W(t,x)}u(T,x)\varphi(x)dx\\
        &\quad+\int_t^T\int_{\RR^d}L^*[e^{W(s,x)-W(t,x)}\varphi(x)]u(s,x)dxds
    \end{align*}
    for all test functions $\varphi$ in $C^\infty_c(\RR^d)$, where $L^*$ is the adjoint of $L$.
\end{remark}

\subsection{Feynman-Kac formula II}\label{sec.feykacii}
In previous subsections, to obtain the Feynman-Kac solution \eref{id.FK}  (See Theorem
\ref{thm.feykac}) we assume that $W$ is only continuous in time but  satisfies
\eref{growth.daf}-\eref{growth.f} for $f=W$.  This  means that we suppose the the first spatial derivatives
of $W$ exist  and are H\"older continuous in order to compensate the lack of regularity in time.
For many other stochastic processes (such as Brownian sheet or fractional Brownian
sheets),  $W$ is H\"older continuous  in time.  In this case, we may use this time regularity to relax   the regularity requirement on space variable.
In this subsection we obtain a Feynman-Kac formula for the solution to
\eref{eqn.Ldw} when $W$ satisfies the conditions  of the type    given in  Section \ref{sec.pathint}.  For example, we do not require  $W$ to possess first derivatives.
More precisely, we assume $W: [0, T]\times
\RR^d\rightarrow \RR$ satisfies the following condition.
\begin{enumerate}[label=\textbf{(FK)}]
    \item\label{cond.fk} There are constants $\tau\,, \la\in (0\,,
 1]$ and  $\beta>0$ such that
\begin{equation}
  \tau +\frac12 \la >1\,,\quad \be+\la <2\ \label{e.exponents}
\end{equation}
   and such that   the seminorm
\begin{equation}
\begin{split}
  &\|W\|_{\be, \tau, \la  }\\
  :&=\sup_{\substack{0\le s<t\le T\\ x,y\in \RR^d;x\neq y}}\frac{
\left|W(s,x)-W(t,x)-W(s,y) + W(t,y)\right|}{
 (1+|x|+|y| )^\be  |t-s|^{\tau}|x-y|^{\lambda}}\\
 &\quad+ \sup_{\substack{0\le s<t\le T\\ x\in \RR^d}}\frac{
\left|W(s,x)-W(t,x)\right|}{
 (1+|x| )^{\be+\lambda}  |t-s|^{\tau}}+\sup_{\substack{0\le t\le T\\ x,y\in \RR^d;x\neq y}}\frac{
\left|W(t,y) - W(t,x)\right|}{
 (1+|x|+|y| )^\be  |x-y|^{\lambda}}
\end{split}\label{e.w-fk}
\end{equation}
is finite.
\end{enumerate}

% In addition to the   conditions \ref{cond.L.elliptic}-\ref{cond.Lbreg} we   assume further
% that $b_i(t,x)$, $i=1, \cdots, d$,  are bounded.
We continue to use
the same notations introduced in  previous subsections. For
example,   $X_t=X_t^{r,x}$ denotes  the solution to the equation
\eref{eqn.diffusion}. The   objectives of this subsection are  to show
that the expression defined by \eref{id.FK}
%\begin{equation}\label{id.FK-4}
%  u(r,x)=\EE^B \left[u_T(X_T^{r,x})\exp\left(\int_r^T
%  W(\ds,X_s^{r,x})\right) \right]
%\end{equation}
is well-defined under the above condition  \ref{cond.fk}  and is the
solution to \eref{eqn.Ldw}.

% First we need some further property on the solution $X_t^{r,x}$.
% \begin{lemma}\label{exponential-int} Suppose that  the conditions \ref{cond.L.elliptic}-\ref{cond.Lbreg}
% %in Section
% %\ref{sec.feykac}
% are satisfied and  suppose that   $b(t,x)$ is
% bounded.  Let  $\al\in(0,1/2)$  be fixed and  define $\|X\|_\al=
% \|X\|_{\al, r, T}:=\sup_{r\le s<t\le T}
% \frac{|X_t^{r,x}-X_s^{r,x}|}{|t-s|^\al}$. Then there is a constant
% $\gamma_\al >0$ such that
% \begin{eqnarray}
% \EE \exp\left\{\ga_\al \|X\|_{\al, r, T}^2\right\}<\infty\,.
% \end{eqnarray}
% \end{lemma}
% \begin{proof} From \eref{eqn.diffusion}, we have
% \begin{eqnarray*}
% \frac{X_t -X_s }{(t-s)^\al }=\frac{\int_s^t b(u,X_u ) du}{(t-s)^\al
% }+\frac{\int_s^t \si(u,X_u ) dB_u}{(t-s)^\al } =:I_1+I_2\,.
% \end{eqnarray*}
% Since $b$ is bounded, we see that
% \[
% \sup_{r\le s<t\le T} |I_1|\le C<\infty\,.
% \]
% Since $\si$ is bounded, from Lemma 2 of \cite{benarous}
% it follows that  there is a   constant $C\in (0, \infty)$ such that
% \[
% P\left(\left\|\int_0^\cdot  \si(s, X_s) dB_s
% \right\|_\al>u\right)\le C \exp\left(-u^2/C\right)\quad \forall \
% u>0\,.
% \]
% This proves that there is a positive  constant $C$ such that
% $\exp\left\{C \sup_{r\le s<t\le T} |I_2|^2\right\}<\infty$.
% Combining this with the estimate for $I_1$ proves the lemma.
% \end{proof}

From $\tau+\frac12 \la>1$, it follows that there is a $\ga\in (0,1/2)$ such that
$\tau+\ga \la>1$.  Since $X_t $ is H\"older continuous of exponent
$\ga$, from Proposition~\ref{prop.riemann.w}, we known that $\int_r^T W(ds,
X_s^{r,x})$ is well-defined and
\begin{equation}
\left|\int_r^T W(ds, X_s )\right| \le   C (1+\|X \|_\infty^\be )(1 +
\|X \|_\ga^\la ) \,.
\end{equation}
Since $\be+\la<2$,   Lemma \ref{exponential-int} yields  that
\[
\EE \exp \left\{c \int_r^T W(ds, X_s )\right\}<\infty
\]
for all $c \in \RR$.  Thus we have
\begin{prop}\label{p.fk}  Assume the conditions \ref{cond.L.elliptic}-\ref{cond.Lbreg}
are satisfied.
% and assume  $b(t,x)$ is bounded
  Let
\eref{e.exponents}-\eref{e.w-fk} be satisfied. If there is an
$\al_0\in (0, 2)$ such that $|u_T(x)|\le C_2e^{C_1|x|^{\al_0}}$,
then $u(r,x)$ defined by \eref{id.FK} is finite. Namely,
\begin{equation}\label{id.FK-4}
  u(r,x)=\EE^B \left[u_T(X_T^{r,x})\exp\left(\int_r^T
  W(\ds,X_s^{r,x})\right) \right]
\end{equation}
is well-defined.
\end{prop}

Now, let $W_n(t,x)$ be a sequence of  functions   in $C^\infty_0([0,
T]\times \RR^d)$ convergent to $W(t,x)$ under the norm
$\|W\|_\infty+\|W\|_{\be, \tau, \la}$.  Denote $v_n(r,x)=\int_r^T W_n(ds, X_s^{r,x})$ and
$v(r,x)=\int_r^T W (ds, X_s^{r,x})$  and
$
\tilde v_n(r,x)=v_n(r,x)-v(r,x)$.   Thus, for any $0\le r<t\le T$, we have
\begin{align*}
  |\tilde v_n(t,x)- \tilde v_n(r,x) |
 &= \left|\int_t^T \tilde W_n(ds, X_s^{t,x} ) - \int_r^T \tilde W_n (ds, X_s^{r,x})   \right|\\
&\le
\left|\int_r^t \tilde W_n(ds, X_s^{r,x} )    \right|  +
\left|\int_t^T \left[ \tilde W_n(ds, X_s^{t,x})-\tilde W_n(ds, X_s^{r,x}) \right]\right|\\
&=: I_1(r, t)+I_2(r,t) \,.
\end{align*}
 Applying  the estimate in Proposition~\ref{prop.riemann.w}   to $\tilde W_n=W_n-W$,  we obtain
\begin{align*}
\frac{I_1(r,t) } {(t-r)^\tau }
&\le \kappa \|\tilde W_n\|_{\tau, \lambda}  (1+\| X_\cdot ^{r,x} \|_\infty )  \left[ 1
+ \|X_\cdot ^{r,x}\|_\gamma (t-r)^{\lambda \gamma }\right]\\
&\le  C  \|\tilde W_n\|_{\tau, \lambda}  (1+\| X_\cdot ^{r,x} \|_\infty )  \left[ 1
+ \|X_\cdot ^{r,x}\|_\gamma \right]  \,.
 \end{align*}
 Lemma \ref{exponential-int} states that
  $\| X_\cdot ^{r,x} \|_\infty$  is bounded in $L^p$ for any $p\ge 1$.   Thus
 \begin{equation}
 \lim_{n\rightarrow \infty} \EE ^B \left\{
 \sup_{0\le r<t\le T} \left|\frac{ I_1(r,t) } {(t-r)^\tau }\right|^p \right\}=0
  \end{equation}
 for any $p\ge 1$.

From Proposition   \ref{prop.int.phi12}  we have with $\tau+\th \la \gamma >1$,
%{
%   \replacecolorred there was a mistake here
%}
\begin{align}
 I_2(r,t)  &\le C\|\tilde W_n\|_{\be, \tau, \la} \|X_\cdot^{t,x}-X_\cdot^{r,x}\|_\infty ^\la  (T-t)^\tau\nonumber\\
 &\quad
 +C\|\tilde W_n\|_{\be, \tau, \la} \|X_\cdot^{t,x}-X_\cdot^{r,x}\|_\infty ^{\la(1-\theta)}
  (T-t)^{\tau+\theta \la \ga}\nonumber\\
  &   \le
  C\|\tilde W_n\|_{\be, \tau, \la} \|X_\cdot^{t,x}-X_\cdot^{r,x}\|_\infty ^{\la(1-\theta)}
 \left[1+  \|X_\cdot^{t,x}-X_\cdot^{r,x}\|_\infty ^{\la\th}  \right]\,.\label{e.i2rt}
 \end{align}
Notice  that $X^{r,x}_s=X_s^{t, X_t^{r,x}}$.  We have  for any $p\ge
1$ and $\ga'<1$, by using the  Markov property of the process
$X_t^{r,x}$,
\begin{align*}
\EE  \sup_{0\le r<t\le T} \frac{ \|X_\cdot^{t,x}-X_\cdot^{r,x}
    \|_\infty^p} {(t-r)^{\ga' p/2} }
&=\EE\sup_{0\le r<t\le T} \frac{  \|X_\cdot^{t,x}-X_\cdot ^{t, X_t^{r,x}} \|_\infty^p} {(t-r)^{\ga'p/2} }  \\
&\le   C\EE \sup_{0\le r<t\le T} \frac{  \|X_t^{r,x}-x\|^p
}{(t-r)^{\ga'p/2} } \le C  \,,
\end{align*}
where the last inequality follows from a similar argument as the proof
of \eref{e.6.11}.
Combining this with
\eref{e.i2rt}
  implies
\[
\EE \sup_{0\le r<t\le T} \left|\frac{I_2(r,t)} { (t-r)^{\ga'\la (1-\th)/2} } \right|^p
\le C
%(t-r)^{p \la(1-\th) /2}
\,.
\]
Assume $\lambda/2+ \tau-1>0$.
For any  $\tau'\in (0,\lambda/2+ \tau-1)$ it is possible to find
$\th\in(0, 1)$ and $0<\ga<1/2$ such that $\tau+\th \la \gamma >1$ and
$\tau'<\ga'\la (1-\th)/2 $.  We see that
$v(\cdot,x)$  is H\"older continuous of exponent $\tau'$
and
\[
\lim_{n\rightarrow \infty}  \|v_n(\cdot, x)-v(\cdot, x)\|_{\tau'}=0\,
\]
uniformly in compact set $K$ of $\RR^d$. From \eref{id.FK-4}
it is easy to see that
\[
\lim_{n\rightarrow \infty}  \|u_n(\cdot, x)-u(\cdot, x)\|_{\tau'}=0\,
\]
uniformly in compact set $K$ of $\RR^d$.
Thus we have
\begin{proposition} Let $W_n$ be a sequence of smooth functions such that
$W_n$ converges to $W$ in the norm $\|W\|_\infty+\|W\|_{\be, \tau, \la}$ and $u_n$
is the solution to \eref{eqn.Ldw}  and $u$ is given by \eref{id.FK-4}.  Then for
any $ \tau'<\lambda/2+ \tau-1$,
%{
%   \replacecolorred $\tau'=\tau \wedge (\la (1-\th)/2)$ (replace this by $\tau'<\lambda/2- \tau-1$)
%},
$u(t, x)$ is H\"older continuous of exponent
$\tau'$ in time variable $t$ and on any compact set $K$ of $\Rd$,
{
%    \replacecolorred
    \begin{equation}
\lim_{n\rightarrow \infty}  \|u_n(\cdot, x)-u(\cdot, x)\|_{\tau'}=0\,
\end{equation}
%Replace this by $\tau+\lambda/4>1$
}
uniformly on $x\in K$.
\end{proposition}

If $\tau+\tau'>1$, then for any $\varphi \in C_{0}^{\infty }\left( \mathbb{R}^{d}\right) $,
we have
\[
\int_{0}^{t}\int_{\mathbb{R}^{d}}u_n(s,x)\left( s,x\right) \varphi (x)\frac{\partial}{\partial s}
W_n(s,x)dsdx
\]
converges to the Young integral
\[
\int_{0}^{t}\int_{\mathbb{R}^{d}}u(s,x)\left( s,x\right) \varphi (x)W(ds,x)dx\,.
\]
% \[
% \int_{0}^{t}\int_{\mathbb{R}^{d}}D_{+0}^\al \left (u_n(s,x) \right)
%  \varphi (x) D_{t-}^{1-\al}
% W_n(s,x)dsdx\,
% \]
% where $\al<\tau$ and $1-\al<\tau'$ and $D $ denotes the fractional derivative with respect to the time variable $s$.
It is obvious that the existence of $\tau'>0$ such that $\tau+\tau'>1$ and
$ \tau'<\lambda/2+ \tau-1$ is equivalent to $\la+4\tau>4$.
The above argument  means  that $u(t,x)$ is
a weak solution to \eref{eqn.Ldw}, in the sense of   next theorem.

%
%
%\begin{definition}
%\label{def weak solu} A random field $u=\{u\left( t,x\right) ,t\geq 0,x\in
%\mathbb{R}^{d}\}$ is called a weak solution to Equation \eref{eqn.Ldw}  if for any $%
%\varphi \in C_{0}^{\infty }\left( \mathbb{R}^{d}\right) $, we have
%\begin{align}
%\int_{\mathbb{R}^{d}} u\left( t,x\right)  \varphi (x)dx
%&=\int_{\mathbb{R}^{d}} u_{0}(x) \varphi (x)dx
%+\int_{0}^{t}\int_{\mathbb{R}^{d}}u\left( s,x\right) L^*\varphi (x)dxds  \notag \\
%&\quad+\int_{0}^{t}\int_{\mathbb{R}^{d}}u\left( s,x\right) \varphi (x)W(ds,x)dx
%\label{e.weak-sol}
%\end{align}%
% for all $t\geq 0$, where the last term is  defined
% by Young integration and $L^*$ is the dual of $L$.
%\end{definition}

 \begin{theorem}Assume the conditions \ref{cond.L.elliptic}-\ref{cond.Lbreg}
are satisfied and assume  there is an
$\al_0\in (0, 2)$ such that $|u_T(x)|\le C_2e^{C_1|x|^{\al_0}}$.
% and assume  $b(t,x)$ is bounded
  Let  $\|W\|_{\be, \tau, \la}$ defined by  \eref{e.w-fk} be finite,  where the H\"older
exponents $\la$ and $\tau$ and the growth exponent $\be$ satisfy
\begin{equation}
\tau>1/2\,,  \quad \be+\la <2\,, \quad     \la +4\tau>4\,.
\label{h.exponents}
\end{equation}
Then $u$ defined by \eref{id.FK-4} is  a weak solution to
 \eref{eqn.Ldw} in the sense that $u$ satisfies
 \begin{align}
\int_{\mathbb{R}^{d}} u\left( t,x\right)  \varphi (x)dx
&=\int_{\mathbb{R}^{d}} u_{0}(x) \varphi (x)dx
+\int_{0}^{t}\int_{\mathbb{R}^{d}}u\left( s,x\right) L^*\varphi (x)dxds  \notag \\
&\quad+\int_{0}^{t}\int_{\mathbb{R}^{d}}u\left( s,x\right) \varphi (x)W(ds,x)dx\,,
\label{e.weak-sol}
\end{align}
where  $\varphi$ is any smooth function with compact support and where
the last integral
%is
% with
%\[
%\int_{0}^{t}\int_{\mathbb{R}^{d}}u\left( s,x\right) \varphi (x)W(ds,x)dx
%% =\int_{0}^{t}\int_{\mathbb{R}^{d}}D_{+0}^\al \left (u_n(s,x) \right)
%%  \varphi (x) D_{t-}^{1-\al}
%% W_n(s,x)dsdx\,,
%\]
is a Young integral.
\end{theorem}

%\begin{proof}
%We only need to explain \eref{h.exponents} implies
%the following conditions for
%$\tau$ and $\la$ that we have used in this subsection:
%\begin{eqnarray}
%&&\tau +\frac12 \la >1\,; \label{la} \\
%&&\tau+\th \la >1\quad\hbox{ for some $\th\in (0, 1)$}\label{th}\\
%&&\tau'+\tau>1\,, \quad \hbox{where $\tau'=\tau \wedge (\la (1-\th)/2)$}\,.
%\label{tau}
%\end{eqnarray}
%When $\tau'=\tau$,  this is obvious. If $\tau'= \la (1-\th)/2$,  then \eref{tau} is
%$
%\la (1-\th)/2+\tau>1
%$.   The existence of $\th\in (0, 1)$ such that $
%\la (1-\th)/2+\tau>1
%$ and $\tau+\th \la >1$ is $\tau+\frac \la3>1$ by choosing $\th$ close
%to $\frac{1-\tau}{\la}$ (which is smaller than $1/3$).     Clearly,
%$\tau+\frac \la3>1$ implies \eref{la}.
%\end{proof}

\begin{remark}  Equation \eref{e.weak-sol} is the definition of the weak solution
used in \cite{hu-lu-nualart12},  \cite{hunualart09}
and \cite{hunualartsong}.
\end{remark}

%Thus for any $p\ge 1$ and for any compact subset $K$ of $\RR^d$ , we
%have
%\[
%\lim_{n\rightarrow \infty} \sup_{r\in [0, T] ,\,  x\in K  } \EE
%\left|\int_r^T W_n(ds, X_s) - \int_r^T W (ds, X_s)\right|^p=0\,.
%\]
%This implies  $\lim_{n\rightarrow \infty} \sup_{r\in [0, T] ,\, x\in
%K  } |u_n(r, x)-u(r,x)|=0$.
%
%This means
%\begin{theorem}\label{t.fk-sec-4} Assume the conditions of Proposition \ref{p.fk}. Then $u_n(t,x)$
%converges to  $\displaystyle \lim_{n\rightarrow \infty} \sup_{r\in [0, T] ,\, x\in
%K  } |u_n(r, x)-u(r,x)|=0$ for any compact set $K$ of $\RR^d$.
%\end{theorem}

%\begin{theorem}  Suppose that  the conditions (L1)-(L3)  are satisfied and suppose that  $b(t,x)$ is bounded.
%Assume there exists an   $\al_0\in (0, 2)$ such that $|u_T(x)|\le
%C_2e^{C_1|x|^{\al_0}}$.   Let \eref{e.exponents}-\eref{e.w-fk} be
%satisfied so that  $u$ given by  \eqref{id.FK-4} is well-defined.
%Then there is a
%  sequence of smooth functions  $W_n$  with compact supports   convergent to $W$ in
%  the norm $\|W\|_\infty+\|W\|_{\be, \tau, \la}$
%  and a sequence of $u_n $ such that $u_n$ converges to $u$ uniformly over all compact
%  sets.  Moreover,
% the sequence  $\int_t^T \int_{\RR^d} \partial_i[e^{W_n(s,x)-W_n(t,x)}\phi(x)]  a^{ij}(s,x) \partial_ju_n(s,x)\dx
%   \ds$ is convergent.  If the limit is denoted by $\int_t^T \int_{\RR^d} \partial_i[e^{W(s,x)-W(t,x)}\phi(x)]
%   a^{ij}(s,x) \partial_ju(s,x)\dx
%   \ds  $,  then  the equation \eref{eqn.uewx} holds. Namely, $u$ is the solution
%  to \eref{eqn.Ldw} in the sense of Definition \ref{def.uew}.
%\end{theorem}

\section{Asymptotic growth of Gaussian
sample paths}\label{sec.cov-path}
\def\TT{\mathbf{T}}

In Sections \ref{sec.pathint}, \ref{sec.diffeqn} and
\ref{sec.feykac}, we  assume the pathwise H\"older continuity and
pathwise growth conditions on $W$ in order to define and  to solve
(partial) differential equations related to the nonlinear integral
$\int W(ds,\varphi_s)$.  For instance, the conditions
\eqref{eq:cond.w2}, \eqref{growth.daf}, \eqref{growth.df},
\eqref{growth.f} are essential in various parts of the paper. {\replacecolorred  In probability
theory,   it is usually hard to obtain  properties for (almost) every sample path
of a stochastic process  from its average properties (from its  probability law).     In this section,
we   investigate  these pathwise H\"older continuity and
pathwise growth   problems for a stochastic process.} We shall focus
 on
Gaussian random fields.  However, our method works well for other processes
provided they satisfy some suitable normal concentration
inequalities (for instance, see the assumptions in Theorem \ref{thm.maj.gaussian}).

Let $W$ be a stochastic process on $[0,T]\times\RR^d$. An
application of  our results yields the asymptotic growth of the
quantity
\begin{equation*}
    I(\delta,R)= \sup_{{{t\in[0,T]}}}\sup_{|x|,|y|\le R;|x-y|\le \delta}
    \frac{|W(t,\square[x,y])|}{|x_1-y_1|^{\lambda_1}\cdots|x_d-y_d|^{\lambda_d}}
\end{equation*}
as $R\to\infty$, where $W(t,\square[x,y])$ denotes the $d$-increment
of $W(t,\cdot)$ over the rectangle $[x,y]$.  More precise definition
is given in Subsection \ref{subsec.majmeas}. If $R$ is fixed, the
quality $I(\delta,R)$ is the objective in our previous work
\cite{hule2012} via a multiparameter version of
Garsia-Rodemich-Rumsey inequality.

Let us mention some historical remarks.  (Pathwise) boundedness and continuity
for stochastic processes have been studied thoroughly in literature.
One of the central ideas is originated in an important early paper
by Garsia, Rodemich and Rumsey  (1970) \cite{garsiarodemich}. This
was developed further by Preston (1971,1972) \cite{preston71,preston72}, Dudley (1973) \cite{dudley73} and
Fernique (1975) \cite{fernique75}.  In these considerations, the parameter space $\TT$
is bounded and treated as a  ``single-dimension" object.  For
instance, the well-known Dudley bound
\[
\EE \sup_{s,t\in \TT}|W(t)-W(s)| \lesssim  \int_0^{d_W(s,t)} \sqrt{\log N(\TT, d_W, \vare)} d\vare
\]
yields modulus of continuity in terms of the entropy number
$N(T,d,\ep)$.  This is extended to a more precise bound in terms of
majorizing measure
\begin{equation*}
    \EE \sup_{\substack{s,t\in\TT\\d_W(s,t)\le \delta}}|W(t)-W(s)|
    \lesssim\sup_{t\in\TT}\int_0^\delta\log^{1/2}\frac1{\mu(B_{d_W}(t,u))}du\,.
\end{equation*}
The majorizing-measure bound turns out to be necessary for processes
which  satisfy normal concentration inequalities. This result by M.
Talagrand is the milestone in theory of Gaussian processes. We refer
the readers to \cite[Chapter 6]{marcus-rosen} and references therein
for details and more historical facts. See also Talagrand's
monograph \cite{talagrand-book} in which the role of majorizing
measure is replaced by a variational quality called
$\gamma_2(\TT,d_W)$.

Estimates for the $d$-increment of $W$ over a rectangle are quite
different.  Difficulties arise since $W(\square[s,t])$ does not
behave nicely as increments. In particular, the corresponding entropic  ``metric"
\begin{equation*}
    (\EE W(\square[s,t])^2)^{1/2}
\end{equation*}
does not satisfy the triangle inequality, but rather behaves like  a
volume metric.  {\replacecolorred  To elaborate  this point, let us consider the
two dimensional case:
\begin{eqnarray*}
W(\square[s,t])
&=&W(s_2, t_2)-W(s_2, t_1)-W(s_1, t_2)+W(s_1, t_1)\\
&=&\Delta_{[t_2, t_1]} W(s_2 )- \Delta_{[t_2, t_1]} W(s_1 )=\Delta_{[s_2, s_1]}\Delta_{[t_2, t_1]} W
\,,
\end{eqnarray*}
where $\tilde W(s):=\Delta_{[t_2, t_1]} W(s )=W(s; t_2)-W(s; t_1)$. This product-like property is essential in our current approach (see for instance inequality \eqref{ineq.mkmk} below). Alternatively, to obtain a sharp bound for the
difference,  one can repeatedly apply the Garsia-Rodemich-Rumsey inequality first
$ \Delta_{[s_2, s_1]}\tilde W$ and then to $\Delta_{[s_2, s_1]}\Delta_{[t_2, t_1]} W$. Indeed,  for bounded parameter domains equipped with Lebesgue measure, this direction was developed by the authors in \cite{hule2012}. This idea, while might be feasible, seems to be more complicated in our current setting with general
(unbounded)  parameter domains equipped with a general measure.}

In Subsection \ref{subsec.real.var.ineq},
we will prove a deterministic inequality, which is more precise than the multiparameter Garsia-Rodemich-Rumsey inequality obtained in \cite{hule2012}. We then apply it to obtain a majorizing-measure bound on the $d$-increments of stochastic processes in Subsection
\ref{subsec.majmeas}. Our formulations are benefited from the treatment in
\cite{marcus-rosen}. We however did not consider the necessary
conditions for these bounds (i.e. lower bounds). {\replacecolorred  Results in these two subsections
are applicable to general stochastic processes.}

Given a well-developed toolbox to treat the case when $\TT$ is
bounded  (or for example, $R$ is fixed in $I(\delta,R)$), the
asymptotic growth for $I(\delta,R)$ as $R\to\infty$ can be obtained
using concentration inequalities {\replacecolorred  for Gaussian processes.
More precise results are given for fractional Brownian fields.
This is done in
Subsection \ref{subsec.asym.growth}.}

\subsection{A deterministic inequality}\label{subsec.real.var.ineq}
Throughout the current subsection,  we put $\Psi(u)=\exp(u^2)-1$.
Suppose $\mu$ is a nonnegative measure on $\TT$ and $X$ is a
measurable function on $\TT$.  We define
\begin{equation*}
    [X]_{\Psi,(\TT,\mu)}:=\inf\left\{\alpha>0: \int_\TT \Psi
    \left(\frac{X(t)}{\alpha}\right)\mu(dt)\le 1 \right\}\,.
\end{equation*}
When the parameter space $\TT$ and the measure $\mu$ are clear from the
context, we often suppress them and write $[X]_{\Psi}$ instead.
The following result, whose proof is given in \cite[pg. 256-258]{marcus-rosen}, is an application of the Young inequality
\begin{equation*}
    ab\le\int_0^a g(x)dx+\int_0^b g^{-1}(x)dx\,,
\end{equation*}
where $g$ is a real-valued, continuous and strictly increasing function.
\begin{lemma}\label{lemma.young.ineq}
    Let $X$ and $f$ be measurable functions on $\TT$, $\mu$ be a
    nonnegative measure on $\TT$. Assume that $[X]_{\Psi,(T,\mu)}$ is finite and $0<\int|f|d \mu<\infty$. Then
    \begin{equation*}
        \int_\TT |X(t)f(t)|\mu(dt)\le 3 [X]_{\Psi,(\TT,\mu)} \int_\TT |f(t)|\log^{1/2}
        \left(1+\frac{|f(t)|}{\int|f(s)|\mu(ds) } \right)\mu(dt)\,.
    \end{equation*}
\end{lemma}
We consider the case when $\TT$ has the form
$\TT=\TT_1\times\cdots\times\TT_\ell$.  A parameter $t$ in $\TT$ has
$\ell$ components, $t=(t_1,\dots,t_\ell)$. For each $i=1,\dots,\ell$, the space $\TT_i$ is equipped with a metric $d_i$. We also denote $d^*(s,t)=d_1(s_1,t_1)\dots d_\ell(s_\ell,t_\ell)$ for every $s,t$ in $\TT$. Let $X$ be a function on
$\TT$. We define the $\ell$-increment of $X$ over a ``rectangle"
$[s,t]$ as
\begin{equation*}
    X(\square[s,t])=\prod_{j=1}^{\ell}(I-V_{j,s})X(t)\,.
\end{equation*}
In the above expression, $I$ is the identity operator, $V_{j,s}$ is
the substitution operator which substitutes the $j$-th component of
a function on $\TT$ by $s_j$, more precisely,
\[
V_{j,s}X(t)=X(t_{1},\dots,t_{j-1},s_{j},t_{j+1},\dots,t_{\ell})\,.
\]
We refer to \cite{hule2012} for a more detailed description on this
$\ell$-increment.

For each $i$, $B^i(t_i,u)$ denotes the open ball with radius $u$ in
the metric  space $(\TT_i,d_i)$ centered at $t_i$. For each $t$ in
$\TT$, we denote $B(t,u)=B^1(t_1,u)\times\cdots\times
B^\ell(t_\ell,u)$.  For each $j$, put $D_j=\sup_{s_j,t_j\in\TT_j}d_j(s_j,t_j)$.

For each $i=1,\dots,\ell$, let $\mu_i$ be a probability measure on $\TT_i$. Let $k=(k_1,\dots,k_\ell)$ be a multi-index in $\NN^\ell$. We define
    \begin{align*}
        &\mu_k^i(t_i)=\mu^i(B^i(t_i,D_i2^{-k_i}))\,,
        & &\rho_{k_i}(t_i,\cdot)=\frac1{\mu_k^i(t_i)}1_{B^i(t_i,D_i2^{-k_i})}(\cdot)
        \\&\mu_k(t)=\prod_{i=1}^\ell \mu_k^i(t_i) \,,
        & &\rho_k(t,\cdot)=\prod \rho_{k_i}(t_i,\cdot)
    \end{align*}
    and
    \begin{equation}\label{eqn.def.mk}
        M_k(t)=\int_\TT \rho_k(t,u)X(u)\mu(du)\,.
    \end{equation}
% By the continuity assumption on  $X$, $\lim_{k_1,\dots,k_\ell\to\infty}M_k(t)=X(t)$ for all $t$.

We use the notations $k+1=(k_1+1,\dots,k_\ell+1)$, $k+1_j=(k_1,\dots,k_{j-1},k_j+1,k_{j+1},\dots,k_\ell)$, $\hat t_i=(t_1,\dots,t_{i-1},t_{i+1},\dots,t_\ell)$ and $\hat\TT_i=\TT_1\times\dots\times\TT_{i-1}\times\TT_{i+1}\times\dots\times\TT_\ell$.
% We put $d^*(s,t) =d_1(s_1,t_1)\cdots d_\ell(s_\ell,t_\ell)$.

\begin{theorem}\label{thm.maj.meas}
    Let $\{X(t),t\in \TT\}$ be a measurable  function on $\TT$.
    We put  $\mu=\mu^1\times\cdots\times \mu^\ell$ and
    \begin{equation*}
        Z=\inf\left\{\alpha>0:\iint_{\TT\times \TT}\Psi\left(\frac{X(\square[u,v])}{\alpha d^*(u,v)}\right) \mu(du)\mu(dv)\le1 \right\}\,.
    \end{equation*}
    Assume that $D_j$, $j=1,\dots,d$, and $Z$ are finite. Then, for every $s,t$ in $\TT$ such that the integral
    \begin{equation*}
        \int_0^{d_1(s_1,t_1)}du_1\cdots\int_0^{d_\ell(s_\ell,t_\ell)}du_\ell
        \left(\log^{1/2}\frac1{\mu(B(s,u))}+\log^{1/2}\frac1{\mu(B(t,u))} \right)
    \end{equation*}
    is finite, $M_k(\square [s,t])$ converges to a limit, denoted by $X'(\square [s,t])$, as $k_1,\dots,k_\ell$ go to infinity. In addition, $X'(\square [s,t])$ satisfies
    \begin{multline}\label{ineq.maj.meas}
        |X'(\square[s,t])|\le C^\ell Z\int_0^{d_1(s_1,t_1)}du_1\cdots\int_0^{d_\ell(s_\ell,t_\ell)}du_\ell
        \\\left(\log^{1/2}\frac1{\mu(B(s,u))}+\log^{1/2}\frac1{\mu(B(t,u))} \right)\,.
    \end{multline}
\end{theorem}
\begin{proof}
    Fix $s,t$ in $\TT$.  We choose the multi-index $n$ such that $D_j2^{-n_j-1}\le d_j(s_j,t_j)\le D_j2^{-n_j}$ for each $j=1,\dots,\ell$. It suffices to show that the following series satisfies the bound in \eqref{ineq.maj.meas}
    \begin{equation}\label{ineq.xsquare}
        |M_n(\square[s,t])|+\sum_{k\ge n}|M_{k+1}(\square[s,t])-M_k(\square[s,t])|\,.
    \end{equation}
    We estimate the first term. Notice that we can write
    \begin{equation*}
        M_n(\square[s,t])=\iint_{\TT\times\TT}X(\square[u,v])\rho_n(s,u)\rho_n(t,v)\mu(du)\mu(dv)\,.
    \end{equation*}
    We consider the function $\{Y(u,v),(u,v)\in\TT\times\TT\}$ defined by
    \begin{equation*}
        Y(u,v)=
        \begin{dcases*}
            \frac{X(\square[u,v])}{d^*(u,v)}& when $d^*(u,v)\neq 0\,,$
            \\0& otherwise\,.
        \end{dcases*}
    \end{equation*}
    It is clear that
    \begin{align*}
        |M_n(\square[s,t])|
        &\le\iint_{\TT\times\TT}|Y(u,v)|d^*(u,v)\rho_n(s,u)\rho_n(t,v)\mu(du)\mu(dv)
        \\&\lesssim(D_12^{-n_1})\cdots(D_\ell2^{-n_\ell}) \iint_{\TT\times\TT}|Y(u,v)|\rho_n(s,u)\rho_n(t,v)\mu(du)\mu(dv)
        \,,
    \end{align*}
    since the support of $\rho_n(s,\cdot)\rho_n(t,\cdot)$, $d^*(u,v)\lesssim (D_12^{-n_1})\cdots(D_\ell2^{-n_\ell})$.
    We now apply Lemma \ref{lemma.young.ineq}  to  the functions $Y$ and $\rho_n(s,\cdot)\otimes\rho_n(t,\cdot)$ on the product space $(\TT\times\TT,\mu\otimes \mu)$, observing that $Z=[Y]_{\Psi}$, $\iint\rho_n(s,\cdot)\rho_n(t,\cdot)=1$ and $\rho_n(s,u)\rho_n(t,v)\le(\mu_n(s,u)\mu_n(t,v))^{-1}$,
    \begin{align*}
        &|M_n(\square[s,t])|
        \\&\lesssim Z(D_12^{-n_1})\cdots(D_\ell2^{-n_\ell})  \iint_{\TT\times\TT} \rho_n(s,u)\rho_n(t,v)\log^{1/2}\left(1+\rho_n(s,u)\rho_n(t,v) \right) \mu(du)\mu(dv)
        \\&\lesssim Z(D_12^{-n_1})\cdots(D_\ell2^{-n_\ell}) \log^{1/2}\left(1+\frac1{\mu_n(s)\mu_n(t)} \right) \,.
    \end{align*}
    Since $d^*(s,t)\asymp(D_12^{-n_1})\cdots(D_\ell2^{-n_\ell})$, this shows
    \begin{align}
        &|M_n(\square[s,t])|
        \nonumber\\&\lesssim Z\int_0^{d_1(s_1,t_1)}du_1\cdots\int_0^{d_\ell(s_\ell,t_\ell)}du_\ell\log^{1/2}\left(\frac1{\mu(B(s,u))}+\frac1{\mu(B(t,u))} \right)\,.
        \label{ineq.mn.square}
    \end{align}
    We now estimate each term in the sum appear in \eqref{ineq.xsquare}. We denote $\tau_0k=k$ and recursively $\tau_jk=\tau_{j-1}k+1_j$ for each $j=1,\dots,\ell$. For example, $\tau_1k=(k_1+1,k_2,\dots,k_\ell)$ and $\tau_\ell k=k+1$. We then write
    \begin{align}\label{ineq.mkmk}
        |M_{k+1}(\square[s,t])-M_k(\square[s,t])|\le\sum_{j=1}^\ell|M_{\tau_jk}(\square[s,t])-M_{\tau_{j-1}k}(\square[s,t])|\,.
    \end{align}
    Note that the multi-indices $\tau_jk$ and $\tau_{j-1}k$ differs by exactly 1 unit at the $j$-th component. Without loss of generality, we consider the case
    \begin{equation*}
        |M_{\tilde k}(\square[s,t])-M_{k}(\square[s,t])|\,,
    \end{equation*}
    where $\tilde k=k+1_\ell=(k_1,\dots,k_{\ell-1},k_\ell+1)$.
    We adopt the notations $w=(w',w_\ell)$ for every $w$ in $\TT$, $$\rho'_k(s',u')=\rho_{k_1}(s_1,u_1)\cdots\rho_{k_{\ell-1}}(s_{\ell-1},u_{\ell-1}) $$ and similarly for $\rho'_k(t',v')$. We then write
    \begin{align*}
        M_k(\square[s,t])=M_k(\square^{\ell-1}[s',t'],s_\ell)-M_k(\square^{\ell-1}[s',t'],t_\ell)
    \end{align*}
    and similarly for $M_{\tilde k}(\square[s,t])$. Thus
    \begin{align}
        &|M_{\tilde k}(\square[s,t])-M_{k}(\square[s,t])|
        \nonumber\\&\le |M_{k+1_\ell}(\square^{\ell-1}[s',t'],s_\ell)-M_k(\square^{\ell-1}[s',t'],s_\ell)|
        \nonumber\\&\quad+|M_{k+1_\ell}(\square^{\ell-1}[s',t'],t_\ell)-M_k(\square^{\ell-1}[s',t'],t_\ell)|
        \label{ineq.mktilde}\\&=I_1+I_2\,.
        \nonumber
    \end{align}
    We only need to estimate $I_1$ since $I_2$ is analogous. We have
    \begin{multline*}
        M_k(\square^{\ell-1}[s',t'],s_\ell)
        \\=\iint_{\TT\times\TT}X(\square^{\ell-1}[u',v'],v_\ell)\rho'_{k}(s',u')\rho'_{k}(t',v')\rho_{\ell+1}(s_\ell,u_\ell)\rho_{\ell}(s_\ell,v_\ell)\mu(du)\mu(dv)
    \end{multline*}
    and similarly
    \begin{multline*}
        M_{\tilde k}(\square^{\ell-1}[s',t'],s_\ell)
        \\=\iint_{\TT\times\TT}X(\square^{\ell-1}[u',v'],u_\ell)\rho'_{k}(s',u')\rho'_{k}(t',v')\rho_{\ell+1}(s_\ell,u_\ell)\rho_{\ell}(s_\ell,v_\ell)\mu(du)\mu(dv)\,.
    \end{multline*}
    Note how the dummy variables $v_\ell$ and $u_\ell$ have
     been switched between the two formulas. Hence
    \begin{multline*}
        |M_{\tilde k}(\square^{\ell-1}[s',t'],s_\ell)-M_k(\square^{\ell-1}[s',t'],s_\ell)|
        \\\le\iint_{\TT\times\TT}|X(\square^{\ell}[u,v])|\rho'_{k}(s',u')\rho'_{k}(t',v')\rho_{\ell+1}(s_\ell,u_\ell)\rho_{\ell}(s_\ell,v_\ell)\mu(du)\mu(dv)\,.
    \end{multline*}
  Similarly to the term $M_n(\square[s,t])$ one can obtain
    \begin{align*}
        |M_{\tilde k}(\square^{\ell-1}[s',t'],s_\ell)&-M_k(\square^{\ell-1}[s',t'],s_\ell)|
        \\&\lesssim Z(D_12^{-k_1})\cdots(D_\ell2^{-k_\ell}) \log^{1/2}\left(1+\frac1{\mu_{\tilde k}(s)\mu_{k}(s)} \right)
        \\&\lesssim  Z(D_12^{-k_1})\cdots(D_\ell2^{-k_\ell}) \log^{1/2}\left(1+\frac1{\mu_{k}(s)} \right)\,.
    \end{align*}
    Therefore, combining altogether \eqref{ineq.mkmk}, \eqref{ineq.mktilde} and the previous estimate, we get
    \begin{equation*}
        |M_{k+1}(\square[s,t])-M_k(\square[s,t])|
        \lesssim Z\ell(D_12^{-k_1})\cdots(D_\ell2^{-k_\ell}) \log^{1/2}\left(1+\frac1{\mu_{k}(s)} \right)\,,
    \end{equation*}
    and hence,
    \begin{align*}
        \sum_{k\ge n}&|M_{k+1}(\square[s,t])-M_k(\square[s,t])|
        \\&\lesssim Z\ell(D_12^{-n_1})\cdots(D_\ell2^{-n_\ell}) \log^{1/2}\left(1+\frac1{\mu_{n}(s)} \right)
        \\&\lesssim Z\ell\int_0^{d_1(s_1,t_1)}du_1\cdots\int_0^{d_\ell(s_\ell,t_\ell)}du_\ell\log^{1/2}\left(\frac1{\mu(B(s,u))}+\frac1{\mu(B(t,u))} \right)\,.
    \end{align*}
    Together with the bound for $M_n(\square[s,t])$ (inequality \eqref{ineq.mn.square}) and \eqref{ineq.xsquare}, this completes the proof.
\end{proof}
\begin{remark}
    In Theorem \ref{thm.maj.meas}, $X'$ may not be defined as a function on $\TT$, that is for each $t$ in $\TT$, there is no a priory reason for $X'(t)$ to be defined. However, in order to keep the representation compact, we have abused of notations and denote the limit as $X'(\square[s,t])$. This object is well-defined for every fixed $s,t$ in $\TT$.
\end{remark}

\subsection{Majorizing measure}\label{subsec.majmeas}
We now suppose that $X$ is a stochastic process with the probability space $(\Omega,\cF,P)$. We introduce the $\ell$-fold volumetric
\begin{equation*}
    d^\ell(s,t)=\left(\EE [X(\square[s,t])]^2\right)^{1/2}\,.
\end{equation*}
Assume that $\sup_{s,t\in\TT}d^\ell(s,t)$ is finite. In addition, for each $i$, there exists a metric $d_i$ on $\TT_i$
such that
\begin{equation*}
    d^\ell(s,t)\le d_1(s_1,t_1)\cdots d_\ell(s_\ell,t_\ell)\,.
\end{equation*}
This is not a restriction since such collection of metrics always exists. For instance, one can choose
\begin{equation*}
    d_1(s_1,t_1)=\sup_{\hat s_1,\hat t_1\in\hat\TT_1}d^\ell(s,t)
\end{equation*}
and recursively
\begin{equation*}
    d_k(s_k,t_k)=\sup_{\hat s_k,\hat t_k\in\hat\TT_k}\frac{d^\ell(s,t)}{\prod_{i=1}^{k-1}d_i(s_i,t_i)}
\end{equation*}
with the convention $0/0=0$.

We denote $Z$ as in Theorem \ref{thm.maj.meas}, that is
\begin{equation}\label{eqn.def.Z}
        Z=\inf\left\{\alpha>0:\iint_{\TT\times \TT}\Psi\left(\frac{X(\square[u,v])}{\alpha d^*(u,v)}\right) \mu(du)\mu(dv)\le1 \right\}\,.
\end{equation}
We assume that $Z$ is finite almost surely.
\begin{example}
    Suppose $X$ is a centered Gaussian process. Then $Z$ has exponential tail. More precisely $P(Z>u)\le (e\log2)^{1/2}u2^{-u^2/2}$ for all $u>(2+1/\log 2)^{1/2}$. This comes from a standard argument by Chebyshev inequality and H\"older inequality, see \cite[pg. 256-258]{marcus-rosen} for details.
    % Indeed, we put
 %    \begin{equation*}
 %        Y(u,v)=
 %        \begin{dcases*}
 %            \frac{X(\square[u,v])}{d^*(u,v)}& when $d^*(u,v)\neq 0\,,$
 %            \\0& otherwise\,.
 %        \end{dcases*}
 %    \end{equation*}
 %    Then $Z=[Y]_{\Psi}$ and $\EE Y^2\le 1$. Thus, for all $\beta<1/2$,
 %    $\sup_{u,v\in T}\EE \exp(\beta Y^2(u,v))<\infty$. This implies, by Fubini's theorem, that for all $\beta<1/2$,
 %    the set $\{\omega:\exp(\beta Y^2(\omega,\cdot))\in L^1(\TT\times\TT,\mu\otimes \mu) \} $ has
 %     probability one. Thus, by dominated convergence theorem
 %    \begin{equation*}
 %        \lim_{\alpha\to\infty}\iint \Psi\left(\frac{Y(\omega,s,t)}{\alpha}\right)\mu(ds)\mu(dt)=0\quad \mbox{ a.s.}
 %    \end{equation*}
 %    $Z(\omega)$ is finite almost sure. To get the estimate for $\EE Z$, we apply Chebyschev inequality and H\"older inequality, for all $u>0$ and $p\ge1$,
 %    \begin{align*}
 %        P(Z>u)&\le P(\iint\exp\frac{Y^2(s,t)}{u^2}\mu(dt)\mu(ds)\ge2)
 %        \\&\le 2^{-p}\EE\left(\iint\exp\frac{Y^2(s,t)}{u^2}\mu(dt)\mu(ds)\right)^p
 %        \\&\le 2^{-p}\EE\iint\exp\frac{pY^2(s,t)}{u^2}\mu(dt)\mu(ds)
 %        \\&\le \frac{2^{-p}}{(1-2p/u^2)^{1/2}}\,.
 %    \end{align*}
 %    For $u>(2+1/\log 2)^{1/2}$, this last term is minimized by $p=u^2/2-1/(2\log2)$. Thus, for $u>(2+1/\log 2)^{1/2}$, $P(Z>u)\le (e\log2)^{1/2}u2^{-u^2/2}$. This also shows $\EE Z<5/2$.
\end{example}

As an application of Theorem \ref{thm.maj.meas}, we have
\begin{theorem}\label{thm.maj.gaussian}
    Let $\{X(t),t\in\TT\}$ be a stochastic process such that $Z$, defined in \eqref{eqn.def.Z}, is finite a.s.\   Then $X$ has a version $X'$ such that for all $\omega\in\Omega$ and $s,t$ in $\TT$
    \begin{multline*}
        |X'(\omega,\square[s,t])|\le C^\ell Z(\omega) \int_0^{d_1(s_1,t_1)}du_1\cdots\int_0^{d_\ell(s_\ell,t_\ell)}du_\ell
        \\\left(\log^{1/2}\frac1{\mu(B(s,u))}+\log^{1/2}\frac1{\mu(B(t,u))} \right)\,.
    \end{multline*}
    In particular, if $\EE Z$ is finite,  then
    \begin{align*}
        \EE\sup_{d_i(s_i,t_i)\le \delta_i, 1\le i\le \ell} |X(\square[s,t])|\le C^\ell(\EE Z)  \sup_{s\in\TT} \int_0^{\delta_1}du_1\cdots\int_0^{\delta_\ell}du_\ell\log^{1/2}\frac1{\mu(B(s,u))}\,.
    \end{align*}
\end{theorem}
\begin{proof}
    First note that for every $t,v$ in $\TT$
    \begin{align}\label{ineq.aniso}
        (\EE|X(t)-X(v)|^2)^{1/2}\le \sum_{i=1}^\ell d_i(t_i,v_i)\,.
    \end{align}
    We recall the notation $M_k(t)$ in \eqref{eqn.def.mk}. We have
    \begin{align*}
        \EE|X(t)-M_k(t)|
        &\le  \int_\TT \EE|X(t)-X(v)|\rho_k(t,v)\mu(dv)
        \\&\le \int_\TT \sum_{i=1}^\ell d_i(t_i,v_i)\rho_k(t,v)\mu(dv)
        \le \sum_{i=1}^\ell D_i2^{-k_i} \,.
    \end{align*}
    Together with Borel-Cantelli lemma, this shows for all $t\in\TT$, $M_k(t)$ converges to $X(t)$ almost surely. On the other hand, Theorem \ref{thm.maj.meas} shows for all $s,t\in\TT$, $M_k(\square[s,t])$ converges to a limit, denoted by $X'(\square[s,t])$. This implies $X(\square[s,t])=X'(\square[s,t])$ almost surely. The result is now followed from Theorem \ref{thm.maj.meas}.
\end{proof}

\subsection{Asymptotic growth}\label{subsec.asym.growth}
Let $W(t,x)$ be a continuous Gaussian process on $[0,T]\times\RR^d$ with mean 0. As in the previous subsection, we define the $d$-fold volumetric
\begin{equation*}
    d(x,y)=\sup_{t\in[0,T]}\left(\EE [W(t,\square[x,y])]^2\right)^{1/2}\,.
\end{equation*}
Without loss of generality, we assume there are metrics $d_1,\dots,d_d$ on $\RR$ such that $d^*(x,y)=d_1(x_1,y_1)\dots d_d(x_d,y_d)$ satisfies $d(x,y)\le d^*(x,y)$.

Let $\delta=(\delta_1,\dots,\delta_\ell)$ be in $(0,\infty)^\ell$. The notation $d^*(x,y)\le\delta$ means $d_i(x_i,y_i)\le \delta_i$ for all $i=1,2,\dots,\ell$. We denote $|x|^*=\max_{1\le i\le d} d_i(0,x_i)$ for every $x\in\RR^d$. We are interested in the asymptotic growth of the process
\begin{equation*}
    W^*(\delta,R)=\sup_{t\in [0,T]}\sup_{\substack{d^*(x,y)\le \delta\\ |x|^*, |y|^*\le R}}|W(t,\square[x,y])|
\end{equation*}
as $R$ gets large and $\delta$ can range freely in a bounded neighborhood of $0$. $W^*$ also depends on $T$.  However since $T$ will always be fixed in our consideration, we suppress the dependence on $T$ in our notations. We put
\begin{equation*}
    \SS_R=\{x\in\RR^d:|x|^*\le R\}\,,
\end{equation*}
\begin{equation*}
    m(\delta,R)=\EE W^*(\delta,R)\,,
\end{equation*}
and
\begin{equation*}
    \sigma(\delta,R)=\sup_{t\le [0,T]}\sup_{\substack{d^*(x,y)\le \delta\\ x,y\in\SS_R}}(E|W(t,\square[x,y])|^2)^{1/2}\,.
\end{equation*}
We first prove the following concentration inequality
\begin{lemma} For any $r>0$,
    \begin{equation}\label{ineq.con.wstar}
        P\left(\frac1{\sigma(\delta,R)}|W^*(\delta,R)-m(\delta,R)|>r\right)\le 2e^{-r^2/2}\,.
    \end{equation}
    As a consequence,
    \begin{equation}\label{ineq.e.psi.w}
        \EE \psi_{\rho}\left( \frac{|W^*(\delta,R)-m(\delta,R)|}{\sigma(\delta,R)}\right)\le c_{\rho}<\infty
    \end{equation}
    for every $\rho<1/2$, where $\psi_{\rho}=\exp(\rho x^2)$.
\end{lemma}
\begin{proof} It suffices to show \eqref{ineq.con.wstar}.
Let $\{X(u),u\in\TT\}$ be a Gaussian process. Assume that $\TT$ is finite. The following concentration inequality is standard
\begin{equation}\label{ineq.concentrate.x}
    P\left(\frac1{\sigma}\left|\sup_{u\in\TT} |X(u)|-\EE
    \left[ \sup_{u\in\TT} |X(u)|\right] \right|>r\right)\le 2e^{-r^2/2}\,,
\end{equation}
for every $\sigma\ge\sup_{u\in\TT}(EX^2(u))^{1/2}$. We refer to \cite{ledoux-book} or \cite[Theorem 5.4.3]{marcus-rosen} for a proof of \eqref{ineq.concentrate.x}. We now fix $(t_1,x_1),\dots,(t_m,x_m)$ in $[0, T]\times \RR^d$ such that $d^*(x_j,x_k)\le \delta$ and $|x_j|^*, |x_k|^*\le R$ for all $j,k$. We denote $x_j\shuffle x_k$ the collection of points $z$ in $\RR^d$ such that each component of $z$ is the corresponding component of either $x_j$ or $x_k$. We consider the centered
Gaussian random process $X(t_i,x_j\shuffle x_k) :=W(t_i,\square[x_j,x_k])$ indexed by the parameters $\{t_i\}_{1\le i\le m}$ and $\{x_j\shuffle x_k\}_{1\le j,k\le m}$. It is clear that
\begin{equation*}
    \EE X^2(t_i,x_j\shuffle x_k)\le \sigma^2(\delta,R)\,.
\end{equation*}
Thus, the inequality \eqref{ineq.concentrate.x} becomes
\begin{equation*}
    P\left(\frac1{\sigma(\delta,R)}\left|\sup_{i,j,k\le m}|W(t_i,\square [x_j,x_k])|-\EE\left[\sup_{i,j,k\le m}|W(t_i,\square [x_j,x_k])|\right]
    \right|>r \right)\le 2e^{-r^2/2}\,.
\end{equation*}
An approximation procedure yields \eqref{ineq.con.wstar}.
\end{proof}

\begin{theorem} With probability one,
    \begin{equation}\label{ineq.asymp}
        \sup_{\delta\in(0,1]^\ell}\limsup_{R\rightarrow \infty}\frac{|W^*(\delta,R)-m(2\delta,2R)|}{\sigma(2\delta,2R)\sqrt{\log(\delta_1^{-1}\cdots\delta_\ell^{-1}\log R)}}\le\sqrt2
    \end{equation}
\end{theorem}
\begin{proof} We put $p(\delta,R)=\delta_1^{-1}\cdots\delta_\ell^{-1}(\log R)^{2}$ and consider the random variable
\begin{equation*}
    \Theta=\sup_{\delta\in(0,1]^\ell,R\ge1}\frac1{p(\delta,R)}\psi_{\rho}\left( \frac1{\sigma(2\delta,2R)}|W^*(\delta,R)-m(2\delta,2R)|\right)\,.
\end{equation*}
For each multi-index $j=(j_1,\dots,j_\ell)$ in $\NN^\ell$, we denote $2^{-j}=(2^{-j_1},\dots,2^{-j_\ell})$. The notation $\delta\le 2^{-j}$ means $\delta_i\le 2^{-j_i}$ for all $i=1,2,\dots,\ell$. Then using the monotonicity of $p$, $\psi_{\rho}$, $W^*$ and $\sigma$, and \eqref{ineq.e.psi.w} we have
\begin{align*}
    \EE\Theta&\le\sum_{k\in\NN,j\in\NN^\ell}\EE\sup_{\substack{2^{-j-1}\le\delta\le 2^{-j}\\2^{k-1}\le R\le 2^k }}\frac1{p(\delta,R)}\psi_{\rho}\left( \frac1{\sigma(2\delta,2R)}|W^*(\delta,R)-m(2\delta,2R)|\right)
    \\&\le\sum_{k\in\NN,j\in\NN^\ell}\frac1{p(2^{-j},2^{k-1})}\EE\psi_{\rho}\left(\frac1{\sigma(2^{-j},2^k)}|W^*(2^{-j},2^k)-m(2^{-j},2^k)|\right)
    \\&\le c_{\rho}\sum_{k\in\NN,j\in\NN^\ell}\frac1{p(2^{-j},2^{k-1})}<\infty\,.
\end{align*}
  Hence, with probability one, $\Theta $ is finite and
  \begin{equation*}
   \psi_{\rho}\left( \frac1{\sigma(2\delta,2R)}|W^*(\delta,R)-m(2\delta,2R)|\right)\le \Theta p(2\delta,R) \,,\quad \forall \delta>0\,,\ \forall R\ge1\,.
  \end{equation*}
In particular,
  \begin{equation*}
    \frac{|W^*(\delta,R)-m(2\delta,2R)|}{\sigma(2\delta,2R)}\le \sqrt{\frac{\log[\Theta p(2\delta,R)]}{\rho} }\,,\quad \forall \delta>0\,,\ \forall R\ge1\,.
  \end{equation*}
  We then use the trivial estimate
  \begin{equation*}
    \sqrt{\log(\Theta p)}\le\sqrt{|\log \Theta |}+\sqrt{|\log p|}
  \end{equation*}
  to get
  \begin{equation*}
    \frac{|W^*(\delta,R)-m(2\delta,2R)|}{\sigma(2\delta,2R)\sqrt{\log(\delta_1^{-1}\cdots\delta_\ell^{-1}\log R)}}\le\sqrt{\frac{ |\log \Theta|}{\rho {|\log\log R|} }} +\sqrt{\frac{\log(\delta_1^{-1}\cdots\delta_\ell^{-1}(\log R)^{2})}{\rho\log(\delta_1^{-1}\cdots\delta_\ell^{-1}\log R)}} \,,
  \end{equation*}
  for all $\delta>0$ and $R\ge1$. Since $\rho$ can be chosen to be any constant less than $1/2$, we can choose a  sequence $\rho_n$ convergent to $1/2$.  Since countable unions of events with probability zero still have probability zero,  we can pass through the limit $n\to\infty$ to get, with probability one,
\begin{equation*}
    \frac{|W^*(\delta,R)-m(2\delta,2R)|}{\sigma(2\delta,2R)\sqrt{\log(\delta_1^{-1}\cdots\delta_\ell^{-1}\log R)}}\le\sqrt{\frac{ 2|\log \Theta|}{ {|\log\log R|} }} +\sqrt{\frac{2\log(\delta_1^{-1}\cdots\delta_\ell^{-1}(\log R)^{2})}{\log(\delta_1^{-1}\cdots\delta_\ell^{-1}\log R)}}\,,
  \end{equation*}
  for all $\delta>0$ and $R\ge1$. Finally, let $R\to\infty$ to complete the proof.
\end{proof}

% Often the case, the term $m(2\delta,2R)$ grows slowlier than $\sigma(2 \delta,2R)\sqrt{|\log\log R|}$ as $R\to\infty$. Hence, the factor $m(2\delta,2R)$ in \eqref{ineq.asymp} can be dropped. The following result gives a simple criteria for this assertion.

In general, it is hard to say anything about the growth of $m(\delta,R)$ as $R$ gets large. In what follows, we restrict ourselves to  a particular (but still sufficiently large) class of Gaussian random fields. To be more precise, for each $i=1,\dots,\ell$, let $\phi_i$ be a majorant for $d_i$, that is, $\phi_i$ is strictly increasing with $\phi_i(0)=0$ and
\begin{equation}\label{eqn.def.phii}
    d_i(x_i,y_i)\le \phi_i(|y_i-x_i|)\,.
\end{equation}
Define
\begin{equation*}
    \tilde\omega_i(\delta_i)=\delta_i\log^{1/2}\frac1{\phi_i^{-1}(\delta_i)}+\int_0^{\phi_i^{-1}(\delta_i)}\frac{\phi_i(u)}{u\log^{1/2}(1/u)}du\,.
\end{equation*}
We will always presume $\tilde\omega_i$'s are finite wherever they appear.
\begin{proposition} Denote  $\tilde\delta_i=\prod_{j\neq i}\delta_j$. Then we have
    \begin{equation}\label{ineq.m.delta.r}
        m(\delta,R)\lesssim \delta_1\cdots \delta_\ell\log^{1/2}\left(\prod_{i=1}^\ell 2\phi_i^{-1}(R)\right)+\sum_{i=1}^\ell \tilde\delta_i\tilde\omega_i(\delta_i)
    \end{equation}
    where the implied constant is independent of $R$ and $\delta$.
\end{proposition}
\begin{proof}
    We take for the majorizing measure $\mu_i=\lambda/(2 \phi_i^{-1}(R))$,
    where $\lambda$ is the Lebesgue measure.   By \eqref{eqn.def.phii}, the ball $B^i(x_i,u_i)$ contains the interval $(x_i-\phi_i^{-1}(u_i),x_i+\phi_i^{-1}(u_i))\cap \{z_i:d_i(z_i,0)\le R\}$, thus,
    \begin{equation*}
        \mu_i(B^i(x_i,u_i))\ge\frac{\phi_i^{-1}(u_i)}{2\phi_i^{-1}(R)}\,.
    \end{equation*}
    Hence, for all $x$ in $\SS_R$,
    \begin{equation*}
        \log\frac1{\mu(B(x,u))}\le\log\left(\prod_{i=1}^d \frac{2\phi_i^{-1}(R)}{\phi_i^{-1}(u_i)}\right)\,.
    \end{equation*}
    Therefore, for $\delta$ sufficiently small,
    \begin{align*}
        &\int_0^{\delta_1}du_1\cdots\int_0^{\delta_\ell} du_\ell\log^{1/2}\frac1{\mu(B(x,u))}
        \\&\le\int_0^{\delta_1}du_1\cdots\int_0^{\delta_\ell} du_\ell\log^{1/2}\left(\prod_{i=1}^d \frac{2\phi_i^{-1}(R)}{\phi_i^{-1}(u_i)}\right)
        \\&\le \delta_1\cdots \delta_\ell \log^{1/2}\left(\prod_{i=1}^\ell 2\phi_i^{-1}(R)\right)+ \int_0^{\delta_1}du_1\cdots\int_0^{\delta_\ell} du_\ell\log^{1/2}\left(\prod_{i=1}^d \frac{1}{\phi_i^{-1}(u_i)}\right)\,.
    \end{align*}
    The last integral in the above formula can be estimated as following
    \begin{align*}
        &\int_0^{\delta_1}du_1\cdots\int_0^{\delta_\ell} du_\ell\log^{1/2}\left(\prod_{i=1}^d \frac{1}{\phi_i^{-1}(u_i)}\right)
        \\&=\int_0^{\phi^{-1}_1(\delta_1)}d\phi_1(u_1)\cdots\int_0^{\phi^{-1}_\ell(\delta_\ell)} d\phi_\ell(u_\ell)\log^{1/2}\left(\prod_{i=1}^d \frac{1}{u_i}\right)
        \\&\le\sum_{i=1}^\ell \tilde\delta_i\int_0^{\phi_i^{-1}(\delta_i)}\log^{1/2}(1/{u_i})d \phi_i(u_i)\,.
    \end{align*}
    Using integration by parts, $\int_0^{\phi_i^{-1}(\delta_i)}\log^{1/2}(1/{u_i})d \phi_i(u_i)\le \tilde\omega(\delta_i)$, which completes the proof.
\end{proof}
\begin{example}
  Let $W=(W(x),x\in\RR^d)$ be a factional Brownian sheet with Hurst parameter
   $H=(H_1,\dots,H_d)\in(0,1)^d$. In particular, the covariance of $W$ is given by
\begin{equation*}
  \EE W(x)W(y)=\prod_{i=1}^d R_{H_i}(x_i,y_i)
\end{equation*}
where
\begin{equation*}
  R_{H_i}(s,t)=\frac12 (|s|^{2H_i}+|t|^{2H_i}-|s-t|^{2H_i})\,.
\end{equation*}
We see that
\begin{equation*}
  (\EE|W(\square[x,y])|^2)^{1/2}=\prod_{i=1}^d |x_i-y_i|^{H_i}\,,
\end{equation*}
thus $\phi_i(\delta_i)=|\delta_i|^{H_i}$ and $\sigma(\delta,R)=\delta_1\cdots \delta_d$. We put
\begin{equation*}
  m(\delta,R)=\EE\sup|W(\square[x,y])|\,.
\end{equation*}
where the supremium is taken over the domain $\{x,y:|x_i|^{H_i},|y_i|^{H_i}\le R \mbox{ and } |x_i-y_i|^{H_i}\le \delta_i\,\forall 1\le i\le d\}$. Note that
\begin{equation*}
    \tilde\omega_i(\delta_i)\lesssim\delta_i\log^{1/2}\frac1{\phi_i^{-1}(\delta_i)}
\end{equation*}
The bound \eqref{ineq.m.delta.r} yields
\begin{equation*}
    m(\delta,R)\lesssim \delta_1\cdots \delta_d\sqrt{\log (R\delta_1^{-1}\cdots\delta_d^{-1})}\,.
\end{equation*}
Theorem \ref{thm.maj.gaussian}  yields
\begin{equation}
    \sup_{|x_i|^{H_i},|y_i|^{H_i}\le R; d^*(x,y)\le \delta}|W(\square[x,y])|\lesssim \delta_1\cdots\delta_d\sqrt{\log (R\delta_1^{-1}\cdots\delta_d^{-1})}\,,
\end{equation}
when $R$ get large.   This implies the inequality of the form
\eqref{eq:cond.w2}   for $W$.

\end{example}
\begin{remark}
    Fractional Brownian sheet belongs to a larger class of random fields called anisotropic random fields. That is,
    \begin{equation}\label{eqn.cond.aniso}
        (\EE |W(y)-W(x)|^2)^{1/2}\asymp\sum_i |y_i-x_i|^{H_i}\,.
    \end{equation}
    These random fields may have different behavior along different directions. In \cite{mwy}, the authors investigate the global moduli of continuity for anisotropic Gaussian random fields. As a result, they establish a sharp result for the global modulus of continuity for fractional Brownian sheets. The conditions considered in the current paper are somewhat more general. For instance, the estimate \eqref{ineq.aniso} implies the upper bound in the anisotropic condition \eqref{eqn.cond.aniso}.  We believe our method (Theorems \ref{thm.maj.meas}, \ref{thm.maj.gaussian}) provides similar results as \cite{mwy} though we do not report them here.
\end{remark}

\begin{appendix}
\renewcommand{\theequation}{\Alph{section}.\arabic{equation}}
\section{Other types of nonlinear stochastic  integral}\label{sec.prelim}

{\replacecolorred
The It\^o  integral is
a fundamental concept in     stochastic analysis.  This integral can be defined
under less condition than the Stratonovich one and has a completely
different  feature   such as the  famous It\^o formula.
%,  compared to the Newton-Leibniz  calculus.
From the modeling point of view,  It\^o type stochastic differential equations are more popular since all terms in  the It\^o equation $dx_t=b(x_t) dt+\si(x_t) \de B_t$ (see also \eref{eqn.diffusion}) have clear meaning: $b(x_t) $ represents the mean rate of change and $\si(x_t) \de B_t$ represents the fluctuation  (it has zero mean contribution).

In this section,  we will introduce    nonlinear It\^o-Skorohod integral.  This integral is  a   probabilistic one  and is defined for almost every sample path while nonlinear Young integral is defined for every sample path. The relation between these two integral is through the nonlinear symmetric (Stratonovich) integral.

%One of our future   goals  is to
%study   the Feynman-Kac formula for  It\^o-Skorohod type stochastic heat
%equations: $\partial _t u(t,x)+Lu(t,x) +u(t,x)\diamond \partial _tW(t,x)=0$,  where
%$\diamond$ denotes the Wick product
%(the term $u(t,x)\diamond \partial _tW(t,x)$ has a zero mean contribution and is an analogue of It\^o differential).
%This type of equations have been studied in
% \cite{hu-hu-nu-ti}, \cite{hu-lu-nualart12}, \cite{hunualartsong}  and
% the references therein when the coefficients in $L$ is independent of
% $W$. The Feynman-Kac formula for the moments can be obtained under less conditions.
}

\subsection{Nonlinear It\^o-Skorohod integral}\label{subsection.noise}
	Let $H\in (\frac12,1)$ and   denote by $R_H(s,t)
	=\frac12\left(s^{2H}+t^{2H}-|t-s|^{2H}\right)$ the covariance function of
	a fractional Brownian motion of Hurst parameter $H$.   %the covariance
	%function of fractional Brownian motion with Hurst parameter $H$.
	Let $q(x,y)$ be a continuous and positive definite function, namely,
	for any $x_i\in \RR^d\,, \ i=1, 2, \cdots,
	m$ and complex numbers $\xi_i, i=1, 2\cdots, m$,  not all $0$, we have
	\[
	\sum_{i,j=1}^m q(x_i, x_j) \bar \xi_i\xi_j\ge 0\,,
	\]
	where $\bar \xi_i$ is the conjugate number of $\xi_i$. For every $s,t\ge0$ and $x,y\in\RR^d$, we denote
	\[
	Q(s,t,x,y)=\frac{\partial ^2 R_H}{\partial s\partial
	t}(s,t)q(x,y)=\alpha_H |s-t|^{2H-2} q(x,y)\,,
	 \]
	 where $\alpha_H=H(2H-1)$.
	Let  $\SS$ be  the
	set of all smooth functions $f: [0,T]\times \RR^d\rightarrow \RR$ such that $f(t,\cdot)$ has
	compact support for every $t\in[0,T]$.  We introduce a scalar product on $\SS$ in the
	following way:
	\begin{equation}\label{inn.hh}
	\sprod{\phi}{\psi}_\HH= \int_{[0, T]^2\times \RR^{2d}}
	\phi(s,x)\psi(t,y)Q(s,t,x,y) d xd yd sd t\,.
	\end{equation}
	% We will assume that $Q$ is symmetric, positive definite and its derivative
	%\[Q''=\frac{\partial^{2d}Q}{\partial x_1\partial y_1\cdots\partial
	%x_d\partial y_d}\] exists and  locally integrable in $\RR^{2d}$ with
	%respect to the Lebesgue measure. We also assume $Q$ satisfies the
	%inverse Cauchy-Schwarz inequality
	%\begin{equation}\label{ineq.in.bcs}
	%Q(x,x)^{\frac12}Q(y,y)^{\frac12}\le c_Q^2|Q(x,y)|,
	%\end{equation} for some constant $c_Q\le 1$.
	%
	%For suitable functions
	%$\phi$ and $\psi$ we will denote
	%\begin{equation}\label{inn.q}
	%\sprod{\phi}{\psi}_Q=\int_{\RR^d}\!\int_{\RR^d}\! \phi(x)\psi(y)Q(x,y)\d x\d y.
	%\end{equation}
	%With this notation, the inner product $\sprod{\cdot}{\cdot}_\HH$ can be written as
	%\begin{equation*}
	 %\sprod{\phi}{\psi}_\HH=\alpha_H\int_0^T\!\int_0^T\!\sprod{\phi(s,\cdot)}{\psi(t,\cdot)}_Q|s-t|^{2H-2}\,\d s\d t.
	%\end{equation*}

	We denote by $\HH$ the Hilbert space
	of the closure of $\SS$ with respect to this inner
	product.   {\replacecolorred  Let $T$ be a bijective Hilbert-Schmidt  operator
	on $\HH$.  Define the Banach  space (in fact, it is a Hilbert space)
	$\Omega$ as the completion of $\HH$ with respect to the norm $\|x\|_\Om:=\sqrt{\langle Tx\,,Tx\rangle_\HH}$.
	Then, it follows from the Bochner-Minlos theorem
 	(see \cite{hidabrownianmotion},  Theorem 3.1)
%	Banach space $\Omega$ }  such that the embeddings
%	$\Omega'\subset\HH'=\HH\subset \Om$ are continuous and
} that there is a
	probability measure $P$ on $(\Omega, \mathcal F)$ such that  $\langle h,\omega\rangle
	$ is a centered Gaussian random variable  with covariance $\EE
	\left[\langle h,\cdot\rangle \langle h',\cdot\rangle \right]=
	\langle h, h'\rangle_\HH$\,,   $\forall \ h, h'\in \Om'$,  where $\Omega'$ is the Banach space
	of all continuous linear functionals on $\Omega$\,;  \ $\cF$
	is the Borel $\si$-algebra generated by the open sets of $\Om$,    and
	$\langle h,\om\rangle $  the pairing between $h\in \Omega'\subset \HH$
	and $\Omega$.  We identity  	$\HH'=\HH$  so that the embeddings
 $\Omega'\subset\HH'=\HH\subset \Om$ are continuous.
	We can define Gaussian random variable $\langle
	h,\omega\rangle$ for all $h\in \HH$ by limiting  argument.

	First we give some specific elements in $\cH$.  For any $x\in \Rd$.
	we denote by $\de_x$  the  Dirac function on $\Rd$.  Namely,
	 $\de_x$ is defined by $\int_{\Rd} \de_x(y) f(y) dy=f(x)$
	 for any smooth function of compact support on $\Rd$.
	 \begin{proposition}\label{lem.delta}  For any $s>0$ and $x\in \Rd$,
	$I_{(0,s]}\delta_{x}$ is an element in $\cH$ and
	\begin{equation}\label{eqn.app1}
	\sprod{I_{(0,s]}\delta_{x}}{I_{(0,t]}\delta_{y}}_\HH=R_H(s,t)q(x,y)
	\end{equation}
	and
	\begin{equation}\label{eqn.app2}
	\|I_{(0,s]}\delta_{x}-I_{(0,t]}\delta_{y}\|^2_\HH\\=s^{2H}q(x,x)
	+t^{2H}q(y,y)-2R_H(s,t)q(x,y).
	\end{equation}
	\end{proposition}
	\begin{proof} For every $\ep>0$ and $x\in \Rd$,  we denote the elementary function
	\[
	\delta^\ep_{x}=(2\ep)^{-d}I_{(x-\ep,x+\ep]}.
	\]
	If $\ep$ tends to 0, the function $\delta^\ep_{x}$ converges in
	$\HH$ to the  generalized function $\delta_{x}$.
	% Moreover,
	%for every $(s,x)$ and $(t,y)$ in $[0,T]\times \Rd$ we have
	Indeed, fix $(s,x)$ and $(t,y)$ in $[0,T]\times \Rd$.  For any
	positive numbers $\ep$ and $\ep'$, we have
	\[
	\sprod{I_{(0,s]}\delta_{x} ^\vare}{I_{(0,t]}\delta_{y}^{\vare'} }_\HH
	=R_H(s,t)(4\ep\ep')^{-d}\int_{y-\ep'}^{y+\ep'}\!\int_{x-\ep}^{x+\ep}\!q(x',y')d	x'd y'\,.
	\]
	Since $q(\cdot,\cdot)$ is continuous,  the above
	right hand side   converges to $q(x,y)$ as $\ep$ and $\ep'$
	tend to 0.   This shows easily that $I_{(0,s]}\delta_{x} ^\vare$
	is a Cauchy sequence in $\cH$ when $\vare\rightarrow 0$.  The limit
	of $I_{(0,s]}\delta_{x} ^\vare$ in $\cH$ as $\vare\rightarrow 0$
	is  $I_{(0,s]}\delta_{x}  $.  The equations \eref{eqn.app1}
	and \eref{eqn.app2} are immediate.
	%Hence for every sequence $\ep_n$ converging to 0, the
	%sequence $\delta^{\ep_n}_{x_0}$ is convergent in $\HH$ and this
	%limit does not depend on the sequence $\{\ep_n\}$. We denote this
	%limit by $\delta_{x_0}$. Moreover, for any elementary function
	%$\phi$ in $\rE$ we have
	%\begin{align*}
	%\sprod{I_{(0,s_0]}\delta_{x_0}}{\phi}_\HH &=\lim_{\ep\to
	 %0}\sprod{I_{(0,s_0]}\delta^\ep_{x_0}}{\phi}_\HH\\&=\ah\!\int_0^T\!\int_0^{s_0}\!\int_{\RR^{d}}\!\phi(t,y)q(x_0,y)|s-t|^{2H-2}\d
	%y\d s\d t.
	%\end{align*}
	%Choose $\phi=I_{(0,t_0]}\delta^\ep_{y_0}$ and let $\ep$ goes to 0,
	%we obtain \eqref{eqn.app1}. The identity \eqref{eqn.app2} is a
	%direct consequence of \eqref{eqn.app1}.
	\end{proof}
	Since $I_{(0,s]}\delta_{x}\in \cH$,  we can define
	\begin{equation}\label{map.b}
	W(s,x, \om)= \langle I_{(0,s]}\delta_{x}, \om\rangle\,, \ \  \om \in
	\Om
	\end{equation}
	Thus $\{W(s, x), t\ge 0\,, x\in \Rd\}$
	 is a multiparameter centered Gaussian process with the
	following covariance
	\[
	\EE \left[ W(s, x)W(t,
	y)\right]=\sprod{I_{(0,s]}\delta_{x}}{I_{(0,t]}\delta_{y}}_\HH=R_H(s,t)q(x,y)\,.
	\]
	We also denote
	\[
	W(\phi):=\int_0^T\!\int_{\RR^d}  \phi(s,x)W(d s,x)d x:=\langle \phi,
	\om\rangle\quad \forall \ \phi\in \HH\,.
	\]
	%As a corollary, we obtain
	%\begin{lemma}
	%Let $W$ be the map in \eqref{map.b}. Then
	%$\{W(s,x)=W(I_{(0,s]}\delta_{x})\}$ is a centered Gaussian field
	%with covariance
	%\[\EE(W(s,x)W(t,y))=R_H(s,t)q(x,y).
	%\]
	%\end{lemma}

	%From Lemma \ref{lem.delta},  the space $\HH$  contains the functions
	%of the  form $I_{(0,s]}\delta_{x}$:
	%\begin{equation}\label{eqn.dirach}
	%\sprod{I_{(0,s]}\delta_{x}}{I_{(0,t]}\delta_{y}}_\HH=R_H(s,t)Q(x,y),
	%\end{equation}
	%for any $(s,x)$ and $(t,y)$ in $[0,T]\times \Rd$.  We define

	%
	%We can find a linear space of functions
	%contained in $\HH$ in the following way. Let $|\HH|$ be the linear space of measurable functions $\phi$ on $[0,T]\times
	%\Rd$ such that
	%\begin{equation}\label{eqn.norm}
	%\|\phi\|^2_{|\HH|}=\ah\int_{0}^T\!\int_{0}^T\!\int_{\RR^{2d}}{|\phi(s,x)|}
	%{|\phi(t,y)|}|s-t|^{2H-2}\d x\d y\d s\d t<\infty.
	%\end{equation}

%\subsection{Malliavin calculus}
	We denote by $\cP$ the set of smooth and cylindrical random variables
	of the following form
	\begin{equation}\label{eqn.ss}
	F=f(W(\phi_1),\dots,W(\phi_n)),
	\end{equation}
	$\phi_i\in \HH$, $f\in C_p^\infty(\RR^n)$ ($f$ and all its partial
	derivatives have polynomial growth). $D$  denotes   the  Malliavin
	derivative. That is, if $F$ is of  the form \eqref{eqn.ss}, then
	$DF$ is the $\HH$-valued random variable defined by
	\begin{equation*}
	DF=\sum_{j=1}^n\frac{\partial f}{\partial x_j}(W(\phi_1),\dots,W(\phi_n))\phi_j.
	\end{equation*}
	The operator $D$ is closable from $L^2(\Omega)$ into
	$L^2(\Omega;\HH)$ and we define the Sobolev space $\DD^{1,2}$ as the
	closure of  $\cP$  under the norm
	\[
	\|F\|_{1,2}=\sqrt{\EE(F^2)+\EE(\|DF\|_\HH^2)}.
	\]
	$D$ can be extended uniquely to an operator  from $\DD^{1,2}$ into
	$L^2(\Omega;\HH)$.   The {\it divergence operator}  $\delta$ is the
	adjoint of the Malliavin derivative operator $D$. We say that a
	random variable $u$ in $L^2(\Omega;\HH)$ belongs to the domain of
	the divergence operator, denoted by $\dom \delta$, if there is a
	constant $c_u\in (0, \infty)$ such that
	\[
	|\EE(\sprod{DF}{u}_\HH)|\le c_u\|F\|_{L^2(\Omega)}\quad \forall \ F\in
	\DD^{1,2}\,.
	\]
	 In this case $\delta(u)$ is defined by the duality
	relationship
	\begin{equation}\label{eqn.dual}
	\EE(\delta(u)F)=\EE(\sprod{DF}{u}_\HH)\quad \forall \ F\in \DD^{1,2}\,.
	\end{equation}

	The following are two basic properties of the divergence operator
	$\delta$.
	\begin{itemize}
	  \item[(i)] $\DD^{1,2}(\HH)\subset \dom\delta$ and for any
	  $u\in\DD^{1,2}(\HH) $
	  \begin{equation}\label{id.covdel}
	  \EE\left(\delta(u)^2\right)=\EE\left(\|u\|_\HH^2\right)
	  +\EE\left(\sprod{Du}{(Du)^*}_{\HH\otimes\HH}\right),
	  \end{equation}
	  where $(Du)^*$ is the adjoint of $Du$ in the Hilbert space
	  $\HH\otimes\HH$.
	    \item[(ii)] For any $F$ in $\DD^{1,2}(\HH)$ and any $u$ in the
	    domain of $\delta$ such that $Fu$ and
	    $F\delta(u)-\sprod{DF}{u}_{\HH}$ are square integrable, then $Fu$
	    is in the domain of $\delta$ and \begin{equation}\label{id.fu}
	    \delta(Fu)=F\delta(u)-\sprod{DF}{u}_{\HH}.
	    \end{equation}
	\end{itemize}

	The operator $\delta$ is also called the Skorokhod integral because
	in the case of Brownian motion, it coincides with the generalization
	of the It\^o stochastic integral to anticipating integrands
	introduced by Skorokhod \cite{skorohod}. On the relation between $\delta$ and $D$, we have the identity
	\begin{equation}\label{id.comm}
	    D \delta(u)=u+\delta(Du)\,.
	 \end{equation}
	We refer to
	Nualart's book \cite{nuabook} for a detailed account of the
	Malliavin calculus with respect to a Gaussian process. Using the
	specific definition of our $\cH$, we also denote
	$\de(u)=\int_0^T \int_{\RR^d} u(t, x) W(\delta t, x)d x$. In addition, we can write the identity
	\eref{id.covdel} as
	\begin{align}
	&\EE \left[ \int_0^T \int_{\Rd}u(t, x) W(\delta t, x)d x\right]^2\nonumber\\
	=& \int_{[0, T]^2\times \RR^{2d} }
	\EE \left[ u(t,x)u(s,y)\right]
	 Q(s,t,x,y) dsdtdxdy\nonumber\\
	  +&\int_{[0, T]^4\times \RR^{4d}} \EE \left[ D_{t_2, x_2} u(t_1,x_1)
	 D_{s_2, y_2}u(s_1,y_1)\right] Q(t_1,s_2, x_2,y_1) Q(t_2,s_1, x_1,y_2)
	% R_H(t_1,s_2)R_H(t_2,s_1)q(x_2,y_1) q(x_1,y_2)
	 ds dt dx dy \,, \nonumber\\ \label{ito-iso}
	 \end{align}
	where in the rest of the paper we shall use $ds=ds_1\cdots ds_k$, $dx=dx_1\cdots dx_m$ and so on,
	the $k$ and $m$ being clear in the context.

%\setcounter{equation}{0}
%\subsection{Nonlinear It\^o-Skorohod  integral}\label{sec.sto.int}
	Let $\{ W(t,x)\,, t\ge 0\,, x\in \Rd\}$ be the Gaussian field
	introduced in Section \ref{subsection.noise},  whose mean is $0$ and
	whose covariance is
	\[
	\EE (W(s,x)W(t,y))=R_H(s,t) q(x,y)\,.
	\]
	Let  $\varphi=\{\varphi_t, t\in [0,T]\}$   be a $\Rd$-valued
	stochastic  process. Our aim in this section is to introduce and
	study the nonlinear stochastic integral $\int_0^T W(\delta t, \varphi_t)$.

	This stochastic integral was studied  earlier in order  to establish   the
	Feynman-Kac formula when $\varphi_t$ is a Brownian motion, independent
	of $W$.   The case $H>1/2$ is discussed in \cite{hunualartsong} and
	the case $H<1/2$ is discussed in \cite{hu-lu-nualart12}.  When $\{
	W(t,x)\,, t\ge 0\} $ is a semimartingale with respect to $t$ (for
	fixed $x\in \RR^d$), this type of stochastic integral  has been
	studied extensively  and generalized
	It\^o formulas have been established.  It has been applied to solve
	some stochastic partial differential equations. See for instance Kunita's book
	\cite{kunita} and the references therein.

	In this section, we will define the stochastic integral $\int
	W(\delta t,\varphi_t)$  based on the covariance structure of $W$. This
	method is closely  tied to the nature of $W$ as a Gaussian process.
	In particular, we introduce here two types of stochastic integrals,
	namely, the divergence type and symmetric type. We also study their
	properties and relation. The divergence type integral turns out to
	have zero mean, thus one can think of it as a generalization of
	It\^o-Skorohod integral. The symmetric integral does not have
	vanishing mean and differs from the divergence type integral by a
	correction term, related to the Malliavin derivative of some random
	variable. One can also view the symmetric integral as a
	generalization of Stratonovich integral.

{\replacecolorred  We shall  define the (nonlinear) It\^o-Skorohod (divergence)
	type integral $\int_0^T W(\delta t, \varphi_t)$ by
  the (linear) multi-parameter integral  $  \int_0^T\int_{\RR^d} \de(\varphi_t-y) W(\delta t,y)d y$.  Here and in the remaining part of the paper, the symbol $\de$ carries two meanings:
   the It\^o-Skorohod integral and the Dirac delta
	function. Difference between the two meanings will be clear from the context.

	Since $\de(\varphi_t-y) $ is a distribution valued random process,
	to define its stochastic integral we need to approximate the Dirac delta function by smooth functions.  Namely, we shall  define
	$  \int_0^T\int_{\RR^d} \de(\varphi_t-y) W(\delta t,y)d y$
 as  the limit $\displaystyle
	\lim_{\vare\downarrow 0}\int_0^T\int_{\RR^d} \eta_\vare(\varphi_t-y)
	W(\delta t, y)d y$,  where $\eta_\vare$ is an approximation of the Dirac delta function $\de$.  To define such sequence $\eta_\vare$, we
	}
	denote by $\eta$ the following bump function
	\[\eta(x)=c_d \exp\{(|x|^2-1)^{-1}\}1_{\{|x|<1\} }\,, x\in \Rd\,,
	\]
	where $|x|$ is the Euclidean distance in $\Rd$ and $c_d$ is the
	positive  constant so that $$\int_{\RR^d}\eta(x)dx =1.$$  The function
	$\eta$ is smooth and compactly supported. Its corresponding
	mollifier is
	\begin{equation}\label{def.eta.ep}
	\eta_\ep(x)=\ep^{-d}\eta\left(\frac x\ep\right).
	\end{equation}
	%For
	% any measurable function $\phi:\RR^d\to \RR$, we define the mollifier
	% of $\phi$ as \[\phi_\ep=\phi*\eta_\ep\] where $*$ is the convolution
	% operator. For $1\le p<\infty$, if $\phi$ belongs to $L^p(\Rd)$, then
	% $\phi_\ep$ converges to $\phi$ in $L^p(\Rd)$ as $\ep$ tends to 0.

%\subsection{Divergence-type stochastic integral}
	
	Here is our definition.
	\begin{definition}\label{def.itosko} Let $\varphi:[0, T]\times
	\Om\rightarrow \Rd$ be a measurable stochastic process.  If $I_\vare= \int_0^T\int_{\RR^d}
	\eta_\vare(\varphi_t-y) W(\delta t, y)d y$ is well-defined and it has a
	limit in $L^2(\Om, \cF, P)$ as $\ep\to0$,  then we define $\int_0^T W(\delta t,
	\varphi_t) $   as the aforementioned limit.
	\end{definition}

	Next, we shall give condition to ensure the existence of the
	stochastic integral  $\int_0^T W(\delta t, \varphi_t) $,
	{\replacecolorred  namely, to ensure the existence of the limit of
	$I_\vare$ in $L^2(\Om, \cF, P)$}.  To express the
	conditions in a more concise way we introduce the following
	notations.
	\begin{equation*}
	q_\varphi(x,y)
	% =\ah\ttint{\EE
	% q(x+\varphi_s,y+\varphi_t)|s-t|^{2H-2}}{s}{t}
	=\alpha\int_0^T\int_0^T \EE q(x+\varphi_s,y+\varphi_t)|s-t|^{2H-2}dsdt
	\end{equation*}
	and
	\begin{multline*}
	q^*_{D\varphi}(x,y)=\alpha_H^2\int\limits_{[0,T]^4\times\RR^{2d}}\!
	\EE D_{s_1,x'}q(x+\varphi_{s_2},y') D_{t_2,y'}
	q(x',y+\varphi_{t_1})\\
	|s_1-t_1|^{2H-2}|s_2-t_2|^{2H-2}d {s_1}d {s_2}d {t_1}d {t_2}d x'd y'
	\end{multline*}
	whenever the integrals on the right hand side make    sense.
	We make the following assumptions on the process $\varphi_t$.
	\begin{enumerate}[label={ \bm{$(A\arabic*)$}}]
	\item\label{cond.qphi0} $\varphi_t$ belongs to $\DD^{1,2}$ for all $t$,
	{\replacecolorred  and for almost every $\om\in \Om$,  the sample path
	$\varphi_t$  is continuous in $t\in [0, T]$}.
	\item\label{cond.qphi1} $|q|_\varphi$ is integrable on a
	  neighborhood of $(0,0)$, that is there exists an open set $U$ in
	$\RR^{2d}$ containing $(0,0)$ such that
	\[\int_{U}
	\int_0^T\int_0^T  \EE |Q(s, t,  x+\varphi_s,y+\varphi_t)|dsdt dxdy<\infty\,.
	\]
	% \item\label{cond.qphi2} $q_\varphi(0,0)$ exists and finite,
	\item\label{cond.qphi3} $q_\varphi(x,y)$ is well-define  in neighborhood of
	 $(0, 0)$  and
	it  is continuous at $(0,0)$.
	\item\label{cond.qphi4} There exists an open set $U$
	in $\RR^{2d}$ containing $(0,0)$ such that
	\begin{multline*}
	\int\limits_{U}\int\limits_{[0,T]^4\times\RR^{2d}}\!
	\EE\left|D_{s_1,x'}q(x+\varphi_{s_2},y') D_{t_2,y'}
	q(x',y+\varphi_{t_1})\right|\\
	|s_1-t_1|^{2H-2}|s_2-t_2|^{2H-2}d {s_1}d {s_2}d {t_1}d {t_2}d
	x'd y'd xd y<\infty\,.
	\end{multline*}
	% \begin{eqnarray*}
	% &&\int\limits_{U}\int\limits_{[0,T]^4\times\RR^{2d}}\!
	% \EE\big|D_{s_1,x_2}Q(s_1, t_1, x+\varphi_{s_2},y') \\
	% &&\qquad \qquad D_{t_2,y'}
	% Q(s_2, t_2, x' ,y+\varphi_{t_1})\big|  \ds\dt\dx'\dy'<\infty
	% \end{eqnarray*}
	 \item\label{cond.qphi5}  $q^*_{D\varphi}(x,y) $ is well-defined  in neighborhood of
	 $(0, 0)$  and it
	  is  continuous at $(0,0)$.
	\end{enumerate}
	% \begin{remark}
	% The conditions A1-A3 are not artificial. Indeed, in the paper
	% \cite{davie} of Davie, the author showed the following estimate
	% \begin{equation*}
	% \EE
	% \end{equation*}
	% \end{remark}

	\begin{theorem}\label{thm.itosko} We assume the conditions \ref{cond.qphi0}-\ref{cond.qphi5} are satisfied.
	Then $\int_0^T\! W(\delta t, \varphi_t) $ is well-defined and
	\begin{multline}\label{id.ddel.phi}
	\EE\left[ \int_0^T\! W(\delta t, \varphi_t) \right]^2 =
	q_{D\varphi}^*(0,0)+q_\varphi(0,0)\\=
	\int_{[0,T]^4}\!\int_{\RR^{2d}}\EE
	D_{s_1,x}Q(s_1,t_1,\varphi_{s_2},y)
	D_{t_2,y}Q(s_2,t_2,x,\varphi_{t_1})
	d xd yd s_1d s_2d t_1 d t_2\\
	+\int_0^T\!\int_0^T\!\EE
	Q(s,t,\varphi_s,\varphi_t)\,d sd t\,.
	\end{multline}
	\end{theorem}

	Before proceeding to the proof, let us make the following remark
	which we will use several times in the future.
	\begin{remark}\label{rem.ibp}
	  Suppose that $f$ and $g$ are  smooth functions, $f$ has compact support,  and $\varphi$
	  is random variable in $\DD^{1,2}$. Then the following integration by parts formula
	  holds almost surely
	   \begin{equation}\label{id.Dibp}
	    \int_{\RR^{d}}D f(x-\varphi) g(x)dx=-\int_{\RR^{d}}f(x) Dg(x+\varphi)dx\,.
	  \end{equation}
	  \end{remark}
	  Indeed, the integration on the left hand side is
	  \begin{equation*}
	    \int_{\RR^{d}} \nabla f(x-\varphi)\cdot D \varphi g(x)dx\,.
	  \end{equation*}
	  Integrating by parts yields
	  \begin{equation*}
	    -\int_{\RR^{d}} f(x-\varphi) D \varphi\cdot \nabla g(x)dx\,.
	  \end{equation*}
	  With the change of the variable $x\mapsto x+\varphi $,
	  \begin{equation*}
	    -\int_{\RR^{d}} f(x) D \varphi\cdot \nabla g(x+\varphi)dx=
	    -\int_{\RR^{d}} f(x) D  g(x+\varphi)dx\,.
	  \end{equation*}

\noindent 	{\it Proof} \  of Theorem \ref{thm.itosko}.\ \
	For any  $\ep>0$, the $\HH$-valued random variable
	$\eta_\ep(\cdot-\varphi_\cdot)$ belongs to $\DDD$, hence belongs to
	$\dom\delta$. Thus, applying \eqref{id.covdel}, for every positive
	numbers $\ep$ and $\ep'$, we obtain
	\begin{multline}\label{eqn.eta.eta}
	\EE(\delta(\eta_\ep(\cdot-\varphi_\cdot))\delta(\eta_{\ep'}(\cdot-\varphi_\cdot)))
	=\EE\sprod{\eta_\ep(\cdot-\varphi_\cdot)}{\eta_{\ep'}(\cdot-\varphi_\cdot)}_\HH\\
	 +\EE\sprod{D\eta_\ep(\cdot-\varphi_\cdot)}{\left(D\eta_{\ep'}(\cdot-\varphi_\cdot)\right)^*}_{\HH\otimes
	\HH}=:E_1+E_2.
	\end{multline}
	Using a change of variable, we have
	\begin{align*}
	E_1&=\ah
	\EE\int_0^T\!\int_0^T\!\int_{\RR^{2d}}\!\eta_\ep(x-\varphi_s)\eta_{\ep'}
	(y-\varphi_t)q(x,y)|t-s|^{2H-2}d xd yd sd t\\ &=\ah
	\EE\int_0^T\!\int_0^T\! \int_{2\Rd}\!\eta_\ep(x)\eta_{\ep'}(y)q(x+\varphi_s,y+\varphi_t) |t-s|^{2H-2}d xd yd sd t
	\\&=\ah\int_0^T\!\int_0^T\!\int_{\RR^{2d}}\!\eta_\ep(x)\eta_{\ep'}(y)\EE
	q(x+\varphi_s,y+\varphi_t)|t-s|^{2H-2}d xd yd sd t.
	\end{align*}
	When $\ep$ and $\ep'$ tend to 0, using the conditons \ref{cond.qphi1}, \ref{cond.qphi3}, this quality converges to
	$$\ah\int_0^T\!\int_0^T\!\EE q(\varphi_s,\varphi_t)|t-s|^{2H-2}\,d
	sd t =q_\varphi(0,0).$$
	Hence, when $\ep$ tends to zero, $\eta_\ep(\cdot-\varphi_\cdot )$
	converges in $L^2(\Omega;\HH)$ to a $\HH$-valued random variable,
	denoted by {\replacecolorred  $\delta_\varphi=\de(\varphi_t-y)$}.

	For the second expectation in \eref{eqn.eta.eta}, we use \eref{ito-iso} to obtain
	\begin{multline*}
	E_2=\alpha^2_H \EE\!\!\!\!\!\!\!\!\int\limits_{[0,T]^4\times\RR^{4d}}
	\!\!\!\!\!\!\!\!\!D_{s_1,x_1}
	\eta_\ep(x_2-\varphi_{s_2})D_{t_2,y_2}\eta_{\ep'}(y_1-\varphi_{t_1})q(x_1,y_1)q(x_2,y_2)\\
	|s_1-t_1|^{2H-2}|s_2-t_2|^{2H-2} dsdtdxdy\,.
	%\d x_1\d x_2\d {y_1}\d y_2\d {s_1}\d
	%{s_2}\d {t_1}\d {t_2}.
	\end{multline*}
	An application of \eqref{id.Dibp} yields
	\begin{multline*}
	  E_2=\alpha^2_H
	\EE\!\!\!\!\int\limits_{[0,T]^4\times\RR^{4d}}\!\!\eta_\ep(x_2)
	D_{s_1,x_1}q(x_2+\varphi_{s_2},y_2)   \eta_{\ep'}(y_1) D_{t_2,y_2}
	q(x_1,y_1+\varphi_{t_1})\\
	|s_1-t_1|^{2H-2}|s_2-t_2|^{2H-2} dsdtdxdy\,.
	%\d x_1\d x_2\d {y_1}\d y_2\d {s_1}\d
	%{s_2}\d {t_1}\d {t_2}.
	\end{multline*}
	% We put
	% \[\d z=|s_1-t_1|^{2H-2}|s_2-t_2|^{2H-2}\d x_1\d x_2\d {y_1}\d y_2\d
	% {s_1}\d {s_2}\d {t_1}\d {t_2}\] because it will be unchanged in what
	% follows. We then proceed, applying integration by part and change of
	% varibles,
	% \begin{multline*}
	% E_2=\alpha^2_H \EE\!\!\!\!\!\!\!\!\int\limits_{[0,T]^4\times\RR^{4d}}\!\!\!\!\!\!\!\!\!D_{s_1,x_1}\eta_\ep(x_2-\varphi_{s_2})D_{t_2,y_2}\eta_{\ep'}(y_1-\varphi_{t_1})q(x_1,y_1)q(x_2,y_2)\d z\\
	% =\alpha^2_H \EE\!\!\!\!\!\!\!\!\!\int\limits_{[0,T]^4\times\RR^{4d}}\!\!\!\!\!\!\!\!\!\nabla\eta_\ep(x_2-\varphi_{s_2})\cdot D_{s_1,x_1}\varphi_{s_2}\nabla\eta_{\ep'}(y_1-\varphi_{t_1})\cdot D_{t_2,y_2}\varphi_{t_1}q(x_1,y_1)q(x_2,y_2)\d z\\
	% =\alpha^2_H
	% \EE\!\!\!\!\!\!\!\!\!\!\int\limits_{[0,T]^4\times\RR^{4d}}\!\!\!\!\!\!\!\!\!\eta_\ep(x_2-\varphi_{s_2})
	% D_{s_1,x_1}\varphi_{s_2}\cdot\nabla_xq(x_2,y_2)
	% \eta_{\ep'}(y_1-\varphi_{t_1})
	% D_{t_2,y_2}\varphi_{t_1}\cdot{\nabla_y} q(x_1,y_1)\d z\\
	% =\alpha^2_H
	% \EE\!\!\!\!\!\!\!\!\!\!\int\limits_{[0,T]^4\times\RR^{4d}}\!\!\!\!\!\!\!\!\!\eta_\ep(x_2)
	% D_{s_1,x_1}\varphi_{s_2}\cdot\nabla_x q(x_2+\varphi_{s_2},y_2)
	% \eta_{\ep'}(y_1) D_{t_2,y_2}\varphi_{t_1}\cdot\nabla_y q(x_1,y_1+\varphi_{t_1})\d z\\
	% =\alpha^2_H
	% \EE\!\!\!\!\!\!\!\!\!\!\int\limits_{[0,T]^4\times\RR^{4d}}\!\!\!\!\!\!\!\!\!\eta_\ep(x_2)
	% D_{s_1,x_1}q(x_2+\varphi_{s_2},y_2)   \eta_{\ep'}(y_1) D_{t_2,y_2}
	% q(x_1,y_1+\varphi_{t_1})\d z.
	% \end{multline*}
	When $\ep$ and $\ep'$ tend to 0, this converges to $q_{D\varphi}^*(0,0)$ by using conditions \ref{cond.qphi4}, \ref{cond.qphi5}.

	Therefore,
	$\delta(\eta_\ep(\cdot-\varphi_\cdot))$ is a Cauchy sequence in $L^2(\Om)$. Since $\delta$ is a closed operator and $\eta_{\ep}(\cdot- \varphi_\cdot)$ converges to $\delta_{\varphi}$, we obtain   that $\delta_{\varphi}$ belongs to the domain of $\delta$.  As a consequence, $\delta(\eta_\ep(\cdot-\varphi_\cdot)) $
	converges to $\delta(\delta_{\varphi})$ when $\ep$
	tends to zero. Thus the integration $\int_0^T W(\delta t, \phi_t)$ is well-defined.
	The equation \eref{id.ddel.phi} is    immediate.\hfil $\Box$
	%Applying Lemma \ref{lemma.dom.del}, notice that
	%$\eta_\ep(\cdot-\varphi_\cdot)$ converges to $\ddp$ in
	%$L^2(\Omega;\HH)$, we imply that $\ddp$ belongs to the domain of the
	%divergence operator. Moreover, passing through a limit in
	%\eqref{eqn.eta.eta} we get \eqref{id.ddel.phi}.
	
	\begin{remark}
	Under the hypothesis of the above theorem, the $\HH$-valued random
	variable $\eta_\ep(\cdot-\varphi_\cdot)$ converges in
	$L^2(\Omega;\HH)$ to {\replacecolorred $\delta_\varphi=\de(\varphi_t-y)$}
	 as $\ep$ tends to zero. Moreover,
	  $\delta_\varphi$ also belongs to the
	domain of the divergence operator and the convergence also holds
	under the divergence $\delta$. Hence, in this case, the stochastic
	integral in Definition \ref{def.itosko} can be viewed as
	$\delta(\delta_\varphi)$, the divergence of $\delta_\varphi$.
	\end{remark}

\subsection{Nonlinear symmetric stochastic  integral}
\label{subsec.sym.int}
	We introduce and study symmetric type stochastic  integral
	by using appropriate approximation. This stochastic integral
	will be different than the It\^o-Skorohod type integral introduced
	in the previous subsection.

	Recall that  $W=\{W(s,x,\omega), \omega\in\Omega\}$ is the
	Gaussian random field (indexed by $(s,x)$) defined in the previous
	subsection. Throughout this subsection, we assume
	that $W$ is almost surely continuous  with respect to $s\ge 0$ and $x\in\Rd$. We define the composition of the random field $W$ and a
	$\Rd$-valued process $\varphi=\{\varphi_s,s\in[0,T]\}$ by
	\begin{equation}
	\begin{split}
	W(s,\varphi_s):\,&\Omega\to \RR\\
	&\omega\mapsto W(s,\varphi_s(\omega),\omega).
	\end{split}
	\end{equation}
	By convention, we will assume that all processes and functions vanish outside the interval $[0,T]$.

\begin{definition}\label{d.symmetric}
	The	symmetric  integral $\int_a^b W(\dsym
	s,\varphi_s)$ is defined as the limit as $\ep$ tends to
	zero of
	\begin{equation}\label{def.symint}
		(2\ep)^{-1}\int_a^b\!
		\left(W(s+\ep,\varphi_s)-W(s-\ep,\varphi_s)\right)\,d s,
	\end{equation}
	provided this limit exists in probability.
\end{definition}

	\begin{example} In the particular case, when $W(s,x)=B_s f(x)$, where $f$ is a nice deterministic function and $\{B_s, s\ge 0\}$ is a Brownian
	motion, the symmetric integral defined above coincides with Stratonovich integral. That is $\int_0^T\! W(\dsym
	s,\varphi_s)=\int_0^T f(\varphi_s) \dcir B_s$.
	\end{example}

	In the following proposition we will see that for a suitable class
	of $\Rd$-valued processes $\{\varphi_t\}$,  the symmetric stochastic integral
	$\int_0^T\! W(\dsym
	s,\varphi_s) $  exists almost surely. This result is an extension of \cite[Proposition 3]{alos-nualart}.
	\begin{prop}\label{prop.sym.int}
	Let $\varphi$ be a $\Rd$-valued process satisfying assumptions \ref{cond.qphi0}-\ref{cond.qphi5}. In addition, suppose that $\varphi$ satisfies
	\begin{equation} \label{cond.pw1}
	  \int_0^T\!\int_{|x|<1}\![\EE q(x+\varphi_s,x+\varphi_s)]^{1/2}\,d xd s <\infty,
	\end{equation}
	\begin{equation}\label{cond.pw2}
	\int_0^T\!\int_{|x|<1}\!\left[\EE\left|\sum_{i,j=1}^d\sprod{D\varphi_s^i}{D\varphi_s^j}_{\HH}\right| q(x+\varphi_s,x+\varphi_s)\right]^{1/2}\,d xd s<\infty
	\end{equation}
	and the function
	\begin{equation}\label{cond.pw3}
	x\mapsto\int_0^T\!\int_0^T\!\int_{\RR^{d}}\!\left|D_{t,y}q(x+\varphi_s,y)\right||s-t|^{2H-2}d
	td sd y
	\end{equation}
	is a.s. well-defined and continuous on a  neighborhood of 0. Assume
	also that the Gaussian field $W$ has continuous sample path. Then
	the symmetric integral \eqref{def.symint} exists and the
	following formula holds almost surely
	\begin{multline}\label{id.pw}
	\int_0^T\! W(\dsym s,\varphi_s)=\int_0^T\! W(\delta s,\varphi_s)\\+\ah\!
	\int_0^T\!\int_0^T\!\int_{\RR^{2d}}\!D_{t,y}q(\varphi_s,y)|s-t|^{2H-2}d
	td sd y.
	\end{multline}
	\end{prop}
	\begin{proof} We  shall  show the convergence in $L^2$
	of \eref{def.symint}.   For every positive $\ep$, since $W$ has
	continuous sample path, we can write
	\begin{align}
	W(s+\ep,\varphi_s)-W(s-\ep,\varphi_s)
	&=\lim_{\ep'\to0}\int_\Rd\!\left[W(s+\ep,x)
	-W(s-\ep,x)\right]\eta_{\ep'}(x-\varphi_s)\,d x\nonumber\\
	&= \lim_{\ep'\to0}\int_\Rd\!\delta(I_{[s-\ep,s+\ep]}\delta_x) \eta_{\ep'}(x-\varphi_s)\,d x\,, \label{integrands}
	\end{align}
	almost surely, where we have used \eqref{map.b} in the last equality. We notice $\eta_{\ep'}(x- \varphi_s)$ belongs to $\DD^{1,2}$ for every $s$ and $x$. Using \eqref{id.fu},  we see that the integrand on the right hand
	side of \eref{integrands} can be written as
	\begin{equation*}
	\delta\left(I_{(s-\ep,s+\ep]}\delta_x\eta_{\ep'}(x-\varphi_s)
	\right)+\sprod{D\eta_{\ep'}(x-\varphi_s)}{I_{(s-\ep,s+\ep]}\delta_x}_\HH.
	\end{equation*}
	 Taking integration with respect to $x$ and $s$, we obtain
	\begin{align}
	&(2\ep)^{-1}\int_0^T\!\int_\Rd\!\left[W(s+\ep,x)
	-W(s-\ep,x)\right]\eta_{\ep'}(x-\varphi_s)\,d xd s\nonumber \\
	&=(2\ep)^{-1}\int_0^T\!\int_\Rd\!\delta\left(I_{(s-\ep,s+\ep]}\delta_x\eta_{\ep'}(x-\varphi_s)
	\right)\,d xd s
	\nonumber\\
	&\quad +(2\ep)^{-1}\int_0^T\!\int_\Rd\!\sprod{D\eta_{\ep'}(x-\varphi_s)}
	{I_{(s-\ep,s+\ep]}\delta_x}_\HH\,d xd s\nonumber\\
	&=: I_1+I_2\,.\label{eqn.pw}
	\end{align}
	%We call the integrals on the above right hand side $I_1$ and $I_2$
	%respectively.
	The proof is now decomposed into several steps.

	{\it Step 1.} Let us show that the integration with respect to $\dx\ds$ in $I_1$ can be
	interchanged with the divergence operator to obtain
	\begin{equation*}
	I_1=\delta\left((2\ep)^{-1}\int_0^T\!\int_{\Rd}\!I_{(s-\ep,s+\ep]}\delta_x\eta_{\ep'}(x-\varphi_s)\,
	d xd s \right)\,.
	\end{equation*}
	In fact, one can view
	the integral in $I_1$ in Bochner sense, that is integration with $L^2$-valued integrand. In this setting, we have
	\begin{equation*}
	  \int_0^T\int_{\RR^d}\delta(u(s,x))\dx\ds=\delta\left(\int_0^T\int_{\RR^d}u(s,x)dxds\right)
	\end{equation*}
	provided that
	\begin{equation}\label{tocheck.u}
	  \int_0^T\int_{\RR^d}\|u(s,x)\|_{\DD^{1,2}}dxds<\infty
	\end{equation}
	and $\delta$ is a bounded operator from $\DD^{1,2}$ to $L^2$. The later fact is automatically guaranteed by \eqref{id.covdel}. It remains to check that $u(s,x)=I_{(s-\ep,s+\ep]}\delta_x \eta_{\ep'}(x-\varphi_s)$ satisfies \eqref{tocheck.u}.
	\begin{align*}
	  \|u(s,x)\|_{\HH}^2&=\int_{s-\ep}^{s+\ep}\int_{s-\ep}^{s+\ep}\frac{\partial^2}{\partial s\partial t} R_H(s',t')\ds'\dt'q(x,x)\EE[\eta^2_{\ep'}(x- \varphi_s)]\\
	  &\le R_H([0,T]^2) q(x,x)\EE[\eta^2_{\ep'}(x- \varphi_s)]\,.
	\end{align*}
	Thus by a change of variable, we obtain
	\begin{align*}
	  \int_0^T\int_{\RR^d}\|u(s,x)\|_{\HH}\dx\ds&\le R_H^{1/2}([0,T]^2)\int_0^T\int_{\RR^d}(\EE q(x,x)\eta^2_{\ep'}(x- \varphi_s))^{1/2}\dx\ds\\
	  &=R_H^{1/2}([0,T]^2)\int_0^T\int_{\RR^d}(\EE q(x+\varphi_s,x+\varphi_s)\eta^2_{\ep'}(x))^{1/2}\dx\ds\\
	  &\le c({\ep',T})\int_0^T\int_{|x|<1}(\EE q(x+\varphi_s,x+\varphi_s))^{1/2}\dx\ds\,.
	\end{align*}
	The last integral is finite thanks to the condition \eqref{cond.pw1}.
	Similarly
	\begin{align*}
	  &\|Du(s,x)\|_{\HH}^2\\
	  &=\EE\int_{s-\ep}^{s+\ep}\int_{s-\ep}^{s+\ep}\frac{\partial^2}{\partial s\partial t} R_H(s',t')\ds'\dt'q(x,x)\partial_i \eta_{\ep'}(x-\varphi_s)\partial_j \eta_{\ep'}(x-\varphi_s)\langle D \varphi_s^i,D \varphi_s^j\rangle_{\HH}\\
	  &\le R_H([0,T]^2)q(x,x)\EE|\sum_{i,j}\partial_i \eta_{\ep'}(x-\varphi_s)\partial_j \eta_{\ep'}(x-\varphi_s)\langle D \varphi_s^i,D \varphi_s^j\rangle_{\HH}|\,.
	\end{align*}
	Thus, by a change of variable and by using the condition \eqref{cond.pw2},
	we obtain
	\begin{align*}
	  &\int_0^T\int_{\RR^d}\|D u(s,x)\|_{\HH\otimes\HH}dxds\\
	  &\le c(T)\int_{0}^T\int_{\RR^d}[\EE q(x,x)|\sum_{i,j}\partial_i \eta_{\ep'}(x-\varphi_s)\partial_j \eta_{\ep'}(x-\varphi_s)\langle D \varphi_s^i,D \varphi_s^j\rangle_{\HH}|]^{1/2}dxds\\
	  &\le c(\ep',T)\int_0^T\int_{|x|<1}[\EE q(x+\varphi_s,x+\varphi_s)|\sum_{i,j}\langle D \varphi_s^i,D \varphi_s^j\rangle_{\HH}|]^{1/2}dxds<\infty.
	\end{align*}

	{\it Step 2.} We show that
	\begin{multline*}
	\delta\left((2\ep)^{-1}\int_0^T\!\int_{\Rd}\!I_{(s-\ep,s+\ep]}\delta_x\eta_{\ep'}(x-\varphi_s)\,
	d xd s \right)\\
	=\delta\left((2\ep)^{-1}\int_0^T\!I_{(s-\ep,s+\ep]}\eta_{\ep'}(\cdot-\varphi_s)\,
	d s \right)\,.
	\end{multline*}
	It suffices to show for every smooth function $\phi$ with compact support
	\begin{equation}\label{id.phidelta}
	\phi=\int_\Rd\!\phi(y)\delta_{y}d y
	\end{equation}
is 	 in $\HH$,  since
 with the choice $\phi=\eta_{\ep'}$, \eqref{id.phidelta}
 will yield  the desired identity. Recall that $\SS$ is the space defined in Subsection \ref{subsection.noise} and is dense in $\HH$. Thus to show \eqref{id.phidelta}, we verify
	 \begin{equation*}
	    \langle\phi,\psi\rangle_{\HH}=\langle\int_\Rd\!\phi(y)\delta_{y}d y,\psi\rangle_{\HH}
	  \end{equation*}
	for every $\psi\in\SS$. Indeed, we have
	\begin{multline*}
	\sprod{\int_\Rd\!\phi(y)\delta_{y}\,d y}{\psi}_\HH
	=\int_\Rd\!\phi(y)\sprod{\delta_{y}}{\psi}_\HH d y\\=
	\int_{\RR^d}\int_0^T\int_0^T\int_{\RR^d}\phi(y)\psi(x) Q(s,t,y,x)\dx\ds\dt\dy =\sprod{\phi}{I_{(0,s]}\delta_{x}}_\HH\,,
	\end{multline*}
	by Fubini's theorem.

	{\it Step 3.} Combining the previous two steps, we obtain
	\begin{equation*}
	  I_1=\delta\left((2\ep)^{-1}\int_0^T\!I_{(s-\ep,s+\ep]}\eta_{\ep'}(\cdot-\varphi_s)\,d s \right)\,.
	\end{equation*}
	It is straightforward to check that when $\ep'$ and $\ep$ tend to zero,
	$I_1$ converges to $\int_0^T\!W(\delta s, \varphi_s)$ in $L^2$.

	{\it Step 4.}
	% \begin{align*}
	%   I_2
	%   &=(2\ep)^{-1} \int_0^T\ds\int_{\RR^d}\dx\int_0^T\dt\int_{\RR^d}\dy  {D_{t,y} \eta_{\ep'}(x- \varphi_s)}\frac{\partial}{\partial t}R([s-\ep,s+\ep],t)q(x,y) \\
	%   &=-(2\ep)^{-1}\int_0^T\ds\int_{\RR^d}\dx\int_0^T\dt\int_{\RR^d}\dy  { \eta_{\ep'}(x)}[\frac{\partial R}{\partial t}(s+\ep,t)-\frac{\partial R}{\partial t}(s-\ep,t)]D_{t,y}q(x+ \varphi_s,y)
	% \end{align*}
	We now show the convergence of $I_2$. A direct computation shows that
	\begin{multline*}
	|I_2|=(2\ep)^{-1}\left|\int_0^T\!\int_\Rd\!\int_{-\ep}^{\ep}\!\int_0^T\!\int_\Rd
	D_{t,y}\eta_{\ep'}(x-\varphi_s)q(x,y)|t-r-s|^{2H-2}dyd td rd xd s\right| \\
	\le d_H\int_0^T\!\int_\Rd\!\int_0^T\!\int_\Rd
	\left|D_{t,y}\eta_{\ep'}(x-\varphi_s)q(x,y)\right||t-s|^{2H-2}dydtdxds,
	\end{multline*}
	where we have used the following inequality
	\begin{equation*}
	(2\ep)^{-1}\int_{-\ep}^{\ep}|t-r-s|^{2H-2}\,d r\le d_H|t-s|^{2H-2}
	\end{equation*}
	for some constant $d_H$, independent $\vare\in (0, 1)$ and $s, t\in \RR$.
	By a change of variable $x- \varphi_s\rightarrow x$, we obtain
	\begin{equation*}
	I_2\le d_H\int_0^T\!\int_0^T\!\int_{\RR^{2d}}
	\left|\eta_{\ep'}(x)D_{t,y}q(x+\varphi_s,y)\right||t-s|^{2H-2}dydxdsd t.
	\end{equation*}
	Hence, by the dominated convergence theorem, when $\ep'$ and $\ep$
	tend to zero, $I_2$ goes to $\int_0^T\!\int_0^T\!\int_{\RR^{d}}
	D_{t,y}q(x+\varphi_s,y)|t-s|^{2H-2}d yd sd t$.
	Therefore, passing through the limits in \eqref{eqn.pw}, we obtain \eqref{id.pw}
	\end{proof}

{\replacecolorred
If the limit  in  Definition \ref{d.symmetric} exists  for  almost   every sample path
of $W$,  then the symmetric integral can also be defined pathwise for a function $
(W(t,x)\,, t\ge 0\,, x\in \RR^d)$. We also call such integral the symmetric integral
and denoted by the same symbol $ \int_0^T W(\dsym s,\varphi_s)$.
}

The following proposition establishes the relation between symmetric integral and
nonlinear Young integral introduced in Section 2.

\begin{proposition}  Assume the hypothesis of Proposition \ref{prop.riemann.w}. Then the symmetric integral exists and the following relation holds
  \begin{equation*}
    \int_0^T W(\dsym s,\varphi_s)=\int_0^T W(ds,\varphi_s)\,.
  \end{equation*}
\end{proposition}
\begin{proof}
  Fix $\epsilon>0$, we put
  \[  W_{\epsilon}(s,x)=(2 \epsilon)^{-1} \int_{-\epsilon}^{\epsilon} W(s+ \eta,x)d \eta\,.
  \]
  We recall that $\int_0^T W(\dsym s,\varphi_s)=\lim_{\epsilon\to0}\int_0^T \partial_t W_{\epsilon}(s,\varphi)\ds $. We put
  \begin{align*}
    & \mu_k(a,b) =W_{\epsilon_k}(b,\varphi_a)-W_{\epsilon_k}(a,\varphi_a)\,,\\
    & \mu(a,b) =W(b,\varphi_a)-W(a,\varphi_a)\,.
  \end{align*}
  Since $W_{\epsilon}$ is continuously differentiable in time, the integral $\int W_{\epsilon}(\ds,\varphi_s)$ is understood in classical sense and is equal  to $\int \partial_t W_{\epsilon}(s,\varphi_s)\ds $. Hence, applying Proposition \ref{prop.int.w12} we obtain, for any $\theta\in(0,1) $ such that $\theta \tau+\lambda \gamma>1$
  \begin{multline*}
    |\int_0^T W(\ds,\varphi_s)-\int_0^T\partial_t W_{\epsilon}(s,\varphi_s)\ds|\\
    \le |W(T,\varphi_0)-W(0,\varphi_0)-W_{\epsilon}(T,\varphi_0)+W_{\epsilon}(0,\varphi_0)|\\
    +c(\varphi)[W-W_{\epsilon}]_{\beta,\tau,\lambda}|b-a|^{\theta\tau+\lambda \gamma}\,.
  \end{multline*}
  It remains to estimate the terms on the right side and show that they all converge to 0 when $\epsilon$ goes to 0. For the first term
  \begin{align*}
    &|W(T,\varphi_0)-W(0,\varphi_0)-W_{\epsilon}(T,\varphi_0)+W_{\epsilon}(0,\varphi_0)|\\
    &\le (2 \epsilon)^{-1}\int_{-\epsilon}^\epsilon |W(T,\varphi_0)-W(0,\varphi_0)-W(T+\eta,\varphi_0)+W(\eta,\varphi_0)|d \eta\\
    &\lesssim (2 \epsilon)^{-1}\int_{-\epsilon}^\epsilon |\eta|^{\tau} d \eta\lesssim \epsilon^\tau\,.
  \end{align*}
  For the second term, we put $F=W-W_{\epsilon}$ and notice that
  \begin{align*}
    &|W_{\epsilon}(s,x)-W_{\epsilon}(s,y)-W_{\epsilon}(t,x)+W_{\epsilon}(t,y)|\\
    &\le(2 \epsilon)^{-1} \int_{-\epsilon}^\epsilon |W(s+\eta,x)-W(s+\eta,y)-W(t+\eta,x)+W(t+\eta,y)| d \eta\\
    &\le [W](1+|x|^\beta+|y|^\beta) (2 \epsilon)^{-1} \int_{-\epsilon}^\epsilon |s-t|^\tau|x-y|^ \lambda  d \eta\\
    &\le [W](1+|x|^\beta+|y|^\beta)  |s-t|^\tau|x-y|^ \lambda \,.
  \end{align*}
  Thus
  \begin{align*}
    |F(s,x)-F(s,y)-F(t,x)+F(t,y)|
    \le 2[W](1+|x|^\beta+|y|^\beta)  |s-t|^\tau|x-y|^ \lambda \,.
  \end{align*}
  On the other hand,
  \begin{align*}
    &|W_{\epsilon}(s,x)-W_{\epsilon}(s,y)-W_{\epsilon}(t,x)+W_{\epsilon}(t,y)|\\
    &\le(2 \epsilon)^{-1} \int_{-\epsilon}^\epsilon d \eta |W(s+\eta,x)-W(s,x)-W(s+\eta,y)+W(s,y)|\\
    &\qquad \qquad\qquad\qquad+|W(t,x)-W(t+\eta,x)-W(t,y)+W(t+\eta,y)|\\
    &\le 2[W](1+|x|^\beta+|y|^\beta)|x-y|^\lambda(2 \epsilon)^{-1} \int_{-\epsilon}^\epsilon  |\eta|^{\tau}d \eta\\
    &\le 2(1+\tau)^{-1} [W](1+|x|^\beta+|y|^\beta)|x-y|^\lambda \epsilon^{\tau}\,.
  \end{align*}
  Hence, combining these two bounds, we get
  \begin{multline*}
    |F(s,x)-F(s,y)-F(t,x)+F(t,y)|\\
    \lesssim [W](1+|x|^\beta+|y|^\beta)  |s-t|^{\theta\tau}|x-y|^ \lambda \epsilon^{\tau(1-\theta)} \,.
  \end{multline*}
  Thus $[W-W_{\epsilon}]_{\beta,\theta \tau,\lambda}\lesssim \epsilon^{\tau(1-\theta)}$ which converges to 0 as $\epsilon\to0$.
\end{proof}

\section{Estimates for diffusion process}\label{app.est.diff}
{\replacecolorred  In this section, we prove the exponential integrability
 of the H\"older norm and the supremum norm of a diffusion
process which is needed in proving the existence of the Feynman-Kac
solution in Section \ref{sec.feykac}. The results obtained here are known in
literature (see for instance   \cite{cerrai-book},
\cite{flandoli-bull}, \cite{nuabook}). However, it is difficult to
find a single-source treatment that suits our purpose. Besides, our method is straightforward and unified.   We  present them here.
}

We recall that $X_t^{r,x}$ satisfies the equation \eref{eqn.diffusion}. We denote
\begin{equation}
	M_t^{r,x} =\sum_{j=1}^d \int_r^t \sigma^{ij}(s,X_s^{r,x})\delta B_s^j\,.
\end{equation}
Since $\sigma$ is bounded (by condition \ref{cond.L.elliptic}), $(M_t^{r,x} ;t\ge r)$ is a continuous $L^2$ martingale. In addition, we have the following properties.
\begin{lemma}
	Let $\alpha$ be a number in $(0,1/2)$.
	There exist some positive constants $\gamma_0$ and $\gamma_ \alpha$ such that
	\begin{equation}\label{exp.M}
		\EE\exp \left\{\gamma_0 \sup_{r\le t\le T} |M_t^{r,x}|^2\right\}\le C(T-r,\Lambda)<\infty
	\end{equation} 	
	and
	\begin{equation}\label{exp.Ma}
		\EE \exp \left\{ \gamma_ \alpha \left(\sup_{r\le s,t\le T}\frac{|M_t^{r,x}-M_s^{r,x}|}{|t-s|^\alpha}\right)^2\right\}\le C(T-r,\Lambda,\alpha)<\infty\,.
	\end{equation}
\end{lemma}
\begin{proof}
	\eqref{exp.M} is well-known and is a direct application of Doob's maximal inequality and Burkholder-Davis-Gundy inequality. \eqref{exp.Ma} is proved in \cite[Lemma 2]{benarous}. However, for readers' convenience, we present a proof of \eqref{exp.Ma} in the following. We will omit the upper indices $r,x$. Applying the Garsia-Rodemich-Rumsey theorem (See \cite{garsiarodemich} and \cite{hule2012}, specifically \cite[Theorem 2.1.3] {stroockvaradhan}) with $\Psi(x)=x^p$ and $p(x)=x^{\al+2/p}$, we have
	\[
	|M_t-M_s|\le 8(1+\frac{2}{\al p}) 4^{1/p} |t-s|^\al \left\{\int_r^T\int_r^T \left(\frac{
	|M_u-M_v| }{|u-v|^{\al+2/p}}\right)^p dudv\right\}^{1/p}\,.
	\]
	Dividing both sides by $|t-s|^\al$ and taking the sup on $r\le s<t\le T$,
	we see that  there is a constant $C=C(\al)$, independent of  $p\ge 1$,  such that
	\begin{equation*}
	\EE \left(\sup_{r\le s<t\le T} \frac{|M_t-M_s|}{|t-s|^\alpha}\right) ^p
	\le  C^p   \int_r^T\int_r^T  \frac{\EE
	|M_u-M_v|^p  }{|u-v|^{\al p+2 }} dudv\,.
	\end{equation*}
	An application of  the Burkholder-Davis-Gundy inequality   gives
	  \begin{equation*}
	    \|M_u-M_v \|_p \le 2 p^{1/2}\|\int_u^v a^{ii}(s,X_s^{r,x})\ds\|_{p/2}^{1/2}\le 2 \Lambda^{1/2} p^{1/2}(t-r)^{1/2}\,.
	  \end{equation*}
	It follows that there is a constant $C$, which may be
	different than the above one, such that the $p$-moments of $\sup_{r\le s<t\le T} \frac{|M_t-M_s|}{|t-s|^\alpha}$ is at most $C^p p^{p/2}(T-r)^{(\frac12-\al)p}$ for all $p>(\frac12- \alpha)^{-1}$, which yields \eqref{exp.Ma}.
\end{proof}
\begin{lemma}\label{exponential-int}
	Fix $\alpha\in(0,1/2)$. There exist positive constants $C_0$, $\gamma_0$ and $\gamma_ \alpha$ such that
	\begin{equation}\label{exp.supX}
		\EE\exp \left\{\gamma_0 \sup_{r\le t\le T} |X_t^{r,x}|^2\right\}\lesssim e^{C_0 |x|^2}
	\end{equation} 	
	and
	\begin{equation}\label{e.6.11}
		\EE \exp \left\{ \gamma_ \alpha \left(\sup_{r\le s,t\le T}\frac{|X_t^{r,x}-X_s^{r,x}|}{|t-s|^\alpha}\right)^2\right\}\lesssim e^{C_0 |x|^2}\,.
	\end{equation}
\end{lemma}
\begin{proof}
	We denote $X_t^*=\sup_{r\le s\le t} |X_s^{r,x}|$ and $M_t^*=\sup_{r\le s\le t}|M_s^{r,x}|$. We first prove \eqref{exp.supX}. Since $b$ has linear growth (by \ref{cond.Lbreg}), from equation \eqref{eqn.diffusion}, we see that
	\begin{equation*}
		|X_t|\le |M_t|+ |x|+\kappa(b) \int_r^t|X_s|ds\,.
	\end{equation*}
	An application of Gronwall's inequality yields
	\begin{equation*}
		|X_t|\le |M_t|+|x|+\kappa(b) \int_r^t (|M_s|+|x|)e^{\kappa(b)(t-s)}ds\,.
	\end{equation*}
	Hence, for all $p\ge0$, applying Jensen's inequality,
	\begin{align*}
		\exp\{pX_T^*\}&\le \exp\{p(M_T^*+|x|)\}\exp\{p\kappa(b)\int_r^T(|M_s|+|x|)e^{\kappa(b)(t-s)}ds\}
		\\&\le  \frac{\exp\{p(M_T^*+|x|)\}}{e^{\kappa(b)(T-r)}-1} \int_r^T\exp\{p(e^{\kappa(b)(t-r)}-1)(|M_s|+|x|)e^{\kappa(b)(T-s)}\}ds
		\\&\lesssim \exp\{p(M_T^*+|x|)\}\int_r^T \exp\{Cp(|M_s|+|x|)\}ds
	\end{align*}
	for some constant $C$ depending on $T-r$ and $\kappa(b)$. We then apply Cauchy-Schwartz inequality
	\begin{align*}
		\EE e^{pX_T^*}&\lesssim \EE e^{2p(M_T^*+|x|)}+\int_r^T \EE e^{2Cp(|M_s|+|x|)}ds\\
		&\lesssim \EE e^{2Cp(M_T^*+|x|)}\,,
	\end{align*}
where the constants (including the implied constant) are independent of $p$.
	Now we choose $p$ according to the distribution $|N(0,a)|$ with $a$ sufficient small, where $N(0,a)$ is a normal distribution independent of $B$. Using \eqref{exp.M}, the elementary estimate $\frac12 e^{\frac{a^2}2 A^2}\le \EE^N e^{pA}\le 2 e^{\frac{a^2}2 A^2}$ (with $A>0$), and the previous estimate, we obtain \eqref{exp.supX}.

	From \eref{eqn.diffusion}, we have
	\begin{eqnarray*}
	\frac{X_t -X_s }{(t-s)^\al }=\frac{\int_s^t b(u,X_u ) du}{(t-s)^\al
	}+\frac{M_t-M_s }{(t-s)^\al } =:I_1+I_2\,.
	\end{eqnarray*}
	Since $b$ has linear growth, $ \sup_{r\le s<t\le T}|I_1|\le c(\kappa(b),T,\alpha)(1+X_T^* )$. \eqref{e.6.11} follows from \eqref{exp.supX} and \eqref{exp.Ma}.
\end{proof}

\section{Schauder estimates} % (fold)
\label{sec:schauder_estimates}
We present the proof of Lemma~\ref{lem.estv}. The estimates \eqref{est.vw}-\eqref{est.dvwa} are similar to Schauder estimates in the classical theory of parabolic equations. Besides the results obtained in Appendix \ref{app.est.diff}, the method adopted here also makes use of Malliavin calculus. For this purpose, we need some preparations.

It is well-known (see e.g. \cite{nuabook})  that $X_t^{r,x}$ is differentiable (in Malliavin
sense) with respect  to the Brownian motion $B_t$. We denote the
Malliavin derivative  of $X$ with respect to $B^j$ by $D^jX$. It is
shown in \cite[Theorem 2.2.1]{nuabook} that $DX=(D^1X, \cdots,
D^dX)^T$ has finite moments of all orders and satisfies
\begin{equation*}
  d D_{\tau}X^{i,r,x}_t=\sigma^{ij}_k(t)D_{\tau}X_t^{k,r,x}\delta B^j_t
  +b^{i}_k(t)D_{\tau}X^{k,r,x}_t\dt\,,\quad D^j_{\tau}X^{r,x}_\tau
  =\sigma^{ij}(\tau,X_{\tau}^{r,x})
\end{equation*}
for $t\ge \tau\ge r$, $D_{\tau}X_t^{r,x}=0$ if $t< \tau\le T$. In
the above equation, we have used the notations
\begin{equation*}
  \sigma^{ij}_k(t)=\partial_{x_k}\sigma^{ij}(t,X_t^{r,x})\,,\quad b^i_k(t)=\partial_{x_k}b^i(t,X_t^{r,x})\,.
\end{equation*}

The matrix $DX$ is understood
as $[DX]^{ij}=D^jX^i$.
Following the proof of \cite[Theorem 2.2.1]{nuabook}, one can show
that the map  $x\mapsto X^{r,x}_t$ is differentiable. We denote
$Y{(t;r,x)}=\frac{\partial}{\partial x}X^{r,x}_t$, the Jacobian of
$x\mapsto X_t^{r,x}$. The matrix $Y$ is understood as $[Y]^{ij}=Y^i_j=\partial_j X^i$.  It follows that the $d\times d$-matrix valued
process $t\mapsto Y{(t;r,x)}$ satisfies
\begin{align}
  d Y_\bullet^i(t;r,x)&=\sigma^{ij}_k(t)Y^k_\bullet(t;r,x)\delta B^j_t+b^{i}_k(t)Y^k_\bullet(t;r,x) \dt\,,\label{eqn.Y}\\
  Y(r;r,x)&=I_{d\times d}\,. \nonumber
\end{align}
Let $Z(t)$ be the $d\times d$ matrix-valued process defined by
\begin{equation*}
  d Z^\bullet_i(t)=-Z^\bullet_\theta(t)\sigma^{\theta l}_i \delta B^l_t-
  Z^\bullet_\theta(t)\left[b^\theta_i(t) - \sigma^{\alpha l}_i(t) \sigma^{\theta l}_\alpha(t)
  \right]\dt\,,\quad Z(r;r,x)=I_{d\times d}\,.
\end{equation*}
By means of It\^o's formula,  we have
\begin{align*}
  d(Z_i^k Y^i_j)=&-Y^i_j Z^k_\theta \sigma^{\theta l}_i \delta B^l_t-Y^i_j Z^k_\theta
   b_i^\theta\dt+ Z^k_\theta Y^i_j \sigma^{\alpha l}_i \sigma^{\theta l}_\alpha\dt\\
  & +Z^k_i Y^\theta_i \sigma^{il}_\theta \delta B^l_t+Z^k_i Y^\theta_j b^i_\theta\dt
  - Z^k_\theta Y^\alpha_j \sigma^{ij}_\alpha \sigma^{\theta l}_i\dt=0
\end{align*}
and similarly for $Y_tZ_t$.  Thus we obtain $Y_tZ_t=Y_tZ_t=I$.   As
a consequence, for every $t\ge r$, the matrix $Y(t;r,x)$ is
invertible and its inverse is $Z(t;r,x)$. It is a standard fact that
$Y$ and $Z$ have  finite moments of all orders. More precisely, one
has
\begin{equation}\label{moment.Y}
   \sup_{{t\in[r,T],x\in\RR^d}}\EE \left[ |Y(t;r,x)|^p+|Y^ {-1} (t;r,x)  |^p\right]\le c(p,T)\,.
\end{equation}
Since the coefficients of $L$ are twice differentiable with bounded
derivatives, $DY$ exists and has finite moment of all orders and
\begin{equation}\label{moment.DY}
  \sup_{t\in[r,T],x\in\RR^d}  \EE\sup_{\tau\in[r,T]} \left[
  |D_{\tau}Y  (t;r,x)|^p+ |D_{\tau}Y^{-1} (t;r,x)|^p\right]\le
  c(p,T)\,.
\end{equation}
Moreover, it is well-known that the following representation holds
(see, for instance \cite[pg. 126]{nuabook})
\begin{equation}\label{id.dx}
  D_{\tau}X_t^{r,x}=Y(t;r,x)Z(\tau;r,x)\sigma(\tau,X_{\tau}^{r,x})\,,\quad\forall \tau\in[r,t]\,.
\end{equation}
As a consequence, if $f$ is a smooth function, we have
\begin{equation}\label{eqn.DfY}
    D_ \tau f(s,X_s^{r,x})^T=\nabla f(s,X_s^{r,x})^TY(s;r,x)Y^{-1}(\tau;r,x)\sigma(\tau,X_{\tau}^{r,x})\,.
  \end{equation}
(where and  in what follows we denote $\nabla
f(s,X_s^{r,x})=\left(\nabla f\right) (s,X_s^{r,x})$).  Later on, we
occasionally make use of its variant
\begin{equation}\label{id.ydy}
  \nabla f(s,X_s^{r,x})^TY(s;r,x)=D_ \tau f(s,X_s^{r,x})^T
  \sigma^{-1}(\tau,X_{\tau}^{r,x})Y(\tau;r,x)\,,\quad\forall \tau\in[r,t]\,.
\end{equation}

\begin{lemma}[Bismut formula]
	Suppose $f$ belongs to $C^2(\RR^{d+1})$ and suppose  $f$ and its derivatives have polynomial growth. Then
  \begin{multline}\label{bismut1}
    \EE \left[\left(\partial_if\right) (s,X_s^{r,x})\right]\\
    =\frac1{s-r}\EE \left[f(s,X_s^{r,x})\int_r^s [\sigma^{-1}(\tau,X_{\tau}^{r,x})
    Y(\tau;r,x)Y^{-1}(s;r,x)]^{ji}\delta B^j_{\tau}\right]
  \end{multline}
	and
  \begin{equation}\label{bismut2}
    \partial_i\EE f(s,X_s^{r,x})=\frac1{s-r}\EE \left[f(s,X_s^{r,x})\int_r^s [\sigma^{-1}
    (\tau,X_{\tau}^{r,x})Y(\tau;r,x)]^{ji}\delta B^j_{\tau}\right] \,.
  \end{equation}
\end{lemma}
\begin{proof}
	Fix $\tau\in[r,s]$.	The identity \eqref{eqn.DfY} yields
	\begin{equation*}
	  \nabla f(s,X_s)^T= \left[D_{\tau}f(s,X_s) \right]^T \sigma^{-1}(\tau)Y(\tau)Y^{-1}(s) \,.
	\end{equation*}
	Integrating  with respect to $\tau$ from $r$ to $s$ and taking the
	expectation  give 
	  \begin{equation*}
	    \EE \nabla f(s,X_s)^T=\frac{1}{s-r}\EE\left[\int_r^s \left[D_{\tau}f(s,X_s) \right]^T
	    \sigma^{-1}(\tau)Y(\tau)Y^{-1}(s)d \tau\right] \,.
	  \end{equation*}
	  Formula \eqref{bismut1} is then followed from the dual relationship
	  \eqref{eqn.dual} between the divergence operator $\delta$ and the Malliavin derivative $D$.

	  To show \eqref{bismut2}, we use \eqref{id.ydy}.
	 We integrate with respect to $\tau$ from $r$ to $s$ and then take the expectation to obtain
	  \begin{equation*}
	    \nabla\EE f(s,X_s)=\frac1{s-r}\EE \left[\int_r^s \left[D_{\tau} f(s,X_s)\right]^T \sigma^{-1}(\tau) Y(\tau)d \tau
\right] \,.
	  \end{equation*}
	  Formula \eqref{bismut2}  follows  from the dual relationship \eqref{eqn.dual} between $\delta$ and $D$.
\end{proof}
\begin{lemma}\label{lem.estdf}  Suppose that $f$ is differentiable and satisfies
  \begin{equation*}
    \sup_{s\in[0,T],x\in\RR^d}\frac{|f(s,x)|}{1+|x|^\beta} \le \kappa
  \end{equation*}
  for some nonnegative constants $\kappa$ and $\beta$. Then we have
  \begin{equation}\label{est.bismut1}
    |\EE \left(\nabla f\right) (s,X_s^{r,x})|\le c(T,\Lambda,\lambda)\kappa(1+|x|^\beta)[1+(s-r)^{-1/2}
    ]
  \end{equation}
  and
  \begin{equation}\label{est.bismut2}
    |\nabla\EE  f(s,X_s^{r,x})|\le c(T,\Lambda,\lambda)\kappa(1+|x|^\beta)(s-r)^{-1/2}\,.
  \end{equation}
\end{lemma}
\begin{proof}
  We only provide details for the proof of  \eqref{est.bismut1}.   The estimate \eqref{est.bismut2}
  is proved similarly, perhaps in an easier manner. Motivated by the formula \eqref{bismut1},
  we first estimate the moment of $\int_r^s[\sigma^{-1}(\tau)Y(\tau)Y^{-1}(s)]^{ji}\delta B_{\tau}^j $.
   From \eqref{id.fu}, we see that
  \begin{align*}
    \int_r^s[\sigma^{-1}(\tau)Y(\tau)Y^{-1}(s)]^{ji}\delta B_{\tau}^j=
    &\int_r^s[\sigma^{-1}(\tau)Y(\tau)]^{jk}\de B_{\tau}^j[Y^{-1}(s)]^{ki}\\
    &-\int_r^s[\sigma^{-1}(\tau)Y(\tau)]^{jk}D^j_{\tau}[Y^{-1}(s)]^{ki}d \tau\,.
  \end{align*}
  From \eqref{moment.Y} and \eqref{moment.DY}, it follows that
  \begin{equation*}
    \sup_{s\in[r,T],x\in\RR^d} \EE |\int_r^s[\sigma^{-1}(\tau)Y(\tau)Y^{-1}(s)]^{ji}\delta B_{\tau}^j|^p\le c(p,T)[(s-r)^{1/2}+(s-r) ]^p\,.
  \end{equation*}
  Hence, applying H\"older inequality in \eqref{bismut1},
  \begin{equation*}
    |\EE \nabla f(s,X_s^{r,x})|\lesssim(1+(s-r)^{-1/2}) [\EE (1+|X_s^{r,x}|^\beta )^2]^{1/2}\,.
  \end{equation*}
  Together with \eqref{exp.supX}, this completes the proof of \eqref{est.bismut1}.
\end{proof}
\begin{proof}[Proof of Lemma~\ref{lem.estv}]
	Throughout the proof, we denote $\kappa_1=[\nabla W]_{\beta_1,\infty}$, $\kappa_2=[\nabla W]_{\beta_2,\alpha}$, $Y=\nabla \varphi$.

  {\bf Uniqueness:} Suppose $v$ is a solution in $C^1([0,T];C^{2}(\RR^d))$.
  We apply It\^o formula to  the process $s\mapsto (v+W)(s,\varphi_s^{r,x})$,
  taking into account the fact that $L_0$ is the generator of
  $\varphi_s^{r,x}$.
  \begin{equation}\label{eqn.itovw}
  \begin{split}
    d (v+W)(s,\varphi_s^{r,x})=&(\partial_t+L_0)(v+W)(s,\varphi_s^{r,x})ds\\
    &+\sigma^{ij}(s,\varphi_s^{r,x})\partial_{x_i}(v+W)(s,\varphi_s^{r,x})\delta B_s^j\,.
  \end{split}
  \end{equation}
  Since $v$ is a strong solution, we see that $v+W$ satisfies
  \begin{equation*}
    (\partial_t+L_0)(v+W)=L_0W\,,\quad (v+W)(T,x)=0\,.
  \end{equation*}
  Thus, integrating \eqref{eqn.itovw} from $r$ to $T$ yields
  \begin{equation*}
    -(v+W)(r,x)=\int_r^T L_0W(s,X_s^{r,x})\ds+\int_r^T\sigma^{ij}(s,X_s^{r,x})\partial_{x_i}(v+W)(s,X_s^{r,x})\delta B_s^j\,.
  \end{equation*}
  Taking expectation in the above identity, we obtain \eqref{eqn.vw}, which also shows the uniqueness of $v$.

  {\bf $C^0$-estimate:} To prove the estimate \eqref{est.vw}, we write $L_0W=\partial_i\left(\frac12a^{ij}\partial_j W \right)+c^j\partial_j W$
  % \begin{equation*}
  %   L_0W=\partial_i\left(\frac12a^{ij}\partial_j W \right)+c^j\partial_j W
  % \end{equation*}
  where $c^j=-1/2\partial_ia^{ij}$. Then
  \begin{equation*}
    \EE \int_r^T L_0W(s,X_s^{r,x})\ds=I_1+I_2
  \end{equation*}
  where
  \begin{equation*}
    I_1=\EE\int_r^T \partial_i\left(\frac12a^{ij}(s,X_s^{r,x}) \partial_j W(s,X_s^{r,x}) \right)\ds
  \end{equation*}
  and
  \begin{equation*}
    I_2=\EE \int_r^T c^j(s,X_s^{r,x})\partial_jW(s,X_s^{r,x})\ds\,.
  \end{equation*}
  It follows from our conditions on $L_0$ and $W$ that
  \begin{equation*}
    \sup_{t\in[0,T],x\in \RR^d}\frac{|a^{ij}(t,x)\partial_iW(t,x)|}{1+|x|^{\beta_1}} \le \Lambda \kappa_1\,\mbox{ and }\sup_{t\in[0,T],x\in \RR^d}\frac{|c^j(t,x)\partial_jW(t,x)|}{1+|x|^{\beta_1}} \le \Lambda\kappa_1 \,.
  \end{equation*}
  % and
  % \begin{equation*}
  %   \sup_{t\in[0,T],x\in \RR^d}\frac{|c^j(t,x)\partial_jW(t,x)|}{1+|x|^{\beta_1}} \le \Lambda\kappa_1 \,.
  % \end{equation*}
  Applying Lemma \ref{lem.estdf}, we obtain
  \begin{equation*}
    I_1\lesssim \kappa_1 \int_r^T ((s-r)^{-1/2}+1)\ds (1+|x|^{\beta_1})\lesssim \kappa_1 [(T-r)^{1/2}+(T-r)] (1+|x|^{\beta_1})\,.
  \end{equation*}
  For the second term, {\replacecolorred  we use \eqref{est.supx1}}
  \begin{equation*}
    I_2\lesssim \kappa_1 \int_r^T \EE(1+|\varphi_s^{r,x}|^{\beta_1} )\ds\lesssim \kappa_1(T-r)( 1+|x|^{\beta_1})\,.
  \end{equation*}
  These inequalities altogether imply \eqref{est.vw}.

  {\bf $C^1$-estimate:} To show \eqref{est.dvw}, we first apply \eqref{bismut2}
  \begin{equation*}
    \nabla \EE L_0W(s,\varphi_s^{r,x})=(s-r)^{-1} \EE [L_0W(s,\varphi_s^{r,x})H(s,x)  ]
  \end{equation*}
  where
  \begin{equation*}
    H(s,x)=\int_r^s [\sigma^{-1}(\tau,\varphi_{\tau}^{r,x})Y(\tau;r,x)]^T \delta B_{\tau}\,.
  \end{equation*}
  We denote $$A(\tau,x)=\sigma^{-1}(\tau,X_{\tau}^{r,x})Y(\tau;r,x)\,.$$ From \eqref{eqn.DfY}, we see that
  \begin{equation*}
    \partial^2_{ij}W(s,\varphi_s^{r,x})=D^k_{\tau}[\partial_j
    W(s,\varphi_s^{r,x})][A(\tau)Y^{-1}(s)]^{ki}\,,\quad \forall
    \tau\in[r,s]\,.
  \end{equation*}
 Thus
  \begin{align*}
    L_0W(s,\varphi_s^{r,x})&=\frac12a^{ij}(s,X_s^{r,x})\partial^2_{ij}W(s,\varphi_s^{r,x})\\
    &= \frac12D^k_{\tau}[\partial_jW(s,\varphi_s^{r,x})][A(\tau)Y^{-1}(s)]^{ki}a^{ij}(s,X_s^{r,x})\\
    &=\frac12(s-r)^{-1} \int_r^sD^k_{\tau}[\partial_jW(s,\varphi_s^{r,x})][A(\tau)Y^{-1}(s)a(s,X_s^{r,x})]^{kj}d \tau\,.
  \end{align*}
  Hence, applying \eqref{eqn.dual},
  \begin{align*}
    &\partial_l\EE L_0W(s,\varphi_s^{r,x})\\
    &=\frac12(s-r)^{-2}\EE\int_r^s D^k_{\tau}[\partial_jW(s,\varphi_s^{r,x})][A(\tau)Y^{-1}(s)a(s,X_s^{r,x})]^{kj}H^l(s,x)d \tau\\
    &=\frac12(s-r)^{-2}\EE \partial_jW(s,\varphi_s^{r,x})\int_r^s [A(\tau)Y^{-1}(s)a(s,X_s^{r,x})]^{kj}H^l(s,x)\delta B^k_\tau\,.
  \end{align*}
  Furthermore, since the random variable
  \begin{equation*}
    G^{jl} (s;r,x):=\int_r^s [A(\tau)Y^{-1}(s)a(s,X_s^{r,x})]^{kj}H^l(s,x)\delta B^k_\tau
  \end{equation*}
  has mean zero, we can write
  \begin{equation}\label{eqn.lEL}
     \partial_l\EE L_0W(s,\varphi_s^{r,x})
     =\frac12(s-r)^{-2}\EE [\partial_jW(s,\varphi_s^{r,x})-\partial_jW(s,x)]G^{jl} (s;r,x)\,.
  \end{equation}
  We now estimate the moment $G(s;r,x)$. Applying \eqref{id.fu}, we have
  \begin{align*}
    G^{jl} (s;r,x)
    &=\int_r^s [A(\tau)]^{km}[Y^{-1}(s)a(s,X_s^{r,x})]^{mj}H^l(s,x)\delta B^k_\tau\\
    &=[Y^{-1}(s)a(s,X_s^{r,x})]^{mj}H^l(s,x)\int_r^s [A(\tau)]^{km}\delta B^k_\tau\\
    &\quad-\int_r^s D^k_{\tau}([Y^{-1}(s)a(s,X_s^{r,x})]^{mj}H^l(s,x))[A(\tau)]^{km}d \tau\,.
  \end{align*}
    Using properties of Malliavin derivative, we have 
  \begin{multline*}
     D^k_{\tau}([Y^{-1}(s)a(s,X_s^{r,x})]^{mj}H^l(s,x))\\
     = D^k_{\tau}[Y^{-1}(s)a(s,X_s^{r,x})]^{mj}H^l(s,x)+[Y^{-1}(s)a(s,X_s^{r,x})]^{mj}D^k_ \tau H^l(s,x)\,.
  \end{multline*}
  Hence
  \begin{equation}\label{form.G}
  \begin{split}
    G^{jl} (s;r,x)
    &=[Y^{-1}(s)a(s,X_s^{r,x})]^{mj}H^l(s,x)\int_r^s [A(\tau)]^{km}\d B^k_\tau\\
    &\quad-\int_r^s  D^k_{\tau}[Y^{-1}(s)a(s,X_s^{r,x})]^{mj}H^l(s,x)[A(\tau)]^{km}d \tau\\
    &\quad-\int_r^s [Y^{-1}(s)a(s,X_s^{r,x})]^{mj}D^k_ \tau H^l(s,x) [A(\tau)]^{km}d \tau\,.
  \end{split}
  \end{equation}
  Since $a$ belongs to $C^2_b$, estimate \eqref{moment.DY} is valid, the moments of $A(\tau)$
  is also uniformly bounded (because $a$ is strictly elliptic), and all the terms appear in $G^{jl}$
  has finite moments of all orders. In addition, observe that
  \begin{equation*}
    D_ \tau^i H^l(s,x) =1_{\{r\le \tau\}}A(\tau)^{il}+\int_r^s D_{\tau}^iA(u)^{kl}\delta B^k_u\,.
  \end{equation*}
  Thus, the $L^p$-norm of $H(s,x)$ and $D H(s,x)$ will contribute a factor $(r-s)^{1/2}$. Therefore,
   it follows from Burkholder-Davis-Gundy  inequality and H\"older inequality that
  \begin{equation}\label{est.Gx}
    \sup_{x\in\RR^d}\|G^{jl}(s;r,x)\|_p\le c(p,\lambda,\Lambda) [(s-r)+(s-r)^{3/2} ]\,,\forall p\ge1\,.
  \end{equation}
  Using the H\"older continuity of $W$, for every $p\ge1$, we have
  \begin{equation*}
    \|\nabla W(s,\varphi_s^{r,x})-\nabla W(s,x)\|_p\le \kappa_2\|(1+|\varphi_s^{r,x}|^{\beta_2}+|x|^{\beta_2}) |\varphi_s^{r,x}-x|^\alpha \|_{p}\,.
  \end{equation*}
  Taking into account {\replacecolorred   the moment estimate \eqref{exp.supX} } and H\"older inequality, this gives
  \begin{equation}\label{est.wphiw}
    \|\nabla W(s,\varphi_s^{r,x})-\nabla W(s,x)\|_p\le c(\alpha,\beta_2,p,\Lambda) \kappa_2(1+|x|^{\beta_2}) (s-r)^{\alpha/2}\,.
  \end{equation}
  Thus, applying Cauchy-Schwartz inequality in \eqref{eqn.lEL} yields
  \begin{equation*}
    |\partial_l\EE L_0W(s,\varphi_s^{r,x})|\le c(\lambda,\Lambda)(s-r)^{-2}\|\nabla W(s,\varphi_s^{r,x})-\nabla W(s,x)\|_2\|G(s;r,x)\|_2\,.
  \end{equation*}
  Applying the moment estimate for $G$ and \eqref{est.wphiw}, we obtain
  \begin{equation*}
    |\partial_l\EE L_0W(s,\varphi_s^{r,x})|\le c(\lambda,\Lambda) [(s-r)^{\alpha/2-1}+(s-r)^{\alpha/2-1/2}]\kappa_2(1+|x|^{\beta_2}) \,,
  \end{equation*}
  which together with \eqref{eqn.vw} implies \eqref{est.dvw}

  {\bf $C^{1,\alpha'}$-estimate:} This is the only place where we use the fact
  that the second derivatives of $a$ are H\"older continuous. Each  term
  appeared
on  the right hand side \eqref{form.G} is either differentiable or H\"older continuous
  in  the $x$-variable. Thus, we obtain easily the estimate
  \begin{equation}\label{est.Gxy}
    \|G(s;r,x)-G(s;r,y)\|_p\le c(p,\lambda,\Lambda)[(s-r)+(s-r)^{3/2}]|x-y|^{\alpha}\,.
  \end{equation}
  From \eqref{est.wphiw}, we see that
  \begin{align*}
    &\|\nabla W(s,\varphi_s^{r,x})-\nabla W(s,x)-\nabla W(s,\varphi_s^{r,y})+\nabla W(s,y)\|_p\\
    &\le\|\nabla W(s,\varphi_s^{r,x})-\nabla W(s,x)\|_p+\|\nabla W(s,\varphi_s^{r,y})-\nabla W(s,y)\|_p\\
    &\le c(\alpha,p,\Lambda)\kappa_2(1+|x|^{\beta_2}+|y|^{\beta_2})(s-r)^{\alpha/2}\,.
  \end{align*}
  On the other hand, we also have
  \begin{align*}
    &\|\nabla W(s,\varphi_s^{r,x})-\nabla W(s,x)-\nabla W(s,\varphi_s^{r,y})+\nabla W(s,y)\|_p\\
    &\le\|\nabla W(s,\varphi_s^{r,x})-\nabla W(s,\varphi_s^{r,y})\|_p+\|\nabla W(s,x)-\nabla W(s,y)\|_p\\
    &\le\kappa_2(W)\|(1+|\varphi_s^{r,x}|^{\beta_2}+|\varphi_s^{r,y}|^{\beta_2} )|\varphi_s^{r,x}-\varphi_s^{r,y}|^{\alpha}\|_p\\
    &\quad+\kappa_2(W)(1+|x|^{\beta_2}+|y|^{\beta_2})|x-y|^{\alpha}\\
    &\le c(\alpha,p,\Lambda)\kappa_2(1+|x|^{\beta_2}+|y|^{\beta_2})|x-y|^{\alpha}\,,
  \end{align*}
  where the last estimate comes from \eqref{exp.supX} and that fact that the derivative of the map $x\mapsto \varphi_s^{r,x}$ has finite moments uniformly in $x$. Interpolating these two inequalities we obtain
  \begin{multline}\label{est.wtheta}
    \|\nabla W(s,\varphi_s^{r,x})-\nabla W(s,x)-\nabla W(s,\varphi_s^{r,y})+\nabla W(s,y)\|_p\\
    \le c(\alpha,p,\Lambda)\kappa_2(1+|x|^{\beta_2}+|y|^{\beta_2})|x-y|^{\vartheta\alpha}(s-r)^{(1- \vartheta)\alpha/2}
  \end{multline}
  for any $\vartheta\in[0,1]$. Thus, from \eqref{eqn.lEL}, applying Cauchy-Schwartz inequality we see that
  \begin{align*}
    &|\nabla \EE L_0W(s,\varphi_s^{r,x})-\nabla\EE L_0W(s,\varphi_s^{r,y})|\\
    &\le (s-r)^{-2}\|\nabla W(s,\varphi_s^{r,x})-\nabla W(s,x)-\nabla W(s,\varphi_s^{r,y})+\nabla W(s,y)\|_2\|G(s;r,x)\|_2\\
    &\quad+(s-r)^{-2}\|\nabla W(s,\varphi_s^{r,y})-\nabla W(s,y)\|_2\|G(s;r,x)-G(s;r,y)\|_2\,.
  \end{align*}
  Using \eqref{est.wtheta}, \eqref{est.wphiw}, \eqref{est.Gx} and \eqref{est.Gxy}, we obtain
  \begin{align*}
    &|\nabla \EE L_0W(s,\varphi_s^{r,x})-\nabla\EE L_0W(s,\varphi_s^{r,y})|\\
    &\le c(\alpha,\lambda,\Lambda)\kappa_2(1+|x|^{\beta_2}+|y|^{\beta_2} )|x-y|^{\vartheta \alpha}[(s-r)^{(1- \vartheta)\alpha/2-1}+(s-r)^{(1- \vartheta)\alpha/2-1/2}]\\
    &\quad+c(\alpha,\lambda,\Lambda)\kappa_2(1+|y|^{\beta_2} )|x-y|^\alpha [(s-r)^{\alpha/2-1}+(s-r)^{\alpha/2-1/2}]\,.
  \end{align*}
  Therefore, choosing $\vartheta<1$, this estimate together with \eqref{eqn.vw} implies \eqref{est.dvwa}.
\end{proof}
% section schauder_estimates (end)

\end{appendix}

\noindent{\bf Acknowledgment}.  The authors sincerely thank the referee for the many constructive and detailed comments which were of great help in revising the manuscript. We also thank Samy Tindel and Massimiliano Gubinelli for their interest in our paper.

 \bibliographystyle{abbrv}
\bibliography{bib}
\end{document}